\documentclass{article}
\usepackage{amssymb,latexsym,
amsmath,amsthm}
\usepackage{ifsym}

\newtheorem{theorem}{Theorem}[section]
\newtheorem{lemma}[theorem]{Lemma}
\newtheorem{corollary}[theorem]{Corollary}
\newtheorem{proposition}[theorem]{Proposition}

\theoremstyle{definition}
\newtheorem{definition}[theorem]{Definition}

\newtheorem{remark}[theorem]{Remark}

\newtheorem{notation}[theorem]{Notation}

\newcommand{\End}{{ \rm End }}

\newcommand{\Coker}{{ \rm Coker }}

\newcommand{\Hom}{{ \rm Hom }}
\newcommand{\Ker}{{ \rm Ker }\,}

\newcommand{\Mod}{{ \rm Mod }}

\newcommand{\rad}{{ \rm rad }}

\renewcommand{\frak}{\mathfrak}

\newcommand{\proj}{{\rm proj }}
\newcommand{\g}{\hbox{-}}
\newcommand{\uddots}{\mathinner{\mkern1mu\raise1pt\vbox{\kern7pt\hbox{.}}
\mkern2mu\raise4pt\hbox{.}\mkern2mu\raise7pt\hbox{.}\mkern1mu}}
\newcommand{\Endol}[1]{\hbox{\rm endol}\,(#1)}

\newcommand{\hueca}[1]{\mathbb{#1}}
\newcommand{\lddots}{
\mathinner{
\mkern1mu\raise1pt}\vbox{\kern7pt\hbox{.}}
\mkern2mu\raise3pt\hbox{.}
\mkern2mu\raise7pt\hbox{.}\mkern1mu}

\renewcommand{\mod}{{\rm mod }}
\renewcommand{\Im}{{\rm Im}\,}

\newcommand{\rightdashmap}[1]{\smash{\mathop{\hbox to 
20pt{-\,-\,-\,\rightarrowfill}}\limits^{#1}}}

\newcommand{\rightmap}[1]{\smash{\mathop{\hbox to 
20pt{\rightarrowfill}}\limits^{#1}}}
\newcommand{\leftmap}[1]{\smash{\mathop{\hbox to 
20pt{\leftarrowfill}}\limits^{#1}}}

\newcommand{\dobleflechava}[2]{\ \smash{\mathop{
   \raise 3pt 
   \hbox to 20pt{\rightarrowfill}\hskip-20pt         \lower 3pt
   \hbox to 20pt{\rightarrowfill}}\limits^{#1}_{#2}}\ }

\newcommand{\rmapdown}[1]{\Bigg\downarrow\rlap{$\vcenter{\hbox{$\scriptstyle#1$}
}$}}

\newcommand{\shortlmapdown}[1]
{\llap{$\vcenter{\hbox{$\scriptstyle#1$}}$}\big\downarrow}
\newcommand{\shortrmapdown}[1]
{\big\downarrow\rlap{$\vcenter{\hbox{$\scriptstyle#1$}}$}}

\newcommand{\shortlmapup}[1]
{\llap{$\vcenter{\hbox{$\scriptstyle#1$}}$}\big\uparrow}
\newcommand{\shortrmapup}[1]
{\big\uparrow\rlap{$\vcenter{\hbox{$\scriptstyle#1$}}$}}

\newcommand{\longrightmap}[1]{\smash{\mathop{\hbox to 
4cm{\rightarrowfill}}\limits^{#1}}}
\newcommand{\longleftmap}[1]{\smash{\mathop{\hbox to 
4cm{\leftarrowfill}}\limits^{#1}}}
\newcommand{\medrightmap}[1]{\smash{\mathop{\hbox to 
2cm{\rightarrowfill}}\limits^{#1}}}
\newcommand{\medleftmap}[1]{\smash{\mathop{\hbox to 
2cm{\leftarrowfill}}\limits^{#1}}}

\newcommand{\dobleflechavieneva}[2]{\ \smash{\mathop{
   \raise 3pt \hbox to 40pt{\leftarrowfill}\hskip-40pt \lower 3pt
   \hbox to 40pt{\rightarrowfill}}\limits^{#1}_{#2}}\ }
\newcommand{\dobleflechavienevabis}[2]{\ \smash{\mathop{
   \raise 3pt \hbox to 20pt{\leftarrowfill}\hskip-20pt \lower 3pt
   \hbox to 20pt{\rightarrowfill}}\limits^{#1}_{#2}}\ }

\newcommand{\idmapdown}[1]
{\hskip-8pt\mathop{\hskip-5pt\raise6pt
\hbox{$\scriptstyle#1$}\hskip-5pt\swarrow}}

\newcommand{\ddmapdown}[1]
{\hskip-5pt\mathop{\searrow\hskip-6pt\raise5pt\hbox{$\scriptstyle#1$}}}

\newcommand{\idmapup}[1]
{\hskip-5pt\mathop{\nwarrow\hskip-6pt \raise5pt\hbox{$\scriptstyle#1$}}}

\newcommand{\ddmapup}[1]
{\hskip-8pt\mathop{\hskip-5pt\raise6pt\hbox{$\scriptstyle#1$}\hskip-5pt\nearrow}
}

\newcommand{\flechypunt}[2]{\ \smash{\mathop{
   \raise 3pt \hbox to 40pt{\rightarrow}\hskip-40pt \lower 3pt
   \hbox to 40pt{\dashrightarrow}}\limits^{#1}_{#2}}\ }

\newcommand{\longequal}{\ \smash{\mathop{
   \raise 5pt \hbox to 35pt{\hrulefill}\hskip-35pt \lower 0pt
   \hbox to 35pt{\hrulefill}}}\ }

\newcommand{\raya}[1]{\ \smash{\mathop{\raise 2pt \hbox to 
10pt{\hrulefill}}\limits^{#1}}\ }

\title{\bf {\Large On generic bricks over tame algebras}}

\author{R. Bautista, E. P\'erez, L. Salmer\'on}

\begin{document}
 \date{}
 
 \maketitle
  
 \renewcommand{\thefootnote}{}



\hbox{\emph{Dedicated to Professor C.M. Ringel in the occasion of his 80th birthday}}

  \begin{abstract}
\noindent We prove that if $\Lambda$ is a tame  finite-dimensional  algebra over an algebraically closed field and $G$ is a generic $\Lambda$-module, then $G$ is a generic brick if and only if it determines a one-parameter family of bricks with the same dimension. In particular, we obtain that $\Lambda$ admits a generic brick if and only if $\Lambda$ is brick-continuous.     
\end{abstract}

\section{Introduction}

Throughout this work,  unless we specify  otherwise, $k$ denotes a fixed algebraically closed field and $\Lambda$ is a tame finite-dimensional   $k$-algebra. We will denote by $\Lambda\g\Mod$ the category of left $\Lambda$-modules and by $\Lambda\g\mod$ its full subcategory of finite-dimensional modules. 

Recall that given $G\in \Lambda\g\Mod$, the \emph{endolength} $\Endol{G}$ of $G$ is the length of $G$ when considered as an $\End_{\Lambda}(G)^{op}$-module. The module $G$ is called \emph{generic} if it is indecomposable, with infinite length as $\Lambda$-module, but 
with finite endolength. Generic modules were introduced by Crawley-Boevey in \cite{CB2}, where he proved that a tame algebra $\Lambda$ is characterized by the fact that, for each endolength $d\in\hueca{N}$, there are only finitely many generic $\Lambda$-modules with endolength $d$. This  characterization of tame-representation type is important, because it provided a useful definition of tame algebras for more general rings, see \cite{CB3}. Crawley-Boevey proved also that 
whenever $G$ is a generic $\Lambda$-module, we have  that $\End_\Lambda (G)/\rad \End_\Lambda (G) \cong k(x)$, the field of rational functions in the variable $x$,  and $\End_\Lambda (G)$ is split over its radical. 
The proofs of these facts use bocses techniques, which go back to the work of the Kiev School in representation theory, see also  \cite{R-K}, \cite{D}, and \cite{CB1}. Moreover, Crawley-Boevey
proved in \cite{CB2} that, in the tame case, any generic $\Lambda$-module $G$ gives rise to  infinite parameter families of non-isomorphic indecomposable $\Lambda$-modules  as follows. 

Recall that a rational algebra $\Gamma$ is a 
finitely generated localization of $k[x]$, so $\Gamma=k[x]_h$, for some polynomial $h\in k[x]$. If $G$ is a generic $\Lambda$-module and $\Gamma$ is a rational algebra, then a \emph{realization of $G$ over $\Gamma$} is a finitely generated $\Lambda\g \Gamma$-bimodule $M$ such that if $K$ is the quotient field of $\Gamma$, then $M\otimes_\Gamma K\cong G$ and $\dim_K(M\otimes_\Gamma K)=\Endol{G}$.  This notion was proposed by 
Crawley-Boevey in \cite{CB2}, where he proved the following.

\begin{theorem}[Crawley-Boevey]\label{T: CB-realizations} If $\Lambda$ is a  tame 
finite-dimensional algebra, then for each generic $\Lambda$-module $G$ there is  a realization $M$ over some rational algebra $\Gamma$  with the following properties: 
\begin{enumerate}
\item  As a right $\Gamma$-module, $M$ is 
free of rank equal to the endolength of $G$. 
\item The functor $M\otimes_\Gamma-:\Gamma\g\Mod\rightmap{}\Lambda\g\Mod$ 
 preserves isomorphism classes,  indecomposability, and Auslander-Reiten sequences. 
\end{enumerate} 
\end{theorem}
 
In this work, given a realization $M$ of $G$ over some rational algebra $\Gamma=k[x]_h$, we are mainly interested in the one-parameter family of indecomposable $\Lambda$-modules $M(\lambda):=M\otimes_\Gamma\Gamma/(x-\lambda)$, for $\lambda\in D(h):=\{\lambda\in k\mid h(\lambda)\not=0\}$.

We recall also that a $\Lambda$-module $N$ is called a \emph{brick} if $\End_\Lambda(N)$ is a division algebra. So, a finite-dimensional $\Lambda$-module $N$  is a \emph{brick}  if $\End_{\Lambda}(N) \cong  k$. For instance, if $\Lambda$ is tame hereditary, a finite-dimensional brick corresponds to a preprojective, or a  preinjective module, or to a module lying in the mouth of a regular component in the Auslander-Reiten quiver of $\Lambda$, see \cite{ars} and \cite{Ri}. 

Following Mousavand and Paquette, we say that  the algebra $\Lambda$ is \emph{brick-infinite} if it admits an infinite family of non-isomorphic finite-dimensional bricks, see \cite{MP2}, \cite{DIJ}, and \cite{Sentieri}. It is called \emph{brick-continuous} if it admits an infinite family of non-isomorphic bricks with the same dimension. A generic $\Lambda$-module $G$ is called a \emph{generic brick} if $\End_{\Lambda}(G)\cong k(x)$. The study of bricks in representation theory of finite-dimensional algebras has a long history, see \cite{Ri-spec}, \cite{B-D}, \cite{MP2}, and \cite{MP3}.   Recently, it has received  a lot of attention because it plays  some relevant role in $\tau$-tilting theory of algebras and  brick analogues to the Brauer-Thrall conjectures have been formulated, see \cite{MP1} and \cite{MP3}. 

Biserial algebras are an important class of tame algebras, see \cite{CB4}.  
In \cite{MP2}, K. Mousavand and C. Paquette  proved  that  a biserial algebra $\Lambda$ is  brick-infinite if and only if it admits a generic brick. 
They conjectured, see \cite[(4.1)]{MP2}, for a  general tame finite-dimensional  algebra  $\Lambda$, 
that $\Lambda$ is brick-continuous if and only if  $\Lambda$ admits a generic brick.

The aim of this work is to prove the following.
\medskip

\begin{theorem}\label{T: gen-bricks vs bricks} Let $\Lambda$ be a tame finite-dimensional  algebra,  
assume that $G$ be a generic $\Lambda$-module, and let $M$ be  a realization of $G$ over a rational algebra $\Gamma$, then the following are equivalent:
\begin{enumerate}
\item $G$ is a generic brick; 
\item  For infinitely many $\lambda\in k$, we have 
that $M(\lambda)$  is a brick.  
\end{enumerate}
\end{theorem}

As a consequence, we obtain the following. 

\begin{corollary}\label{C: tame conj M-P} Let $\Lambda$ be a tame finite-dimensional algebra. Then, $\Lambda$ is brick-continuous if and only if $\Lambda$  admits a generic brick. 
\end{corollary}

While we are on the subject, recall that Crawley-Boevey has proved, in \cite[(9.5)]{CB3}, that, for any finite-dimensional algebra $A$:  for some $d\in \hueca{N}$  there are infinitely many non-isomorphic finite-dimensional indecomposable $A$-modules with endolength $d$ if and only if $A$ has a generic module. Then, the following natural question arises for a finite-dimensional algebra $A$. Is it true that,  there is some   $d\in \hueca{N}$ and infinitely many non-isomorphic finite-dimensional $A$-bricks  with endolength $d$ if and only if $A$ has a generic brick?

Any two realizations $M$ and $N$ of the same generic $\Lambda$-module $G$, determine one-parameter families $\{M(\lambda)\}_\lambda$ and  $\{N(\lambda)\}_\lambda$, respectively, such that $M(\lambda)\cong N(\lambda)$, for almost all $\lambda\in k$, see \cite[(5.2)]{CB2}. Thus, in order to prove Theorem (\ref{T: gen-bricks vs bricks}), 
it will be enough to show one realization $M$ over some rational algebra $\Gamma$ that  satisfies the  equivalence stated in the preceding theorem.  

In order to prove this,  
we may assume, and we will  from now on, that $\Lambda$ is basic and indecomposable. This is the case because,   if $\Lambda_b$ is the basic algebra of $\Lambda$ and $P\otimes_{\Lambda_b}-:\Lambda_b\g\Mod\rightmap{}\Lambda\g\Mod$ is a Morita equivalence, then, any generic $\Lambda$-module $G$ admits a  realization of the form $M\cong P\otimes_{\Lambda_b} M_b$, where $M_b$ is realization of a generic $\Lambda_b$-module $G_b$ such that    $P\otimes_{\Lambda_b} G_b\cong G$. This follows from the argument in the proof of (\ref{T: gens for tame algs}).    

Furthermore, we assume 
 that $\Lambda=kQ/I$, where $Q$ is a finite connected quiver and $I$ is an admissible ideal of $kQ$, see \cite[II]{ASS}.  

The idea of the proof is to use the fact that any generic $\Lambda$-module $G$ admits a realization  $M$ such that $G$ and $M$ are obtained from a minimal ditalgebra ${\cal B}$ through an appropriate composition of functors 
$${\cal B}\g\Mod\rightmap{F}{\cal P}^1(\Lambda)\rightmap{ \ \Coker \ }\Lambda\g\Mod,$$
where $F$ is full and faithful, and $\Coker$ is the usual cokernel functor from the category ${\cal P}^1(\Lambda)$ of morphisms between projectives with superfluous image.  If $\hat{G}$ and $\{X_\lambda\}_\lambda$ are objects in ${\cal P}^1(\Lambda)$ such that $\Coker (\hat{G})\cong G$ and $\Coker (X_\lambda)\cong M(\lambda)$, what we have to show is that every endomorphism  in the radical of $\End_{{\cal P}(\Lambda)}(\hat{G})$ factors through objects in ${\cal P}(\Lambda)$ of the form $P\rightmap{id}P$ and $P\rightmap{}0$ if and only if, for infinitely many $\lambda\in k$, 
every endomorphism of the radical of $\End_{{\cal P}(\Lambda)}(X_\lambda)$ factors through objects in ${\cal P}(\Lambda)$ of the form $P\rightmap{id}P$ and $P\rightmap{}0$. 

We will assume that the reader has some familiarity with bocses techniques in terms of ditalgebras (an acronym for ``differential tensor algebras''), as presented in \cite{BSZ}. Compositions of reduction functors between module categories of ditalgebras are an important tool in our arguments. In very general terms, we work with  layered ditalgebras ${\cal A}=(T,\delta)$ such that  $T$ is a graded tensor algebra $T=T_R(W)$, with $R$ a commutative algebra and $W=W_0\oplus W_1$  an $R$-$R$-bimodule, such that, in the gradding of $T$, the elements of   $R$ and $W_0$ have degree 0 and those of  $W_1$ have degree 1. The category of ${\cal A}$-modules ${\cal A}\g\Mod$ has as objects the left $A$-modules, where $A=T_R(W_0)$, and the morphisms $f:X\rightmap{}Y$ in  ${\cal A}\g\Mod$ are pairs of morphisms $f=(f^0,f^1)$, with $f^0\in \Hom_R(X,Y)$ and $ f^1\in \Hom_{A\g A}(V,\Hom_k(X,Y))$, where $V=AW_1A$ is the homogeneous component of $T$ of degree 1. 
These pairs must satisfy the condition $f^0(ax)=af^0(x)+f^1(\delta(a))(x)$, for $x\in X$ and $a\in A$, that is 
$f^0$ is almost a morphism of $A$-modules, we have a correction which depends on $\delta(a)$ and $f^1$.  The composition is performed according to a special rule which uses again the differential $\delta$ of ${\cal A}$, see \cite[\S2]{BSZ}. There is a canonical embedding functor $L_{\cal A}: A\g\Mod\rightmap{}{\cal A}\g\Mod$ which is the identity on objects and maps each morphism of $A$-modules $f:X\rightmap{}Y$ onto the morphism $(f,0):X\rightmap{}Y$ of ${\cal A}$-modules. 

  An important example is the Drozd's ditalgebra ${\cal D}$ associated to the algebra $\Lambda$, which is very useful because  there is an equivalence of categories $\Xi_\Lambda:{\cal D}\g\Mod\rightmap{}{\cal P}^1(\Lambda)$. This   was used by Drozd to prove his famous tame-wild dichotomy theorem, see \cite{D}. His main idea to prove this theorem is to construct a minimal ditalgebra ${\cal B}$ and a composition of reduction functors   
  $H:{\cal B}\g\Mod\rightmap{}{\cal D}\g\Mod$, 
which, composed with $\Xi_\Lambda$ and $\Coker$, permitted some understanding of $\Lambda\g\Mod$. This idea was enriched by Crawley-Boevey in \cite{CB1} and \cite{CB2} to show the characterization of tame algebras and his theorem on realizations of generic modules  mentioned before.  The advantage of minimal ditalgebras ${\cal B}$ is that their module category is relatively simple, for instance the radical morphisms between important ${\cal B}$-modules admit  manageable   descriptions which permit to make explicit computations.

\section{Kernel-generating systems for $\Coker^1$}
  In this section we introduce a family of special morphisms, which will play a crucial role in our arguments.  This construction has some general features, so we start by considering  an artin algebra $A$  over a commutative artin ring. We will later specialize to the case of a finite-dimensional $k$-algebra  $\Lambda=kQ/I$ over an arbitrary field $k$.   We denote by $J$ the radical of the artin algebra $A$.  
   
   \begin{remark} Recall that the category ${\cal P}(A)$ consists of the triples $X=(P_1,P_2,\phi)$ where $P_1,P_2$ are projective $A$-modules and $\phi:P_1\rightmap{} P_2$ is a morphism in $A\g\Mod$. A morphism $u:X\rightmap{}X'$ in ${\cal P}(\Lambda)$ is a pair $u=(u_1,u_2)$ of morphisms of $A$-modules such that the following diagram commutes
$$\begin{matrix}P_1&\rightmap{\phi}&P_2\\
\shortlmapdown{u_1}&&\shortrmapdown{u_2}\\
P'_1&\rightmap{\phi'}&P'_2.\\
\end{matrix}$$  
We denote by ${\cal P}^1(A)$ the full subcategory of ${\cal P}(A)$ whose objects are the triples $(P_1,P_2,\phi)$ such that $\phi(P_1)\subseteq JP_2$. 
For each projective $P\in A\g\Mod$, we consider the objects $S(P)=(P,0,0)\in {\cal P}^1(A)$ and $U(P)=(P,P,id_P)\in {\cal P}(A)$. The cokernel functor $\Coker:{\cal P}(A)\rightmap{}A\g\Mod$ is full and dense, and maps $S(P)$ and $U(P)$ to $0$. 
\end{remark}

\begin{lemma}\label{L: facts en P(Lambda)}
If $u : (P_1,P_2,\phi)\rightmap{}(P'_1,P'_2,\phi')$ is any morphism in ${\cal P}(A)$ with $\Coker(u) = 0$, then $u = v + w$, for some morphisms $v, w : (P_1,P_2,\phi)\rightmap{}(P'_1,P'_2,\phi')$, where $v$ factors through  $S(P_1)$ and $w$ factors through  $U(P_2)$ in ${\cal P}(A)$. 
\end{lemma}

\begin{proof} Set $M:=\Coker(\phi)$ and $M':=\Coker(\phi')$. Then, we have the  commutative diagram
$$\begin{matrix}
P_1&\rightmap{\phi}&P_2&\rightmap{\eta}&M\\
\shortlmapdown{u_1}&&\shortlmapdown{u_2}&&\shortlmapdown{0}\\
P'_1&\rightmap{\phi'}&P'_2&\rightmap{\eta'}&M'\\
\end{matrix}$$
where $\eta$ and $\eta'$ are the canonical projections. Since $\eta'u_2=0$, there is a morphism $s:P_2\rightmap{}P'_1$ such that $\phi's=u_2$. Then, we have the commutative diagrams 
$$\begin{matrix} P_1&\rightmap{\phi}&P_2\\
\shortlmapdown{id}&&\shortrmapdown{0}\\
P_1&\rightmap{0}&0\\
\shortlmapdown{u_1-s\phi}&&\shortrmapdown{0}\\
P'_1&\rightmap{\phi'}&P'_2\\
\end{matrix} \hbox{ \  \ \ \ and \ \ \  \ }
\begin{matrix} P_1&\rightmap{\phi}&P_2\\
\shortlmapdown{\phi}&&\shortrmapdown{id}\\
P_2&\rightmap{id}&P_2\\
\shortlmapdown{s}&&\shortrmapdown{u_2}\\
P'_1&\rightmap{\phi'}&P'_2\\
\end{matrix}$$
Hence, $v:=(u_1-s\phi,0)$ and $w=(s\phi,u_2)$ are the desired morphisms. 
\end{proof}

\begin{lemma}\label{L: morfismos universales de P} If $P$ is an indecomposable projective $A$-module,  we have:
\begin{enumerate}
\item There is a morphism of $A$-modules $\phi_P : L_P \rightmap{} P$ with  $L_P$ projective, which is not a retraction (that is such that  $\Im \phi_P\subseteq JP$),   such that for any projective $P'$, any morphism $h:P'\rightmap{}P$ that is not a retraction factors through $\phi_P$. 
Moreover, whenever $u : L_P \rightmap{} L_P$ is a morphism such that 
$\phi_Pu=\phi_P$, we get that $u$ is an isomorphism. 

\item There is a morphism of $A$-modules 
$\psi_P: P \rightmap{} R_P$ with $R_P$ projective, which is not a section (that is such that  $\Im \psi_P\subseteq JR_P$),  such that, for any projective $P'$,  any morphism $g:P\rightmap{}P'$ that is not a section factors through $\psi_P$. Moreover, whenever  $v:R_P\rightmap{}R_P$ is a morphism such that $v\psi_P=\psi_P$, we get that $v$ is an isomorphism.
\end{enumerate}
\end{lemma}  

\begin{proof} (1): Consider a projective cover $p:L_P\rightmap{}JP$ and the inclusion $s:JP\rightmap{}P$. Then, it is easy to show that  the composition $\phi_P:=sp:L_P\rightmap{}P$ satisfies the  properties listed in the first item.  

\noindent(2): Consider the  duality $(-)^*=\Hom_{A^{op}}(-,A^{op}):A^{op}\g\proj\rightmap{}A\g\proj$. Given an indecomposable projective $A^{op}$-module $P$, we have the morphism of $A^{op}$-modules $\phi_P:L_P\rightmap{}P$, and the morphism of $A$-modules $\psi_{P^*}:=(\phi_P)^*:P^*\rightmap{}R_{P^*}$, where $R_{P^*}:=(L_P)^*$ is projective and $\Im \psi_{P^*}\subseteq JR_{P^*}$. It follows that $\psi_{P^*}$ satisfies the properties stated in the second item for $P^*$. So, the second item holds.  
\end{proof}

\begin{remark}\label{N: notation alpa, beta y gamma} Given any indecomposable projective $A$-module $P$, the morphisms $\phi_P$ and $\psi_P$ are uniquely determined up to isomorphism,  respectively, by the properties listed in items (1) and (2) of the last lemma. 

Once we have chosen morphisms $\phi_P$ and $\psi_P$,   we have  triples in ${\cal P}^1(A)$: 
$$L(P):=(L_P,P,\phi_P) \hbox{ \ and \ } R(P):=(P,R_P,\psi_P).$$  
Moreover, we have the  following morphisms in ${\cal P}(A):$
$$ \alpha_P := (\phi_P,id) : L(P)\rightmap{}U(P)\hbox{ \ and \ } \beta_P := (id,\psi_P) : U(P)\rightmap{}R(P);$$
and, in ${\cal P}^1(A)$, the morphism 
$\gamma_P := 
\beta_P\alpha_P= (\phi_P, \psi_P) : L(P)\rightmap{}R(P)$.
\end{remark}

\begin{lemma}\label{L: facts en P1(Lambda)}
Given $X\in {\cal P}^1(A)$ and an indecomposable projective $A$-module $P$, we have:  
\begin{enumerate}
\item Any morphism $u:X\rightmap{}U(P)$ in ${\cal P}(A)$ factors through $\alpha_P$.
\item Any morphism $v:U(P)\rightmap{}X$ in ${\cal P}(A)$ factors through $\beta_P$. 
\end{enumerate}
\end{lemma}

\begin{proof} Assume $X=(P_1,P_2,\phi)$,  $u=(u_1,u_2)$, and $v=(v_1,v_2)$.  

In order to prove (1), notice that, since $u:X\rightmap{}U(P)$ is a morphism and $X\in {\cal P}^1(A)$, we have $u_1=u_2\phi$ and $\phi(P_1)\subseteq JP_2$. Thus, $u_1(P_1)\subseteq JP$ and $u_1$ is not a retraction. So,  there is a morphism $h:P_1\rightmap{}L_P$ such that $u_1=\phi_P h$,  and we obtain the commutative diagram 
$$\begin{matrix} 
P_1&\rightmap{\phi}&P_2\\
\shortlmapdown{h}&&\shortrmapdown{u_2}\\
L_P&\rightmap{\phi_P}&P\\
\shortlmapdown{\phi_P}&&\shortrmapdown{id}\\
P&\rightmap{id}&P\\
\end{matrix}$$ 
Then, we get $u=\alpha_P\zeta$, with $\zeta:=(h,u_2)$. 

In order to prove (2), notice that, since $v:U(P)\rightmap{}X$ is a morphism and $X\in {\cal P}^1(A)$, we get $v_2=\phi v_1$ and $v_2(P)\subseteq JP_2$. Thus,   $v_2$ is not a section. So, there is a morphism  $g:R_P\rightmap{}P_2$ such that $g\psi_P=v_2$ and we obtain the commutative diagram
$$\begin{matrix}
 P&\rightmap{id}&P\\
\shortlmapdown{id}&&\shortrmapdown{\psi_P}\\
P&\rightmap{\psi_P}&R_P\\
\shortlmapdown{v_1}&&\shortrmapdown{g}\\
P_1&\rightmap{\phi}&P_2\\
\end{matrix}$$ 
So, we get $v=\xi \beta_P$, with $\xi:=(v_1,g)$. 
\end{proof}

\begin{proposition}\label{P: caract de u con cokeru=0}
Let $u:X\rightmap{}X'$ be a morphism in ${\cal P}^1(A)$, where $X$ and $X'$ have underlying finite-length modules. Then, $\Coker(u)=0$ if and only if $u$ is a finite sum of morphisms which factor either by some object $S(P)$ or by some morphism $\gamma_P:L(P)\rightmap{}R(P)$, with $P$ indecomposable. 
\end{proposition}

\begin{proof} Suppose that  $u=\sum_{i=1}^rg_i\gamma_{P_i}h_i+\sum_{j=1}^sg'_jh'_j$,  for some morphisms $h_i:X\rightmap{}L(P_i)$, $g_i:R(P_i)\rightmap{}X'$, $h'_j:X\rightmap{}S(P'_j)$, and $g'_j: S(P'_j)\rightmap{}X'$, with $P_i$ and  $P'_j$ indecomposable projectives. Then, applying the Coker functor to this sum,   we get
$$\Coker (u)=\sum_i\Coker(g_i)\Coker(\gamma_{P_i})\Coker (h_i)+\sum_j\Coker(g'_j)\Coker(h'_j)=0,$$
because each $\gamma_{P_i}$ factors through $U(P_i)$.  Conversely, assume that $\Coker(u)=0$. From (\ref{L: facts en P(Lambda)}), 
we know that $u=v+w$, where $v$ factors through some $U(P)$ and $w$ factors through some $S(P')$, where $P$ and $P'$ are finite-length projectives. Then, $U(P)=\bigoplus_{i=1}^r U(P_i)$ and $S(P')=\bigoplus_{j=1}^sS(P'_j)$, where $P_i$ and $P'_j$ are indecomposable projective $A$-modules.  Hence, $v=gh$, for some morphisms  
$h:X\rightmap{}\bigoplus_{i=1}^rU(P_i)$ and  $g: \bigoplus_{i=1}^rU(P_i)\rightmap{}X',$
and $w=g'h'$, for some morphisms 
$h':X\rightmap{}\bigoplus_{j=1}^sS(P'_j)$  and  $g': \bigoplus_{j=1}^sS(P'_j)\rightmap{}X'$. 
Let $\sigma_i$ denote the $i^{th}$-injection of $U(P_i)$ in $U(P)$ and $\pi_i$ the $i^{th}$-projection of $U(P)$ on  $U(P_i)$. Denote by $\sigma'_j$ and $\pi'_j$ the injections and projections associated to the direct sum $S(P')$. Then, we have 
$$v+w=\sum_{i=1}^rg\sigma_i\pi_i h+\sum_{j=1}^sg'\sigma'_j\pi'_jh'.$$  
From (\ref{L: facts en P1(Lambda)}), we get $\pi_i h=\alpha_{P_i}h_i$, for some $h_i:X\rightmap{}L(P_i)$, and $g\sigma_i= g_i\beta_{P_i}$, for some $g_i:R(P_i)\rightmap{}X'$. The proposition follows from this.  
\end{proof}

\begin{lemma}\label{L: L(Lambdaei), R(Lambdaej) inesc no isomorfos entre si} The objects  
$L(P)$ and $R(P)$ of ${\cal P}^1(A)$  are indecomposable, for any indecomposable projective $P$. 
\end{lemma}

\begin{proof} Notice first that $L(P)$ is a minimal projective presentation of the simple $A$-module $S=P/JP$, so it is indecomposable. The duality $(-)^*=\Hom_{A^{op}}(-,A^{op}): A^{op}\g \proj\rightarrow{}A\g\proj$ transforms  the indecomposable  $L(P)$ onto $R(P^*)$, so this last one is also indecomposable. Our lemma follows from this.  
\end{proof}

In the following definition we specialize to the case where $A$ is a finite-dimensional quotient of a quiver  algebra over an arbitrary field $k$.  

\begin{definition}\label{D: L(Lambdaex) R(Lambdaex)} Assume that $\Lambda=kQ/I$,   where $Q$ is a quiver, $k$ is any field, and $I$ is an admissible ideal of $kQ$. Consider the complete set $\{e_1,\ldots,e_n\}$ of primitive ortogonal idempotents of $\Lambda$, formed by the classes $e_t:=\tilde{\tau}_t$ of the trivial paths $\tau_t$ based at the vertices $t$ of $Q$. 
 
For any  vertex $t$, we consider  the object  $L(\Lambda e_t):=(L_{\Lambda e_t},\Lambda e_t, \phi_t)$ in ${\cal P}^1(\Lambda)$ given by the morphism 
$$\begin{matrix}L_{\Lambda e_t}:=\left[\bigoplus_{\alpha:t\rightarrow s }\Lambda e_s\right]&\rightmap{ \ \ \phi_t \ \ }&\Lambda e_t\end{matrix}$$
where $\phi_t=(\rho_\alpha)_{\alpha:t\rightarrow s}$ and each  $\rho_\alpha:\Lambda e_s\rightmap{}\Lambda e_t$
is the right multiplication by the arrow $\alpha$. For any vertex $t$ of $Q$, we have  the object  $R(\Lambda e_t):=(\Lambda e_t,R_{\Lambda e_t},\psi_t)$ in ${\cal P}^1(\Lambda)$ given by the morphism
$$\begin{matrix}\Lambda e_t&\rightmap{ \  \ \psi_t \ \ }&\left[\bigoplus_{\beta:s\rightarrow t} \Lambda e_s\right]=:R_{\Lambda e_t}\end{matrix}$$
where $\psi_t=(\rho_\beta)^t_{\beta:s\rightarrow t}$. 
For $t\in [1,n]$, we have the following morphisms in ${\cal P}(\Lambda)$:  
$$\alpha_t:=(\phi_t,id):L(\Lambda e_t)\rightmap{}U(\Lambda e_t) \hbox{ and }
\beta_t:=(id,\psi_t):U(\Lambda e_t)\rightmap{}R(\Lambda e_t);$$
and, in ${\cal P}^1(\Lambda)$, the morphism 
$\gamma_t:=\beta_t\alpha_t=(\phi_t,\psi_t):L(\Lambda e_t)\rightmap{}R(\Lambda e_t)$. 
\end{definition}

\begin{definition}\label{D: gen system for Coker Cok^1}  Given any artin algebra  $A$, we  say that the families  of morphisms 
$\{id_{\hueca{S}_t}\}_{t\in[1,n]} \hbox{ and } 
\{\gamma_t:\hueca{L}_t\rightmap{}\hueca{R}_t\}_{t\in [1,n]}$
\emph{are a kernel-generating system for $\Coker^1:{\cal P}^1(A)\rightmap{}A\g\Mod$} if we have:
 \begin{enumerate}
\item  $\hueca{L}_t$, $\hueca{R}_t$, $\hueca{S}_t$ are indecomposables objects of ${\cal P}^1(A)$ and $\gamma_t$ is a non-zero morphism in ${\cal P}^1(A)$, for all $t\in [1,n]$. 
\item $\hueca{L}_t\not\cong \hueca{L}_s$, $\hueca{R}_t\not\cong \hueca{R}_s$, and $\hueca{S}_t\not\cong \hueca{S}_s$, whenever $t\not=s$.
\item Given any morphism $u:X\rightmap{}X'$ in ${\cal P}^1(A)$, where $X$ and $X'$ have underlying finite-length  modules, we have: $\Coker(u)=0$ if and only if $u$ is a finite sum of morphisms which factor either by some object $\hueca{S}_i$ or by some morphism $\gamma_t$.  
\end{enumerate}
\end{definition}

\begin{remark}\label{R: isomorfos a generadores son generadores} 
Let $A$ be an artin algebra and 
suppose that  the families  
$$
\{id_{\hueca{S}_t}\}_{t\in [1,n]} \hbox{ and } 
\{\gamma_t:\hueca{L}_t\rightmap{}\hueca{R}_t\}_{t\in [1,n]}$$
  are a kernel-generating system for $\Coker^1:{\cal P}^1(A)\rightmap{}A\g\Mod$. 
 Then, any collection of isomorphisms  
$\hueca{L}_t\cong \hueca{L}'_t$, 
$\hueca{R}_t\cong \hueca{R}'_t$, and 
$\hueca{S}_t\cong \hueca{S}'_t$, for $t\in [1,n]$, in ${\cal P}^1(A)$ induces  families of  morphisms 
$\{id_{\hueca{S}'_t}\}_t$ and $\{\gamma'_t:\hueca{L}'_t\rightmap{}\hueca{R}'_t\}_t$ that are a kernel-generating system for  $\Coker^1:{\cal P}^1(A)\rightmap{}A\g\Mod$. 
\end{remark}

We say that a connected finite quiver $Q$ is \emph{trivial} iff it is only one vertex with no arrows. 

\begin{lemma}\label{R: vale Prop para las gammas(t)}  With the preceding notation, assume furthermore that $\Lambda=kQ/I$ for a  non-trivial connected quiver $Q$. Then, the families 
$$
\{id_{S(\Lambda e_t)}\}_{t\in [1,n]} \hbox{ and } 
\{\gamma_t:L(\Lambda e_t)\rightmap{}R(\Lambda e_t)\}_{t\in [1,n]}$$
 are a kernel-generating system for $\Coker^1:{\cal P}^1(\Lambda)\rightmap{}\Lambda\g\Mod$.
 \end{lemma}
 
\begin{proof}
 The morphisms $\phi_t$ and $\psi_t$ are special morphisms associated to the indecomposable projective $\Lambda$-module $P = \Lambda e_t$, as in (\ref{L: morfismos universales de P}), for $A=\Lambda$. Moreover, the morphisms $\alpha_t$, $\beta_t$, and $\gamma_t$ are the special morphisms associated to $P$ in (\ref{N: notation alpa, beta y gamma}). So, the characterization of the morphisms (\ref{P: caract de u con cokeru=0}) in ${\cal P}^1(\Lambda)$ with zero cokernel can be stated using these specific morphisms.

 Having in mind (\ref{L: L(Lambdaei), R(Lambdaej) inesc no isomorfos entre si}), we only need to show that $\gamma_t\not=0$, for all $t\in [1,n]$. If $\gamma_t=0$, then $\phi_t=0$ and $\psi_t=0$. Notice that $\phi_t=0$ implies that there are no arrows in $Q$ starting at the vertex $t$ and $\psi_t=0$ implies that there are no arrows in $Q$ ending at the vertex $t$. Thus $t$ is an isolated vertex in the connected quiver $Q$. So, 
 $Q$ is trivial,  a contradiction.  
\end{proof}

\begin{remark}\label{R: puedo tener R(Lambda et) cong L(Lambda es)} Notice that $L(\Lambda e_t)\not\cong R(\Lambda e_t)$,   if $\Lambda=kQ/I$ has infinite representation type and $Q$ is connected. 
 But it is possible to have $L(\Lambda e_t)\cong R(\Lambda e_s)$, when $t\not=s$. For example, for the algebra with the following quiver and radical square zero
$$1\rightmap{ \ \gamma \ }2\dobleflechava{a}{b}3,$$
we have  $L(\Lambda e_1)\cong R(\Lambda e_2)$. 

For the first remark, assume that the algebra $\Lambda=kQ/I$  has  infinite representation type.   If $L(\Lambda e_t)\cong R(\Lambda e_t)$, for some $t\in [1,n]$, then there is only one arrow with initial point $t$ and there is only one arrow with terminal point $t$. Since $Q$ is connected, this implies that $Q$ has only one vertex and a loop. This implies that $kQ/I$ is of finite representation type, contradicting our hypothesis. 
\end{remark}

From now until the end of this section, we work under the assumptions and with the notation of the following statement.    

\begin{lemma}\label{L: extensiones y algebras de carcaj}\label{R: realizacion de hat(Lambda) y }
Let $Q$ be a finite quiver and $I$ and admissible ideal of the quiver algebra $kQ$, where $k$ is any field. Assume that $K:k$ is any field extension of $k$. Then, the ideal 
$\langle I\rangle$ generated by $I$ is an admissible ideal of the quiver algebra $KQ$ and we have an isomorphism of $K$-algebras 
$$\theta: KQ/\langle I\rangle\rightmap{}(kQ/I)\otimes_kK$$
such that $\theta(\bar{e}_t)=e_t\otimes 1$, for any vertex $t$, and $\theta(\bar{\alpha})=\tilde{\alpha}\otimes 1$, for any arrow $\alpha$, where $\tilde{\alpha}$ is the class of $\alpha$ modulo $I$.   
\end{lemma}

\begin{proof} Since $I$ is an admissible ideal of $kQ$,  there is some $n$ such that every path $\gamma$ of $Q$ with length $\geq n$ belongs to $I$, so it also belongs to $\langle I\rangle$. Moreover, $I$ is generated by some linear combinations of  paths of length $\geq 2$, then so is $\langle I\rangle$. Hence,  $\langle I\rangle $ is an admissible ideal of $KQ$. 
We have the exact sequence of vector spaces
$$0\rightmap{}I\rightmap{}kQ\rightmap{}kQ/I\rightmap{}0.$$
Then, we have the commutative diagram with exact rows
$$\begin{matrix}0&\rightmap{}&I\otimes_kK&\rightmap{}&kQ\otimes_kK&\rightmap{}&(kQ/I)\otimes_kK&\rightmap{}&0\\
&&\shortlmapdown{\mu'}&&\shortlmapdown{\mu}&&\shortlmapdown{\overline{\mu}}&&\\
0&\rightmap{}&\langle I\rangle &\rightmap{}&KQ&\rightmap{}&KQ/\langle I\rangle&\rightmap{}&0,
\end{matrix}$$
where $\mu:kQ\otimes_kK\rightmap{}KQ$ denotes the multiplication map, which is a morphism of $K$-algebras and restricts to the morphism $\mu':I\otimes_kK\rightmap{}\langle I\rangle$. 

Now, consider the unique morphism of $K$-algebras $\phi:KQ\rightmap{}
kQ\otimes_k K$ such that $\phi(\tau_t)=\tau_t\otimes 1$, for any trivial path $\tau_t$ of $Q$, and $\phi(\alpha)=\alpha\otimes 1$, for any arrow $\alpha$ of $Q$. Then, $\mu$ is an isomorphism  with inverse $\phi$. It follows that $\mu'$ is also an isomorphism. Therefore, we have an induced isomorphism of $K$-algebras $\bar{\mu}$. The inverse $\theta$ of $\bar{\mu}$ is the isomorphism described in the statement of this lemma. 
\end{proof}

\begin{remark}\label{R: extension and radical} From the preceding statement it follows that, for $\Lambda=kQ/I$ and $\hat{\Lambda}=\Lambda\otimes_kK$, we have $(\rad\Lambda) \otimes_k K = \rad(\Lambda \otimes_k K)$  and that each $\hat{\Lambda}\hat{e}_t$ is an indecomposable projective $\hat{\Lambda}$-module, where $\hat{e}_t=e_t\otimes 1$, see \cite[(3.3)]{Kasjan}. 
\end{remark}
 
\begin{notation}\label{N: bar notation para KQ/(I)} Once we have a fixed field extension $K:k$, as before, we keep the notation described in (\ref{D: L(Lambdaex) R(Lambdaex)}) for the $k$-algebra $\Lambda$, and we denote with an overline the corresponding morphisms for the $K$-algebra $\bar{\Lambda}=KQ/\langle I\rangle$. Thus:  
$$L_{\bar{\Lambda}\bar{e}_t}=[\oplus_{\alpha:t\rightarrow  s}\bar{A}\bar{e}_s]\rightmap{\bar{\phi}_t}\bar{\Lambda}\bar{e}_t
\hbox{ \  , \  }
\bar{\Lambda}\bar{e}_t\rightmap{\bar{\psi}_t}R_{\bar{\Lambda}\bar{e}_t}=[\oplus_{\beta:s\rightarrow  t}\bar{A}\bar{e}_s]
$$
$$L(\bar{\Lambda}\bar{e}_t)\rightmap{\overline{\alpha}_t}U(\bar{\Lambda}\bar{e}_t)
\hbox{ \ , \  }
U(\bar{\Lambda}\bar{e}_t)\rightmap{\overline{\beta}_t}
R(\bar{\Lambda}\bar{e}_t)\hbox{  \ ,  \ }
\bar{\gamma}_t=\bar{\beta}_t\bar{\alpha}_t:L(\bar{\Lambda}\bar{e}_t)\rightmap{}R(\bar{\Lambda}\bar{e}_t)$$
\end{notation}

\begin{notation}\label{R: def alfa, beta, gama para Lambda(otimes)k(x)} Now, we concentrate on  the finite-dimensional $K$-algebra $\hat{\Lambda}:=\Lambda\otimes_kK$.  
  We use again the arrows of $Q$ in the following descriptions.  We have  the object  $L(\Lambda e_t)\otimes_kK$ in ${\cal P}^1(\hat{\Lambda})$  given by the morphism 
$$\begin{matrix}L_{\Lambda e_s}\otimes_k K=\left[\bigoplus_{\alpha:t\rightarrow s }\Lambda e_s\otimes_kK\right]&\rightmap{ \ \ \phi_t\otimes id \ \ }&\Lambda e_t\otimes_kK\end{matrix}$$
where $\phi_t=(\rho_\alpha)_{\alpha:t\rightarrow s}$ and each  $\rho_\alpha:\Lambda e_s\rightmap{}\Lambda e_t$
is the right multiplication by the arrow $\alpha$. 
We also have the object $R(\Lambda e_t)\otimes_kK$  in ${\cal P}^1(\hat{\Lambda})$  given by the morphism 
$$\begin{matrix}\Lambda e_t\otimes_kK&\rightmap{ \  \ \psi_t \otimes id\ \ }&\left[\bigoplus_{\beta:s\rightarrow t} \Lambda e_s\otimes_kK\right]=R_{\Lambda e_t}\otimes_kK\end{matrix}$$
where $\psi_t=(\rho_\beta)^t_{\beta:s\rightarrow t}$.  
Then, we have the following morphisms in ${\cal P}(\hat{\Lambda})$:  
$$\hat{\alpha}_t:=(\phi_t\otimes id,id):L(\Lambda e_t)\otimes_kK\rightmap{}U(\Lambda e_t)\otimes_kK, $$
$$\hat{\beta}_t:=(id,\psi_t\otimes id):U(\Lambda e_t)\otimes_k K\rightmap{}R(\Lambda e_t)\otimes_k K;$$
and, in ${\cal P}^1(\hat{\Lambda})$, the morphism 
$$\hat{\gamma}_t:=\hat{\beta}_t\hat{\alpha}_t=(\phi_t\otimes id,\psi_t\otimes id):L(\Lambda e_t)\otimes_kK\rightmap{}R(\Lambda e_t)\otimes_kK.$$
\end{notation}

\begin{remark}\label{R: comparing bar morphisms with morphisms tensor id} For the fixed field extension $K:k$, we have the $k$-algebra $\Lambda=kQ/I$ and the $K$-algebra 
$\bar{\Lambda}:=KQ/\langle I\rangle$. We will relate the morphisms of (\ref{D: L(Lambdaex) R(Lambdaex)}) with those of (\ref{N: bar notation para KQ/(I)}). For this, we consider the $K$-algebra  $\hat{\Lambda}:=(kQ/I)\otimes_k K$ and the isomorphism of $K$-algebras $\theta:\bar{\Lambda}\rightmap{}\hat{\Lambda}$ such that $\theta(\bar{e}_t)=e_t\otimes 1$ and $\theta(\bar{\alpha})=\tilde{\alpha}\otimes 1$, for any arrow $\alpha$ of $Q$, as in (\ref{L: extensiones y algebras de carcaj}). Then, $\theta$ has an inverse   $\tau: \hat{\Lambda}\rightmap{}\bar{\Lambda}$, which defines 
an isomorphism of categories $F_\tau:\bar{\Lambda}\g\mod\rightmap{}\hat{\Lambda}\g\mod$ obtained by restriction through $\tau$. 

Notice that the map $\theta$ induces an isomorphism of $\hat{\Lambda}$-modules $$\theta_t:F_\tau(\bar{\Lambda}\bar{e}_t)={\,}_\tau\bar{\Lambda}\overline{e}_t\rightmap{}\hat{\Lambda}(e_t\otimes 1)=\Lambda e_t\otimes_kK$$
such that, for each arrow $\alpha:t\rightmap{}s$ of $Q$, we have the following commutative diagram in $\hat{\Lambda}\g\mod$
$$\begin{matrix}
{\,}_\tau\bar{\Lambda}\overline{e}_s&\rightmap{\theta_s}&\hat{\Lambda}(e_s\otimes 1)&=&\Lambda e_s\otimes_kK\\
\shortlmapdown{\rho_\alpha}&&\shortlmapdown{\rho_{\alpha\otimes 1}}&&\shortlmapdown{\rho_\alpha\otimes id}\\
{\,}_\tau\bar{\Lambda}\overline{e}_t&\rightmap{\theta_t}&\hat{\Lambda}(e_t\otimes 1)&=&\Lambda e_t\otimes_k K\\
\end{matrix}$$
where the maps $\rho_\alpha$ denote right multiplication by $\alpha$ and 
$\rho_{\alpha\otimes 1}$ is right multiplication by $\alpha\otimes 1$.   
From this we obtain commutative squares in $\hat{\Lambda}\g\mod$
$$\begin{matrix}
\begin{matrix}
F_\tau(L_{\bar{\Lambda}\bar{e}_t})&
\rightmap{\bar{\phi}_t}&F_\tau(\bar{\Lambda}\bar{e}_t)\\
\shortlmapdown{\oplus \theta_s}&&\shortlmapdown{\theta_t}\\
L_{\Lambda e_t}\otimes_kK &\rightmap{\phi_t\otimes id}&\Lambda e_t\otimes_k K\\
\end{matrix}
\hbox{ \ \ and \ \  }
\begin{matrix}
F_\tau(\bar{\Lambda}\bar{e}_t)
&
\rightmap{\bar{\psi}_t}&F_\tau(R_{\bar{\Lambda}\bar{e}_t})\\
\shortlmapdown{\theta_t}&&\shortlmapdown{\oplus \theta_s}\\
\Lambda e_t\otimes_kK&\rightmap{\psi_t\otimes id}&R_{\Lambda e_t}\otimes_k K\\
\end{matrix}
\end{matrix}$$ 
where the vertical morphisms are isomorphisms. The isomorphism $\tau$ induces an isomorphism of categories
 ${\cal P}(\bar{\Lambda})\rightmap{}{\cal P}(\hat{\Lambda})$ which we denote with the same symbol $F_\tau$. Then, with the notation of (\ref{R: def alfa, beta, gama para Lambda(otimes)k(x)}), we get isomorphisms in ${\cal P}(\hat{\Lambda})$: 
$$F_\tau(L(\bar{\Lambda}\bar{e}_t))\cong  
L(\Lambda e_t)\otimes_k K \hbox{ \ \ and \ \ }
F_\tau(R(\bar{\Lambda}\bar{e}_t))\cong  
R(\Lambda e_t)\otimes_k K. 
$$
 Similarly, the isomorphisms $\theta_t$ induce the following isomorphisms in ${\cal P}(\hat{\Lambda})$ 
$$F_\tau(S(\bar{\Lambda}\bar{e}_t))\cong S(\Lambda e_t)\otimes_k K \hbox{ \ \ \ and \ \ \ } F_\tau(U(\bar{\Lambda}\bar{e}_t))\cong U(\Lambda e_t)\otimes_k K.$$ Moreover, with these isomorphisms, the following diagrams  commute 
$$\begin{matrix}F_\tau(L(\bar{\Lambda}\bar{e}_t))&\rightmap{\bar{\alpha}_t}&F_\tau(U(\bar{\Lambda}\bar{e}_t))&\rightmap{\bar{\beta}_t}&F_\tau(R(\bar{\Lambda}\bar{e}_t))\\
\shortlmapdown{\cong}&&\shortlmapdown{\cong}&&\shortlmapdown{\cong}\\
L(\Lambda e_t)\otimes_kK&\rightmap{\hat{\alpha}_t}&U(\Lambda e_t)\otimes_kK&\rightmap{\hat{\beta}_t}&R(\Lambda e_t)\otimes_kK\\
\end{matrix}
$$
Therefore, the morphisms $F_\tau(\bar{\alpha}_t)$, $F_\tau(\bar{\beta}_t)$, and  $F_\tau(\bar{\gamma}_t)$ are, respectively, isomorphic to the morphisms a $\hat{\alpha}_t$,  $\hat{\beta}_t$, and $\hat{\gamma}_t$ in   ${\cal P}(\hat{\Lambda})$. 
\end{remark}

\begin{lemma}\label{L: extendidos generadores de Coker  canónicos son generadores} If the algebra $\Lambda=kQ/I$ has 
 a non-trivial connected quiver $Q$,  then,  the families  
$$
\{id_{S(\Lambda e_t)\otimes_kK}\}_{t\in [1,n]}\hbox{ and } 
\{\hat{\gamma}_t:L(\Lambda e_t)\otimes_kK\rightmap{}R(\Lambda e_t)\otimes_kK\}_{t\in [1,n]}$$
are a kernel-generating system for $\Coker^1:{\cal P}^1(\hat{\Lambda})\rightmap{}\hat{\Lambda}\g\Mod$. 
\end{lemma}

\begin{proof} The families 
$
\{id_{S(\bar{\Lambda}\bar{e}_t)}\}_{t\in [1,n]}
\hbox{ and } 
\{\bar{\gamma}_t:L(\bar{\Lambda} \bar{e}_t)\rightmap{}R(\bar{\Lambda} \bar{e}_t)\}_{t\in [1,n]}$ are a kernel-generating system for $\Coker^1:{\cal P}^1(\bar{\Lambda})\rightmap{}\bar{\Lambda}\g\Mod$, because (\ref{R: vale Prop para las gammas(t)}) holds for the $K$-algebra $\bar{\Lambda}$. Consider their images $
\{id_{F_\tau(S(\bar{\Lambda}\bar{e}_t))}\}_{t\in [1,n]}
\hbox{ and } 
\{F_\tau(\bar{\gamma}_t):F_\tau(L(\bar{\Lambda} \bar{e}_t))\rightmap{}F_\tau(R(\bar{\Lambda} \bar{e}_t))\}_{t\in [1,n]}$ under the equivalence 
 $F_\tau:{\cal P}^1(\bar{\Lambda})\rightmap{}{\cal P}^1(\hat{\Lambda})$. These images are a kernel-generating system for $\Coker^1:{\cal P}^1(\hat{\Lambda})\rightmap{}\hat{\Lambda}\g\Mod$ because  we have the commutative diagram of functors 
 $$\begin{matrix}{\cal P}^1(\bar{\Lambda})&\rightmap{\Coker}&\bar{\Lambda}\g\Mod\\
 \shortlmapdown{F_\tau}&&\shortlmapdown{F_\tau}\\
 {\cal P}^1(\hat{\Lambda})&\rightmap{\Coker}&\hat{\Lambda}\g\Mod.\\
 \end{matrix}$$
Indeed the first vertical functor $F_\tau$  maps morphisms $u$ in ${\cal P}^1(\bar{\Lambda})$ between objects with finite-length underlying modules to with $\Coker(u)=0$ onto the corresponding morphisms $F_\tau(u)$ in ${\cal P}^1(\hat{\Lambda})$.  
 Then, from (\ref{R: comparing bar morphisms with morphisms tensor id}) and (\ref{R: isomorfos a generadores son generadores}), our claim follows. 
\end{proof}

\begin{remark}\label{R: extendidos de generadores son generadores} 
Assume that the families 
$
\{id_{\hueca{S}_t}\}_{t\in [1,n]}\hbox{ and } 
\{\gamma'_t:\hueca{L}_t\rightmap{}\hueca{R}_t\}_{t\in [1,n]}$
 are produced from the families 
$
\{id_{S(\Lambda e_t)}\}_{t\in [1,n]}\hbox{ and }
\{\gamma_t:L(\Lambda e_t)\rightmap{}R(\Lambda e_t)\}_{t\in [1,n]}$ by the application of (\ref{R: isomorfos a generadores son generadores}) to some given  isomorphisms 
 $L(\Lambda e_t)\cong \hueca{L}_t$, $R(\Lambda e_t)\cong \hueca{R}_t$, and $S(\Lambda e_t)\cong \hueca{S}_t$. 
We know that these families are kernel-generating systems for $\Coker^1:{\cal P}^1(\Lambda)\rightmap{}\Lambda\g\Mod$. 

Then the families 
$
\{id_{\hueca{S}_t\otimes_kK}\}_t \hbox{ and } 
\{\gamma'_t\otimes id:\hueca{L}_t\otimes_kK\rightmap{}\hueca{R}_t\otimes_kK\}_t$
 are a kernel-generating  system  for $\Coker^1:
{\cal P}^1(\hat{\Lambda})\rightmap{}\hat{\Lambda}\g\Mod$. 
Indeed, the given isomorphisms induce  isomorphisms $S(\Lambda e_t)\otimes_kK\cong \hueca{S}_t\otimes_kK$ and commutative diagrams  
$$\begin{matrix}L(\Lambda e_t)\otimes_kK&\rightmap{\gamma_t\otimes id}&R(\Lambda e_t)\otimes_kK\\
\shortlmapdown{\cong}&&\shortrmapdown{\cong}\\ \hueca{L}_t\otimes_k K&\rightmap{\gamma'_t\otimes id}&\hueca{R}_t\otimes_kK.\\
\end{matrix}$$
Then, our claim follows from (\ref{R: isomorfos a generadores son generadores}) and (\ref{L: extendidos generadores de Coker  canónicos son generadores}). 
\end{remark}

\section{Setting for reduction to minimal ditalgebras}

From now on, $k$ is an algebraically closed field, $\Lambda=Q/I$ is a tame  finite dimensional basic indecomposable $k$-algebra, and $1=\sum_{t=1}^ne_t$ is the canonical decomposition of the unit of $\Lambda$ as sum of primitive orthogonal idempotents.
In the following, we recall some basic notions and  fundamental facts relating $\Lambda\g\Mod$ with the category of modules ${\cal B}\g\Mod$, over some minimal ditalgebra ${\cal B}$. 

\begin{remark}\label{L: ditalg minimales} Recall that a minimal ditalgebra ${\cal B}$ is a layered ditalgebra with triangular seminested layer $(R,W)$ with $W_0=0$, see \cite[\S23]{BSZ}. Thus $R$ is a \emph{minimal algebra}, that is a  finite product of indecomposable  algebras $R=\prod_{z\in \frak{P}}R_z$.
The factors $R_z$ are of two different types as follows: 
we have  $\frak{P}=\frak{M}\uplus \frak{Z}$ with  
$R_z\cong k[x]_{h_z}$,  for 
$z\in \frak{M}$, and $R_z\cong k$, for $z\in \frak{Z}$.
 Moreover, if $1=\sum_{z\in \frak{P}}\frak{f}_z$ is the corresponding decomposition of the unit in $R$ as a sum of primitive orthogonal idempotents,  we consider the factors $R_z$ as subspaces of $R$ through the isomorphisms $R_z\cong k[x]_{h_z}\frak{f}_z$, for $z\in \frak{M}$, and $R_z\cong k\frak{f}_z$, for $z\in \frak{Z}$.  

The elements of $\frak{P}$ are called \emph{the points of} ${\cal B}$ and the elements of $\frak{M}$ are called the \emph{marked points of} ${\cal B}$. 

Given $z\in \frak{M}$, the indecomposable finite-dimensional $k[x]_{h_z}$-modules are, up to isomorphism, of the form $k[x]_{h_z}/(x-\lambda)^m$, where $m\in \hueca{N}$ and $\lambda\in k$ with $h_z(\lambda)\not=0$. In this situation, $h_z$ acts invertibly on the  $k[x]$-module $k[x]/(x-\lambda)^m$, so the latter is naturally a $k[x]_{h_z}$-module, an $R_z$-module, and an $R$-module. We have an isomorphism of $R_z$-modules  $k[x]_{h_z}/(x-\lambda)^m\cong k[x]/(x-\lambda)^m=[k[x]/(x-\lambda)^m]\frak{f}_z$. 
So, the indecomposable finite-dimensional ${\cal B}$-modules are, up to isomorphism, of the form  
$[k[x]/(x-\lambda)^m]\frak{f}_z$, with $z\in \frak{M}$, $\lambda\in k$ with  $h_z(\lambda)\not=0$, and $m\in \hueca{N}$, which we call the \emph{marked indecomposables}, and $S_{\frak{f}_z}=k\frak{f}_z$, with $z\in \frak{Z}$, see \cite[(32.3)]{BSZ}. The \emph{marked simple  ${\cal B}$-modules based at $z$} are     $S_{z,\lambda}=[k[x]/(x-\lambda)]\frak{f}_z$, for $z\in \frak{M}$ and $h_z(\lambda)\not=0$. The generic ${\cal B}$-modules are, up to isomorphism,  the fields of fractions $K_z$ of $R\frak{f}_z$, for $z\in \frak{M}$, with their  natural structure of $R\frak{f}_z$-module. Thus, $\frak{f}_z$ acts centrally on $K_z$, and we can identify it with $k(x)\frak{f}_z$.  We say that a ${\cal B}$-module $N$ is \emph{based at the point $z$} if $\frak{f}_zN=N$. 
\end{remark}

The argument for the proof of (\ref{T: gen-bricks vs bricks}) relies  on the following fundamental fact, 
see \cite{CB2} and \cite[(29.9)]{BSZ}. 

\begin{theorem}\label{T: gens for tame algs}
 Let $\Lambda$ be a finite-dimensional tame algebra  and let $d\in \hueca{N}$. Then, there is a minimal ditalgebra ${\cal B}$ and a full functor  which preserves indecomposables and reflects isomorphism classes 
 $$T:{\cal B}\g\Mod\rightmap{}\Lambda\g\Mod,$$ 
 such that its composition with the canonical embedding $L_{\cal B}:B\g\Mod\rightmap{}{\cal B}\g\Mod$ is isomorphic to the tensor product $Z\otimes_B-$ by some $\Lambda\g B$-bimodule $Z$, where $B=R$ denotes the subalgebra of degree-zero elements of the underlying tensor algebra of ${\cal B}$. Moreover, $Z$ is finitely generated by the right and such that:
\begin{enumerate}
\item Every generic $\Lambda$-module with endolength $\leq d$ is isomorphic to $Z\otimes_BK$, for some generic ${\cal B}$-module $K$ of ${\cal B}$.
\item  Every $\Lambda$-module of dimension $\leq d$ is isomorphic to $Z\otimes_BN$ for some $N\in {\cal B}\g\mod$.  
\end{enumerate}
\end{theorem}

\begin{remark}\label{R: first choice of min ditalg} Assume that $G$ is a  generic $\Lambda$-module  with endolength $d_0$. Given $d\in \hueca{N}$, the functor $T$ of the last theorem is obtained as a composition of functors 
$${\cal B}\g\Mod\rightmap{H}{\cal D}\g\Mod\rightmap{\Xi_\Lambda}{\cal P}^1(\Lambda)\rightmap{ \ \Coker \ }\Lambda\g\Mod,$$
where ${\cal D}$ is the Drozd's ditalgebra of $\Lambda$,  $\Xi_\Lambda$ is the usual equivalence, see \cite[\S19]{BSZ}, and $H$ is a composition of  full and faithful reduction functors. There, $H$ is such that any $N\in {\cal D}\g\Mod$ with endolength $\leq (1+\dim_k\Lambda)\times d$ or with $\dim_kN\leq d$, are of the form $H(M)\cong N$, see the prof of \cite[(29.9)]{BSZ}. Using \cite[(29.5)]{BSZ}, we see that taking $d>d_0$ big enough, we get that $L(\Lambda e_t)$,  $R(\Lambda e_t)$, and $S(\Lambda e_t)$, for $t\in [1,n]$,  are isomorphic to some objects in the image of $\Xi_\Lambda H$, and also $G\cong Z\otimes_{R_{z_0}} k(x)\frak{f}_{z_0}$, for some marked point $z_0$ of ${\cal B}$.  So, $\Xi_\Lambda H(L_{l_t})\cong L(\Lambda e_t)$, $\Xi_\Lambda H(L_{r_t})\cong R(\Lambda e_t)$, $\Xi_\Lambda H(L_{z_t})\cong S(\Lambda e_t)$, for appropriate $L_{l_t}, L_{r_t}, L_{z_t}\in {\cal B}\g\mod$, and $\Xi_\Lambda H(k(x)\frak{f}_{z_0})\cong \hat{G}$, where $\Coker (\hat{G})\cong G$. 

The bimodule $Z$ of the last theorem is $\Coker\Xi_\Lambda H(B)$. 
From (\ref{T: gens for tame algs}), we know that the generic $\Lambda$-module $G$ has the form $G\cong Z\otimes_BK_{z_0}$, for some generic ${\cal B}$-module $K_{z_0}=k(x)\frak{f}_{z_0}$ of ${\cal B}$. Following Crawley-Boevey's argument for the proof of (\ref{T: CB-realizations}), we know that the $\Lambda\g B\frak{f}_{z_0}$-bimodule $M:=Z\frak{f}_{z_0}$ is a realization  of $G\cong Z\otimes_BK_{z_0}\cong M\otimes_{B\frak{f}_{z_0}}k(x)\frak{f}_{z_0}$, over the rational algebra $\Gamma:=B\frak{f}_{z_0}$. 
\end{remark}

\begin{remark}\label{R: variar parametrizaciones} From now on, 
$G$ denotes a fixed generic $\Lambda$-module  with endolength $d_0$. Assume that $T'$ is a functor  given as a composition 
of functors 
$${\cal B}'\g\Mod\rightmap{H'}{\cal D}\g\Mod\rightmap{\Xi_\Lambda}{\cal P}^1(\Lambda)\rightmap{ \ \Coker \ }\Lambda\g\Mod,$$
where ${\cal B}'$ is a minimal ditalgebra, $H'$ is a composition of full and faithful reduction functors. Suppose that 
 there are $L_{l_t}, L_{r_t}, L_{z_t}\in {\cal B}'\g\mod$, such that $\Xi_\Lambda H'(L_{l_t})\cong L(\Lambda e_t)$, $\Xi_\Lambda H'(L_{r_t})\cong R(\Lambda e_t)$, and $\Xi_\Lambda H(L_{z_t})\cong S(\Lambda e_t)$, for $t\in [1,n]$. Furthermore, assume that there is a marked point $z_0$ of ${\cal B}'$ such that $\Xi_\Lambda H'(k(x)\frak{f}_{z_0})\cong \hat{G}$, where $\Coker (\hat{G})\cong G$. Then, as before, we get a $\Lambda\g B'$-bimodule $Z'=\Coker \Xi_\Lambda H'(B')$ and a realization $M'=Z'\frak{f}_{z_0}$ of $G$. 
 
 In the following, we will make some convenient adjustments to the minimal ditalgebra ${\cal B}$ of (\ref{R: first choice of min ditalg}), to obtain another minimal ditalgebra ${\cal B}'$ satisfying the requirements specified in the preceding paragraph. This, will permit us to 
study the relation of morphisms between relevant  indecomposables in ${\cal B}'\g\Mod$ and ${\cal P}^1(\Lambda)$ using the new realization $M'$ of $G$. 

We start by fixing the minimal ditalgebra ${\cal B}$ and the context  presented in (\ref{R: first choice of min ditalg}), including the realization $M$ of $G$  which appeared there, but we will adjust it conveniently in the next paragraphs.  
\end{remark}

In the following, if ${\cal L}$ is a finite set of non-isomorphic indecomposables in ${\cal B}\g\mod$, where ${\cal B}$ is a minimal ditalgebra, we denote by $m({\cal L})$ the number of marked indecomposables in ${\cal L}$. In the next lemma, ${\cal B}^X$ denotes the ditalgebra constructed from ${\cal B}$, using  an admissible $R$-module $X$, see \cite[\S12]{BSZ}.

\begin{lemma}\label{L reemplazando marcados} Let ${\cal B}$ be a minimal ditalgebra and ${\cal L}$ a finite set of non-isomorphic indecomposables in ${\cal B}\g\mod$ with $m({\cal L})\geq 1$. Then there is a complete admissible $R$-module $X$ such that  ${\cal B}^X$ is a minimal ditalgebra and the associated  
full and faithful functor 
$$F^X: {\cal B}^X\g\Mod\rightmap{}{\cal B}\g\Mod$$ 
satisfies: 
\begin{enumerate}
\item  There is a finite set ${\cal L}^X$ of indecomposables in ${\cal B}^X\g\mod$ with $m({\cal L}^X)<m({\cal L})$, such that $F^X$ induces a bijection from ${\cal L}^X$ onto ${\cal L}$.
\item There is a bijection between marked points $\sigma:\frak{M}\rightmap{}\frak{M}^X$ such that:  
 $$F^X(k(x)\frak{f}_{\sigma(z)})\cong k(x)\frak{f}_z, \hbox{ for all } z\in \frak{M}.$$ 
\end{enumerate} 
\end{lemma}

\begin{proof} Take any marked indecomposable $N\in {\cal L}$, so there are a marked point $z_0\in \frak{M}$, some $\lambda_0\in k$ with $h_{z_0}(\lambda_0)\not=0$, and $r\in \hueca{N}$, with $$N\cong [k[x]/(x-\lambda_0)^r]\frak{f}_{z_0}.$$
We can choose $N\in {\cal L}$ such that 
$$r=\max\{\dim_k \hat{N}\mid \hat{N}\in {\cal L} \hbox{ is marked and based at }z_0\}.$$
Consider the complete admissible $B$-module 
$$X:=R(1-\frak{f}_{z_0})\oplus (R\frak{f}_{z_0})_{x-\lambda_0}\oplus (\oplus_{s=1}^rk[x]/(x-\lambda_0)^s).$$  
Then, we have the minimal ditalgebra ${\cal B}^X$ with layer $(S,W^X)$ such that 
$$S=\End_R(X)^{op}/\rad\End_R(X)^{op} =R(1-\frak{f}_{z_0})\times (R\frak{f}_{z_0})_{x-\lambda_0}\times (\times_{s=1}^r k\frak{f}''_s),$$
with marked points $\frak{M}^X=\frak{M}$ and non-marked points $\frak{Z}^X=\frak{Z}\bigcup\{1,\ldots,r\}$, where each idempotent $\frak{f}''_s$ corresponds to the indecomposable direct summand $k[x]/(x-\lambda_0)^s$ of $X$. Denote by  $\frak{f}_{z_0}'$  the idempotent corresponding to the direct factor $(R\frak{f}_{z_0})_{x-\lambda_0}$ of $S$. We have chosen a notation such that we can take $\sigma$ as the identity map.

The functor $F^X:{\cal B}^X\g\Mod\rightmap{}{\cal B}\g\Mod$ is full and faithful because $X$ is  admissible and complete, see \cite[\S12 and \S13]{BSZ}. 

Consider the finite set ${\cal L}^X$ of ${\cal B}$-modules in ${\cal B}^X\g\mod$ formed by the indecomposables $M\frak{f}_z$ in ${\cal L}$ based at points $z\not= z_0$; the indecomposables of the form $k\frak{f}''_s$ where $[k[x]/(x-\lambda_0)^s]\frak{f}_{z_0}$ belongs to ${\cal L}$; and the indecomposables  of the form $[k[x]/(x-\lambda)^m]\frak{f}'_{z_0}$, with $m\in \hueca{N}$ and $\lambda\not=\lambda_0$, such that $[k[x]/(x-\lambda)^m]\frak{f}_{z_0}$ belongs to ${\cal L}$. So, clearly, $m({\cal L}^X)<m({\cal L})$. 

We recall that the finite-dimensional 
${\cal B}$-modules $Y$ of the form $F^X(Y')$, for some $Y'\in {\cal B}^X\g\Mod$ are those ${\cal B}$-modules  $Y$ such that, as an $R$-module, are of the form $Y\cong X\otimes_SY'$, for some $S$-module $Y'$, see \cite[(16.1)]{BSZ}. 

Here, we have  $N\cong f''_r(X)\cong X\frak{f}''_r\cong X\otimes_Sk\frak{f}''_r$, where $\frak{f}''_r$ is the idempotent corresponding to the direct summand 
$k[x]/(x-\lambda_0)^r$ of $X$. Thus, $F^X(kf''_r)\cong N$. Similarly, we get that $F^X(k\frak{f}''_s)\cong [k[x]/(x-\lambda_0)^s]\frak{f}_{z_0}$, whenever this last module belongs to ${\cal L}$.    
For $m\in \hueca{N}$ and $\lambda\not=\lambda_0$, we have 
$$X\otimes_S [k[x]/(x-\lambda)^m]\frak{f}'_{z_0}\cong (R\frak{f}_{z_0})_{x-\lambda_0}\otimes_S[k[x]/(x-\lambda)^m]\frak{f}'_{z_0}\cong [k[x]/(x-\lambda)^m]\frak{f}_{z_0},$$
so $F^X([k[x]/(x-\lambda)^m]\frak{f}'_{z_0})\cong [k[x]/(x-\lambda)^m]\frak{f}_{z_0}$. 
 Similarly, for any indecomposable $M\frak{f}_z\in {\cal L}^X$ with $z\not=z_0$, we have $F^X(M\frak{f}_z)\cong M\frak{f}_z$. So,   $F^X$ induces a bijective map from ${\cal L}^X$ to ${\cal L}$ and $(1)$  holds. 

We have $k(x)\frak{f}_{z_0}\cong (R\frak{f}_{z_0})_{x-\lambda_0}\otimes_S k(x)\frak{f}'_{z_0}\cong X\otimes_Sk(x)\frak{f}'_{z_0}$, which implies that $F^X(k(x)\frak{f}'_{z_0})\cong k(x)\frak{f}_{z_0}$.   
Similarly, one shows that $F^X(k(x)\frak{f}_z)\cong k(x)\frak{f}_z$ when $z\in \frak{M}\setminus\{z_0\}$. 
\end{proof}

\begin{proposition}\label{P: ditalg min simplificada} Let ${\cal B}$ be a minimal ditalgebra, ${\cal L}$ a finite set of non-isomorphic indecomposables in ${\cal B}\g\mod$, and $K_{z_0}$ a generic ${\cal B}$-module, based at the marked point $z_0$. Then there is a minimal ditalgebra ${\cal B}'$ and a full and faithful functor  
$$F':{\cal B}'\g\Mod\rightmap{}{\cal B}\g\Mod$$
such that:
\begin{enumerate}
\item ${\cal B}'$ has a unique marked point $y_0$, thus a unique generic module $K_{y_0}'$. 
\item There is a finite set ${\cal L}'$ of non-marked  indecomposables in ${\cal B}'\g\mod$  such that $F'$ induces a bijection from ${\cal L}'$ onto ${\cal L}$.
\item We have $F'(K'_{y_0})\cong K_{z_0}$. 
\end{enumerate}
\end{proposition}

\begin{proof} Applying repeatedly (\ref{L reemplazando marcados}), we obtain a minimal ditalgebra ${\cal B}^u$ and a full and faithful functor $F^u:{\cal B}^u\g\Mod\rightmap{}{\cal B}\g\Mod$ such that: 
\begin{enumerate}
\item The functor $F^u$ induces a bijection between a set of non-marked indecomposables ${\cal L}^u$ in ${\cal B}^u\g\mod$ and ${\cal L}$. 
\item There is a marked vertex $y_0$ of ${\cal B}^u$ such that the generic ${\cal B}^u$-module $K_{y_0}^u$ based at $y_0$ satisfies that $F^u(K^{u}_{y_0})\cong K_{z_0}$.   
\end{enumerate}
Now, consider the minimal ditalgebra ${\cal B}':={\cal B}^{ud}$ obtained from ${\cal B}^u$ by deletion of the idempotents $\frak{f}_y$ corresponding to marked points $y$ of ${\cal B}^u$ with $y\not=y_0$. Denote by $F'$ the composition 
$${\cal B}^{ud}\g\Mod\rightmap{F^d}{\cal B}^u\g\Mod\rightmap{F^u}{\cal B}\g\Mod.$$
Then, the minimal ditalgebra ${\cal B}'$ and the functor $F'$ satisfy the conditions of the proposition. 
\end{proof}

Consider the family ${\cal L}$ of objects in ${\cal P}^1(\Lambda)$ formed by a complete set of representatives of the isomorphism classes of the objects $L(\Lambda e_t)$, $R(\Lambda e_t)$, $S(\Lambda e_t)$, with $t\in [1,n]$. 

\begin{proposition}\label{P: el funtor F}
There is a minimal ditalgebra 
${\cal B}'$, 
with a unique marked point $z_0$, and a full and faithful functor $\Xi_\Lambda H':  {\cal B}'\g\Mod \rightmap{} {\cal P}^1(\Lambda)$ as in (\ref{R: variar parametrizaciones})
with the following properties:
\begin{enumerate}
\item ${\cal B}'$ has a layer 
$(R', W_1)$ with $R' = k[x]_h\frak{f}_{z_0} \times\left( \times_{z\in \frak{Z}} k\frak{f}_z\right)$.
\item For each $X\in {\cal L}$, there is some $z\in \frak{Z}$ such that $\Xi_\Lambda H'(S_{\frak{f}_z})\cong X$.
\item  $ \Xi_\Lambda H'(k(x)\frak{f}_{z_ 0}) \cong \hat{G}$. 
\end{enumerate}
\end{proposition}

\begin{proof} Adopt the notation of (\ref{R: first choice of min ditalg}). Applying (\ref{P: ditalg min simplificada}) to the finite family ${\cal L}'$ formed by the objects $L_{l_t}, L_{r_t}, L_{z_t}$ corresponding to the isoclasses of $L(\Lambda e_t), R(\Lambda e_t), S(\Lambda e_t)$ which conform ${\cal L}$, we have a minimal ditalgebra ${\cal B}'$ with a unique marked point $y_0$ and a full and  faithful functor  
$F'$ such that the composition 
$${\cal B}'\g\Mod\rightmap{F'}{\cal B}\g\Mod\rightmap{H}{\cal D}\g\Mod\rightmap{\Xi}{\cal P}^1(\Lambda)$$
has the following properties: there are non-marked points $l_t, r_t, z_t$ of ${\cal B}'$, indexed by points $t\in [1,n]$, such that
$$\begin{matrix}\Xi HF'(S_{\frak{f}_{l_t}})&\cong& \Xi H(L_{l_t})\cong L(\Lambda e_t),\hfill\\  
\Xi HF'(S_{\frak{f}_{r_t}})&\cong& \Xi H(L_{r_t})\cong R(\Lambda e_t),\hfill\\ 
\Xi HF'(S_{\frak{f}_{z_t}})&\cong& \Xi H(L_{z_t})\cong S(\Lambda e_t),\hfill\\
\Xi HF'(k(x)\frak{f}_{y_0})&\cong& \Xi H (k(x)\frak{f}_{z_0})\cong \hat{G}.\hfill
\end{matrix}$$
Thus, taking $H'=HF'$, we get what we wanted.  
\end{proof}

As remarked in (\ref{R: variar parametrizaciones}), we can assume that ${\cal B}$ already has the properties without primes, required in the last  proposition for ${\cal B}'$ with primes. We will denote by $F$ the composition $\Xi_\Lambda H:{\cal B}\g\Mod\rightmap{}{\cal P}^1(\Lambda)$ and again by $M$ the new associated realization of $G$.   

We are interested in the study of the relation between $\rad \End_\Lambda(G)$ and the families $\rad\End_\Lambda(M(\lambda))$. We will achieve this by reducing to the same study of the corresponding precedents $k(x)\frak{f}_{z_0}$ and $S_{z_0,\lambda}$ in ${\cal B}\g\Mod$ of $G$ and $M(\lambda)$, respectively.  

\section{Radical morphisms of ${\cal B}$-modules} 

From now on, we adopt the setting of (\ref{P: el funtor F}).  So ${\cal B}$ is a minimal ditalgebra with a unique marked point $z_0$ and a full and faithful functor $F:{\cal B}\g\Mod\rightmap{}{\cal P}^1(\Lambda)$ satisfying the unprimed conditions of that proposition. So, the layer $(R,W)$ of ${\cal B}$ has the form $R = k[x]_h\frak{f}_{z_0} \times\left( \times_{z\in \frak{Z}} k\frak{f}_z\right)$ and we have the idempotents $\frak{f}_{l_t}$ with  
$\hueca{L}_t:=F (S_{\frak{f}_{l_t}})\cong  L(\Lambda e_t)$, $\frak{f}_{r_t}$ with 
$\hueca{R}_t:=
F (S_{\frak{f}_{r_t}})
\cong  
 R(\Lambda e_t)$, 
 $\frak{f}_{z_t}$ with  
$\hueca{S}_t:=
F (S_{\frak{f}_{z_t}}) 
\cong S(\Lambda e_t)$. So, we have families $\{id_{\hueca{S}_t}\}_t$ and $\{\gamma'_t:\hueca{L}_t\rightmap{}\hueca{R}_t\}_t$ that are a kernel-generating system for $\Coker:{\cal P}^1(\Lambda)\rightmap{}\Lambda\g\Mod$. For the sake of simplicity, we will use the symbol $\gamma_t$ instead of $\gamma'_t$. Since ${\cal B}$ has only one marked point $z_0$, we simplify the notation $S_\lambda:=S_{z_0,\lambda}$, thus $\Coker F(S_\lambda)\cong M(\lambda)$, for 
 all $\lambda\in k$ with $h(\lambda)\not=0$, where $M$ is the realization of $G$ associated to this setting. Moreover,  we define  $\frak{f}_0:=\frak{f}_{z_0}$ and  $D(h):=\{\lambda\in k\mid h(\lambda)\not=0\}$,  and  we denote by $\frak{F}$ the whole set of primitive idempotents 
$\{\frak{f}_0\}\cup \{\frak{f}_z\mid z\in \frak{Z}\}$ of $R$.

\begin{remark}\label{R: morfismos en B-Mod}
Since ${\cal B}$ is minimal,  the $R$-$R$-bimodule  $W_1$ is  freely generated by a finite directed subset, see \cite[(23.2)]{BSZ}. As remarked in  \cite[(15.7)]{BSZ}, this is equivalent to the existence of an isomorphism of $R$-$R$-bimodules $W_1\cong R\otimes_S\hat{W}_1\otimes_SR$, for some finite-dimensional $S$-$S$-subbimodule $\hat{W}_1$ of $W_1$, where $S$ is the subalgebra  
 $$S=k\frak{f}_0\times \left(\times_{z\in \frak{Z}}k\frak{f}_z\right) \subseteq R.$$

Given $X,Y\in {\cal B}\g\Mod$, any morphism $\hat{h}^1\in \Hom_{R\g R}(W_1,\Hom_k(X,Y))$ determines a morphism $(0,\hat{h}^1)\in \Hom_{\cal B}(X,Y)$. Thus, any bimodule morphism  
$h^1\in \Hom_{S\g S}(\hat{W}_1,\Hom_k(X,Y))$ determines a morphism $(0,\hat{h}^1)\in \Hom_{\cal B}(X,Y)$,  where $\hat{h}^1$ denotes the extension of $h^1$ in $\Hom_{R\g R}(W_1,\Hom_k(X,Y))$. In order to simplify the notation, we will abuse of the language and use the same symbol $h^1$ to denote $h^1$ and its extension   $\hat{h}^1$: the context will always avoid any possible confusion. For the sake of precision, we will break this convention
in some definitions. 
\end{remark}

 We need to study some hom-spaces $\Hom_{\cal B}(X,Y)$ and compare them with $\Hom_{{\cal P}^1(\Lambda)}(F(X),F(Y))$ using the functor $F$, for special ${\cal B}$-modules $X$ and $Y$. So, we need to give a practical description of these morphisms. For this the following will be useful. 

\begin{lemma}\label{L: isos entre duales}
Consider the left dual  
$\,^*\hat{W}_1:=\Hom_S(_S\hat{W}_1,S)$ and, for $\frak{f},\frak{f}'\in \frak{F}$, the usual dual spaces $D(\frak{f}'\hat{W}_1\frak{f})=\Hom_k(\frak{f}'\hat{W}_1\frak{f},k)$. Then, 
for any $\frak{f},\frak{f}'\in \frak{F}$, we have the linear  isomorphism  
$$
D(\frak{f}'\hat{W}_1 \frak{f})=\Hom_k(\frak{f}'\hat{W}_1 \frak{f},k)\rightmap{\overline{\,\cdot\,}}\frak{f}\,\Hom_S(\hat{W}_1,S)\frak{f}'=
\frak{f}(\,^*\hat{W}_1)\frak{f}',$$
which maps each linear map $\xi\in D(\frak{f}'\hat{W}_1 \frak{f})$ onto  $\overline{\xi}\in \,^*\hat{W}_1$ defined by $\overline{\xi}(w)=\xi(\frak{f}'w\frak{f})\frak{f}'$, for $w\in \hat{W}_1$.  
\end{lemma}

\begin{proof} Given $\xi\in D(\frak{f}'\hat{W}_1\frak{f})$, $w\in \hat{W}_1$, and $s\in S$, we have 
$$\overline{\xi}(sw)=\xi(\frak{f}'sw\frak{f})\frak{f}'=\xi(s\frak{f}'w\frak{f})\frak{f}'=
s\xi(\frak{f}'w\frak{f})\frak{f}'=s\overline{\xi}(w), $$
where the third equality holds because $s\frak{f}'=\mu \frak{f}'$, for a scalar $\mu\in k$. So, indeed we have $\overline{\xi}\in \,^*\hat{W}_1$.  Since $(\frak{f}\overline{\xi}\frak{f}')(w)=\overline{\xi}(w\frak{f})\frak{f}'=\xi(\frak{f}'w\frak{f})\frak{f}'=\overline{\xi}(w)$, for all $w$, we see that, indeed,  $\overline{\xi}\in \frak{f}(\,^*\hat{W}_1)\frak{f}'$.

The linear map $\overline{ \,\cdot\,  }$ is bijective with inverse given by the composition of 
 the following chain of isomorphisms
$$\begin{matrix}\frak{f}(\,^*\hat{W}_1)\frak{f}'&=&\frak{f}\Hom_S(\hat{W}_1,S)\frak{f}'
\cong \Hom_S(\hat{W}_1\frak{f},S\frak{f}')\hfill\\ 
&\cong& \Hom_S(\frak{f}'\hat{W}_1\frak{f},S\frak{f}')\cong \Hom_k(\frak{f}'\hat{W}_1\frak{f},k).\\\end{matrix}$$
\end{proof}

\begin{remark}\label{R: base de hatW para descripcion de morfismos en B-mod} The notation given in the following will be kept throughout the whole paper. 
For any  $\frak{f}',\frak{f}\in \frak{F}$ consider a $k$-basis 
$\hueca{B}_{\frak{f}',\frak{f}}$ of   $\frak{f}'\hat{W}_1 \frak{f}$, and consider also the $k$-basis 
$\hueca{B} = \bigcup_{\frak{f}',\frak{f}\in \frak{F}} \hueca{B}_{\frak{f}',\frak{f}}$ of the vector space $\hat{W}_1$. 

For each $v\in \hueca{B}_{\frak{f}',\frak{f}}$ we consider the element $v^*\in D(\frak{f}'\hat{W}_1\frak{f})$ defined by $v^*(w)=0$, for all $w\in \hueca{B}_{\frak{f}',\frak{f}}\setminus\{v\}$, and $v^*(v)=1$. Thus, $\{v^*\mid v\in \hueca{B}_{\frak{f}',\frak{f}}\}$ is a $k$-basis for $D(\frak{f}'\hat{W}\frak{f})$ and, by the last lemma, $\{\overline{v^*}\mid v\in \hueca{B}_{\frak{f}',\frak{f}}\}$ is a $k$-basis for $\frak{f}(\,^*\hat{W}_1)\frak{f}'$.  Since, 
$\,^*\hat{W}_1=\bigoplus_{\frak{f}',\frak{f}\in \frak{F}}\frak{f}(\,^*\hat{W}_1)\frak{f}'$, we have the $k$-basis $\{\overline{v^*}\mid v\in \hueca{B}\}$ for $\,^*\hat{W}_1$.   

For each $v\in \hueca{B}_{\frak{f}',\frak{f}}$, the basis  element 
$\hat{v}:=\overline{v^*}\in  \,^*\hat{W}_1$  satisfies  
$\hat{v}(w)=0$ for all $w\in \hueca{B}\setminus{\{v\}} $ and $\hat{v}(v)=\frak{f}'$. Notice that $\hat{v}\in \frak{f}(^*\hat{W}_1)\frak{f}'$. 

 The $k$-basis $\{\hat{v}\mid v\in \hueca{B}\}$ 
 for $^*\hat{W}_1$, formed by directed elements, will be very useful in the descriptions of some  especial morphisms in ${\cal B}\g\Mod$. 
\end{remark}

\begin{remark}\label{R: el radical} For any additive $k$-category ${\cal C}$, 
we will denote by $\rad_{\cal C}$ \emph{the radical} of  ${\cal C}$, see for instance \cite[(2.10)]{Kr}. Thus, 
$f \in \Hom_{\cal C}(X,Y)$  belongs to the radical if and only if $id_Y - fg$ has a right inverse for every $g \in \Hom_{\cal C}(Y, X)$. Moreover, for each $X\in {\cal C}$, $\rad_{\cal C}(X,X)$ coincides with the Jacobson radical of the algebra $\End_{\cal C}(X)$.  We simply write $\rad_{\cal B}$ for the radical of the category ${\cal B}\g\Mod$. 
\end{remark}

\begin{lemma}\label{L: descripcion de rad(N,N')}
For any simple modules $N,N'\in {\cal B}\g\mod$ and $f=(f^0,f^1)$ with $f^0\in \Hom_R(N,N')$ and $f^1\in \Hom_{R\g R}(W_1,\Hom_k(N,N'))$, we have that   $f=(f^0,f^1)\in \rad_{\cal B}(N,N')$ iff $f^0=0$. 
\end{lemma}

\begin{proof}  Consider a pair $f=(0,f^1)\in \Hom_{\cal B}(N,N')$ and take any $g\in \Hom_{\cal B}(N',N)$. Then we have that $id_{N'}-fg=id_{N'}-(0,h^1)$, for some $(0,h^1)\in\End_{\cal B}(N')$. From \cite[(5.4)]{BSZ}, the morphism $(0,h^1)$ is nilpotent in the algebra $\End_{\cal B}(N')$. So,  $id_{N'}-fg$ is invertible, and $f=(0,f^1)\in \rad_{\cal B}(N,N')$. 
 
Reciprocally, assume that $f\in \rad_{\cal B}(N,N')$ with $f^0\not=0$. Since $N,N'$ are 
one-dimensional,  $f^0:N\rightmap{}N'$ is an $R$-linear isomorphism. Let $g^0:N'\rightmap{}N$ be the inverse of $f^0$, then $g=(g^0,0)\in \Hom_{\cal B}(N',N)$ is such that $id_{N'}- fg=(0,h^1)$, for some $h^1\in \Hom_{R\g R}(W_1,\Hom_k(N',N'))$. This  is not possible because  $id_{N'}-fg$ is a retraction. 
\end{proof}

\begin{lemma}\label{L: para construir morfismos entre simples} Assume $N,N'\in \{S_\lambda\mid \lambda\in D(h)\}\cup \{S_{\frak{f}}\mid \frak{f}\not=\frak{f}_0\}$. Then,  as vector spaces, we have $N=k\frak{f}$ and $N'=k\frak{f}'$, for some  $\frak{f},\frak{f}'\in \frak{F}$. 
We have  the following  linear isomorphism 
$$\theta:D(\frak{f}'\hat{W}_1\frak{f})\rightmap{}\Hom_{S\g S}(\hat{W}_1,\Hom_k(N,N')),$$
which maps each linear map $\xi\in D(\frak{f}'\hat{W}_1\frak{f})$ onto $\theta(\xi)$ defined by $\theta(\xi)(w)[\frak{f}]=\xi(\frak{f}'w\frak{f})\frak{f}'$, for $w\in \hat{W}_1$. 
\end{lemma}

\begin{proof} This is similar to the proof of 
(\ref{L: isos entre duales}). Given $\xi\in D(\frak{f}'\hat{W}_1\frak{f})$, $w\in \hat{W}_1$, and $s,s'\in S$, we have 
$$\begin{matrix}
\theta(\xi)(sws')[\frak{f}]
&=&
\xi(\frak{f}'sws'\frak{f})\frak{f}'=\xi(s\frak{f}'w\frak{f}s')\frak{f}'\hfill\\ &=&s\xi(\frak{f}'w\frak{f})s'\frak{f}'=
s\theta(\xi)(w)[s'\frak{f}]=(s\theta(\xi)(w)s')[\frak{f}]\hfill\\ 
\end{matrix}$$
where the third equality holds because $s\frak{f}'=\mu \frak{f}'$ and $\frak{f}s'=\mu'\frak{f}$, for some  scalars $\mu,\mu'\in k$. So, indeed $\theta$ is a well defined map.  
The linear map $\theta$ is bijective with inverse given by the composition of 
 the following chain of isomorphisms
 $$\begin{matrix}
\Hom_{S\g S}(\hat{W}_1,\Hom_k(N,N'))&\cong&
\Hom_{S\g S}(\hat{W}_1,\Hom_k(k\frak{f},k\frak{f}'))\hfill\\
&\cong&
\Hom_{S\frak{f}'\g S\frak{f}}(\frak{f}'\hat{W}_1\frak{f},\Hom_k(k\frak{f},k\frak{f}'))\hfill\\
&\cong&
\Hom_k(\frak{f}'\hat{W}_1\frak{f},k)=D(\frak{f}'\hat{W}_1\frak{f}).\hfill\\
\end{matrix}$$
 \end{proof}

\begin{lemma}\label{L: bases para rad entre simples} Whenever $\frak{f},\frak{f}'$ are non-marked idempotents  and  $\lambda\in D(h)$, we have the following:  
\begin{enumerate}
\item  For each $v\in \hueca{B}_{\frak{f}',\frak{f}}$, consider the morphism  $f_v:=(0,\hat{\psi}_v)\in \rad_{\cal B}(S_{\frak{f}},S_{\frak{f}'})$, where   $\hat{\psi}_v\in \Hom_{R\g R}(W_1,\Hom_k(S_{\frak{f}},S_{\frak{f}'}))$ extends the homomorphism $\psi_v\in \Hom_{S\g S}(\hat{W}_1,\Hom_k(S_{\frak{f}},S_{\frak{f}'}))$, given by  $\psi_v(w)[\frak{f}]=\hat{v}(w)\frak{f}'$, for $w\in \hat{W}_1$. Then,  
$$\{f_v\mid v\in \hueca{B}_{\frak{f}',\frak{f}}\}\hbox{ is a $k$-basis for }
\rad_{\cal B}(S_{\frak{f}},S_{\frak{f}'}).$$   
 \item  For each $v\in \hueca{B}_{\frak{f}_0,\frak{f}_0}$, consider the morphism  $f^\lambda_v:=(0,\hat{\psi}^\lambda_v)\in \rad_{\cal B}(S_\lambda,S_{\lambda})$, where   $\hat{\psi}^\lambda_v\in \Hom_{R\g R}(W_1,\Hom_k(S_\lambda,S_{\lambda}))$ extends the homomorphism  $\psi^\lambda_v\in \Hom_{S\g S}(\hat{W}_1,\Hom_k(S_\lambda,S_{\lambda}))$, given by  $\psi^\lambda_v(w)[\frak{f}_0]=\hat{v}(w)\frak{f}_0$, for $w\in \hat{W}_1$. Then,  
$$\{f^\lambda_v\mid v\in \hueca{B}_{\frak{f}_0,\frak{f}_0}\}\hbox{ is a $k$-basis for }\rad_{\cal B}(S_\lambda,S_{\lambda}).$$ 
\item For each $v\in \hueca{B}_{\frak{f},\frak{f}_0}$, consider the morphism  $f^\lambda_v:=(0,\hat{\psi}^\lambda_v)\in \rad_{\cal B}(S_\lambda,S_{\frak{f}})$, where   $\hat{\psi}^\lambda_v\in \Hom_{R\g R}(W_1,\Hom_k(S_\lambda,S_{\frak{f}}))$ extends the homomorphism  $\psi^\lambda_v\in \Hom_{S\g S}(\hat{W}_1,\Hom_k(S_\lambda,S_{\frak{f}}))$, given by  $\psi^\lambda_v(w)[\frak{f}_0]=\hat{v}(w)\frak{f}$, for $w\in \hat{W}_1$. Then,  
$$\{f^\lambda_v\mid v\in \hueca{B}_{\frak{f},\frak{f}_0}\}\hbox{ is a $k$-basis for }\rad_{\cal B}(S_\lambda,S_{\frak{f}}).$$ 
 \item For each $v\in \hueca{B}_{\frak{f}_0,\frak{f}}$, consider the morphism  $f^\lambda_v:=(0,\hat{\psi}^\lambda_v)\in \rad_{\cal B}(S_{\frak{f}},S_\lambda)$, where
 $\hat{\psi}^\lambda_v\in \Hom_{R\g R}(W_1,\Hom_k(S_{\frak{f}},S_\lambda))$ extends the homomorphism  $\psi^\lambda_v\in \Hom_{S\g S}(\hat{W}_1,\Hom_k(S_{\frak{f}},S_\lambda))$, given by  $\psi^\lambda_v(w)[\frak{f}]=\hat{v}(w)\frak{f}_0$, for $w\in \hat{W}_1$. Then,  
$$\{f^\lambda_v\mid v\in \hueca{B}_{\frak{f}_0,\frak{f}}\}\hbox{ is a $k$-basis for }
\rad_{\cal B}(S_{\frak{f}},S_\lambda).$$  
\end{enumerate}
\end{lemma}

\begin{proof}  In each case, we have to describe $\rad_{\cal B}(N,N')$ with $N$ and $N'$ simple ${\cal B}$-modules. Consider the idempotents $\frak{f}, \frak{f}'\in \frak{F}$ such that, as vector spaces,  $N=k\frak{f}$ and $N'=k\frak{f}'$, and take $v\in \hueca{B}_{\frak{f}',\frak{f}}$.  We have the isomorphisms:
$$\begin{matrix}
D(\frak{f}'\hat{W}_1\frak{f})
&\cong &
\Hom_{S\g S}(\hat{W}_1,\Hom_k(N,N'))\hfill\\
&\cong&
\Hom_{R\g R}(W_1,\Hom_k(N,N'))\hfill\\
&\cong&
\rad_{\cal B}(N,N'),
\hfill\\
\end{matrix}$$
where the first one is given in (\ref{L: para construir morfismos entre simples}), the second  is the extension, and the third one follows from (\ref{L: descripcion de rad(N,N')}). Now, we notice that the $k$-basis $\{v^*\mid v\in \hueca{B}_{\frak{f}',\frak{f}}\}$ of $D(\frak{f}'\hat{W}_1\frak{f})$ is mapped by this composition to the corresponding proposed basis of radical morphisms. In order to verify this, we  notice that, for $w\in \hat{W}_1$, we have  
$$\theta(v^*)(w)[\frak{f}]=v^*(\frak{f}'w\frak{f})\frak{f}'=\hat{v}(w)\frak{f}'.$$
\end{proof}

\begin{remark}[{\bf Choice of a special basis $\hueca{B}$}]\label{R: choice of special basis} Let us show that we can choose the basis $\hueca{B}$ in such a way that, for each $t\in [1,n]$, the linear isomorphism   
$$F_\vert:\rad_{\cal B}(S_{\frak{f}_{l_t}},S_{\frak{f}_{r_t}})\rightmap{}\rad_{{\cal P}^1(\Lambda)}(\hueca{L}_t,\hueca{R}_t),$$ 
induced by $F:\Hom_{\cal B}(S_{\frak{f}_{l_t}},S_{\frak{f}_{r_t}})\rightmap{}\Hom_{{\cal P}^1(\Lambda)}(\hueca{L}_t,\hueca{R}_t)$ 
is such that, with the notation of (\ref{L: bases para rad entre simples}), we have  $F(f_{u_t})=\gamma_t$, for a suitable $u_t\in \hueca{B}_{\frak{f}_{r_t},\frak{f}_{l_t}}$. Notice that $F_\vert$ is indeed an isomorphism because $F$ is full and faithful. 
As we know, $\hueca{L}_i\not\cong \hueca{L}_j$ and $\hueca{R}_i\not\cong \hueca{R}_j$, for any different $i,j\in [1,n]$. 

Recall that $\hueca{B}=\bigcup_{\frak{f}',\frak{f}\in \frak{F}}\hueca{B}_{\frak{f}',\frak{f}}$. We will only replace the sets  
$\hueca{B}_{\frak{f}_{r_t},\frak{f}_{l_t}}$, for $t\in [1,n]$. 
For a fixed $t\in [1,n]$, assume that $F(g_t)=\gamma_t$ with $g_t\in \rad_{\cal B}(S_{f_{l_t}},S_{f_{r_t}})$. 
The isomorphism $\phi:D(\frak{f}_{r_t}\hat{W}_1\frak{f}_{l_t})\rightmap{}\rad_{\cal B}(S_{\frak{f}_{l_t}},S_{\frak{f}_{r_t}})$ maps some 
$\xi_t\in D(\frak{f}_{r_t}\hat{W}_1\frak{f}_{l_t})$ onto $\phi(\xi_t)=g_t=(0,\hat{g}^1_t)$. So the morphism of $R$-$R$-bimodules $\hat{g}^1_t$ extends the morphism $g^1_t\in \Hom_{S\g S}(\hat{W}_1,\Hom_k(S_{\frak{f}_{l_t}},S_{\frak{f}_{r_t}}))$ such that 
$g^1_t(v)[\frak{f}_{l_t}]
=
\theta(\xi_t)(v)[\frak{f}_{l_t}]
=
\xi_t(\frak{f}_{r_t}v\frak{f}_{l_t})\frak{f}_{r_t}$, for $v\in \hat{W}_1$. Recall that $\gamma_t\not=0$, so $\xi_t\not=0$. Consider the vector space $K_t:=\Ker (\xi_t)\leq \frak{f}_{r_t}\hat{W}_1\frak{f}_{l_t}$ and an element $u_t\in \frak{f}_{r_t}\hat{W}_1\frak{f}_{l_t}$ such that $\xi_t(u_t)=1$.  Hence we can consider a $k$-basis 
$\hueca{B}'_{\frak{f}_{r_t},\frak{f}_{l_t}}$ for $\frak{f}_{r_t}\hat{W}_1\frak{f}_{l_t}$ formed by a $k$-basis of $K_t$ together with $u_t$. With this basis in mind, we have 
$u_t^*=\xi_t$.  
So we get $g^1_t(v)[\frak{f}_{l_t}]
=\hat{u}_t(v)\frak{f}_{r_t}$. Hence, replacing 
$\hueca{B}_{\frak{f}_{r_t},\frak{f}_{l_t}}$ by $\hueca{B}'_{\frak{f}_{r_t},\frak{f}_{l_t}}$, we get, with the notation of (\ref{L: bases para rad entre simples}) relative to the new basis $\hueca{B}'$, that $F(f_{u_t})=\gamma_t$. 

We make such replacements for each $t\in [1,n]$ and we obtain what we wanted. From now on, we work with this new basis, which we denote again by $\hueca{B}$. 
\end{remark}

\begin{lemma}\label{L: radEnd(k(x)f0}
We have the following vector space decomposition of the radical of the endomorphism algebra $\End_{\cal B}(k(x)\frak{f}_0)$:
$$\rad\End_{\cal B}(k(x)\frak{f}_0)=\bigoplus_{v\in \hueca{B}_{\frak{f}_0,\frak{f}_0}}{\cal H}_v$$
where ${\cal H}_v$ is the space of morphisms $f_{v,q}=(0,\hat{\psi}_{v,q})$ with $q\in \End_k(k(x)\frak{f}_0)$ and $\psi_{v,q}(w)[z]=\hat{v}(w)q(z)$, for $w\in \hat{W}_1$ and $z\in k(x)\frak{f}_0$. 
\end{lemma}

\begin{proof} Define $E:=\End_{\cal B}(k(x)\frak{f}_0)$  and  recall that any morphism $f\in E$ is a sum $f=(f^0,f^1)=(f^0,0)+(0,f^1)$ of morphisms in $E$. Here $f^0\in \End_R(k(x)\frak{f}_0)\cong k(x)$ and $(0,f^1)$ is nilpotent. 
Thus non-invertible elements $f$ of $E$ are such that $f^0=0$.  Thus the elements in $\rad E$ are those of the form $(0,\hat{\psi})$, where $\hat{\psi}$ is the extension of some $\psi\in \Hom_{S\g S}(\hat{W}_1,\Hom_k(k(x)\frak{f}_0,k(x)\frak{f}_0))$.  

For each linear transformation $q$, we have a linear map 
$$\theta_q:D(\frak{f}_0\hat{W}_1\frak{f}_0)\rightmap{}\Hom_{S\g S}(\hat{W}_1,\Hom_k(k(x)\frak{f}_0,k(x)\frak{f}_0)),$$
  such that $\theta_q(\xi)(w)[z]=\xi(\frak{f}_0w\frak{f}_0)q(z)$, for $\xi\in D(\frak{f}_0\hat{W}_1\frak{f}_0)$, $w\in \hat{W}_1$ and $z\in k(x)\frak{f}_0$. For $v\in \hueca{B}_{\frak{f}_0,\frak{f}_0}$, 
  we have $\theta_q(v^*)(w)[z]=\hat{v}(w)q(z)$, so $\theta_q(v^*)=\psi_{v,q}$. 
Thus, as remarked at the beginning of this proof, the morphisms $f_{v,q}$ belong to $\rad E$. If $q=\lambda q_1+q_2$, then $f_{v,q}=\lambda f_{v,q_1}+f_{v,q_2}$, so ${\cal H}_v$ is indeed a vector subspace. 

Now, assume that $\sum_{i=1}^mf_{v_i,q_i}=0$, for some different elements $v_1,\ldots,v_m\in \hueca{B}_{\frak{f}_0,\frak{f}_0}$ and some non-zero linear maps $q_1,\ldots,q_m$. If $q_1(z_1)\not=0$,  evaluating the second components of the morphisms $f_{v_i,q_i}$ in $v_1$, we obtain 
$0\not=q_1(z_1)=\hat{v}_1(v_1)q_1(z_1)=\psi_{v_1,q_1}(v_1)[z_1]$, but
 $\sum_{i=2}^m\psi_{v_i,q_i}(v_1)[z_1]
=\sum_{i=2}^m\hat{v}_i(v_1)q_i(z_1)=0$, a contradiction.  

It remains to show that each $f=(0,\psi)\in \rad E$ is a finite sum of elements of the form  $f_{v,q}$. This is the case because we have 
$$\psi=\sum_{v\in \hueca{B}_{\frak{f}_0,\frak{f}_0}}f_{v,q_v},$$
where $q_v:=\psi(v)$, for $v\in \hueca{B}_{\frak{f}_0,\frak{f}_0}$. 
Indeed, notice that, for $w\in \hat{W}_1$, we have $\psi(w)=\psi(\frak{f}_0w\frak{f}_0)$. So, $\psi$  is determined by the values $\psi(w)$, with $w\in \hueca{B}_{\frak{f}_0,\frak{f}_0}$. For $w\in \hueca{B}_{\frak{f}_0,\frak{f}_0}$ and $z\in k(x)\frak{f}_0$, we have 
$$\sum_v\psi_{v,q_v}(w)[z]=\sum_v\hat{v}(w)q_v(z)=q_w(z)=\psi(w)[z].$$
From this, we get the wanted equality.
\end{proof}

\begin{lemma}\label{L: decompos of HomB(gen,sumadesimples}
We have the following vector space decompositions:
\begin{enumerate}
\item For any non-marked idempotent $\frak{f}$ and $V\in k\g\Mod$, we have 
$$\Hom_{\cal B}
(k(x)\frak{f}_0,S_{\frak{f}}\otimes_kV)
=
\bigoplus_{v\in \hueca{B}_{\frak{f},\frak{f}_0}}{\cal H}_v^a$$
where ${\cal H}^a_v$ is the vector subspace formed by the morphisms $f_{v,q}=(0,\psi_{v,q})$, for some $q\in \Hom_k(k(x)\frak{f}_0,S_{\frak{f}}\otimes_kV)$, with $\psi_{v,q}(w)[z]=\hat{v}(w)q(z)$. 
\item For any non-marked idempotent $\frak{f}$ and $V\in k\g\Mod$, we have 
$$\Hom_{\cal B}
(S_{\frak{f}}\otimes_kV,k(x)\frak{f}_0)
=
\bigoplus_{v\in \hueca{B}_{\frak{f}_0},\frak{f}}{\cal H}_v^b$$
where ${\cal H}^b_v$ is the vector subspace formed by the morphisms $f_{v,q}=(0,\psi_{v,q})$, for some $q\in \Hom_k(S_{\frak{f}}\otimes_kV,k(x)\frak{f}_0)$, with $\psi_{v,q}(w)[z]=\hat{v}(w)q(z)$. 
\end{enumerate}
\end{lemma}

\begin{proof} (1): Since $\Hom_R(k(x)\frak{f}_0,S_{\frak{f}}\otimes_kV)=0$, the first component of the morphisms in $\Hom_{\cal B}
(k(x)\frak{f}_0,S_{\frak{f}}\otimes_kV)$ is zero. 
 Thus the elements in this space are those of the form $(0,\hat{\psi})$, where $\hat{\psi}$ is the extension of some morphism $\psi\in \Hom_{S\g S}(\hat{W}_1,\Hom_k(k(x)\frak{f}_0,S_{\frak{f}}\otimes_kV))$.  

For each linear map $q\in \Hom_k(k(x)\frak{f}_0,S_{\frak{f}}\otimes_kV))$, we have a linear map 
$$\theta_q:D(\frak{f}\hat{W}_1\frak{f}_0)\rightmap{}\Hom_{S\g S}(\hat{W}_1,\Hom_k(k(x)\frak{f}_0,S_{\frak{f}}\otimes_kV)),$$
  such that $\theta_q(\xi)(w)[z]=\xi(\frak{f}w\frak{f}_0)q(z)$, for $\xi\in D(\frak{f}\hat{W}_1\frak{f}_0)$, $w\in \hat{W}_1$, and $z\in k(x)\frak{f}_0$. 
  We have $\theta_q(v^*)(w)[z]=\hat{v}(w)q(z)$, so $\theta_q(v^*)=\psi_{v,q}$. 
Thus, as remarked at the beginning of this section,  $f_{v,q}\in \Hom_{\cal B}(k(x)\frak{f}_0,S_{\frak{f}}\otimes_kV)$. The rest of the proof is similar to the proof of (\ref{L: radEnd(k(x)f0}). 

The proof of (2) is similar. 
\end{proof}

\section{Factoring morphisms of ${\cal B}$-modules}

In the following, we keep the notation of the last section. We will denote by $\delta:T_R(W)\rightmap{}T_R(W)$ the differential 
of the minimal ditalgebra ${\cal B}$.

\begin{remark}\label{R: compos de segundas comp en B-Mod} Assume that we have composable morphisms in ${\cal B}\g\Mod$ of the form 
$(0,\psi_1):N_1\rightmap{}N_2$ and $(0,\psi_2):N_2\rightmap{}N_3$.
That is 
$$\psi_1\in \Hom_{R\g R}(W_1,\Hom_k(N_1,N_2)) \hbox{ \ and \ } \psi_2\in \Hom_{R\g R}(W_1,\Hom_k(N_2,N_3)).$$
By definition of the composition in ${\cal B}\g\Mod$, we see that $(0,\psi_2)(0,\psi_1)=(0,\psi)$, with $\psi\in \Hom_{R\g R}(W_1,\Hom_k(N_1,N_3))$  such that 
$\psi(v)=\pi (\psi_2\otimes \psi_1)(\delta(v))$, for any $v\in W_1$, where $\pi$ denotes the composition in  
$$W_1\otimes_RW_1\rightmap{ \ \psi_2\otimes\psi_1 \ }\Hom_k(N_2,N_3)\otimes_R\Hom_k(N_1,N_2)\rightmap{\pi}\Hom_k(N_1,N_3).$$
Similarly, if we have another morphism of the form $(0,\psi_3): N_3\rightmap{}N_4$, then the triple  composition has the form 
$(0,\psi_3)(0,\psi_2)(0,\psi_1)=(0,\psi')$, where 
$\psi'(v)=\pi(\psi_3\otimes\psi_2\otimes\psi_1)(\delta\otimes 1)\delta(v)$, for $v\in \hueca{B}$, where 
$$W_1\otimes_RW_1\otimes_RW_1\rightmap{ \  \psi_3\otimes  \psi_2\otimes\psi_1 \ }\Hom_k(N_3,N_4)\otimes_R\Hom_k(N_2,N_3)\otimes_R\Hom_k(N_1,N_2)$$
and $\pi$ is the composition map to $\Hom_k(N_1,N_4)$, see \cite[(2.3)]{BSZ}. 

In order to work with the compositions of some special morphisms in ${\cal B}\g\Mod$, we have to deal with the differential $\delta$ of ${\cal B}$. We will need descriptions for the values $\delta(v)$ and $(1\otimes \delta)\delta(v)=(\delta\otimes 1)\delta(v)$, for some $v\in \hueca{B}$. 
\end{remark}

\begin{remark}\label{R: descriptions of delta and}
 We have $R_0=R\frak{f}_0=k[x]_{h(x)}\frak{f}_0$ and then $R_0\otimes_kR_0\cong k[x,y]_{h(x)h(y)}$. So the $R_0\g R_0$-bimodule $\frak{f}_0W_1\frak{f}_0$ is a $k[x,y]_{h(x)h(y)}$-module.  

We will consider the $R_0\g R_0$-bimodules 
$$\frak{f}_0W_1\otimes_R W_1\frak{f}_0=[\bigoplus_{z\in \frak{Z}}\frak{f}_0W_1\frak{f}_z\otimes_R\frak{f}_zW_1\frak{f}_0]\oplus H, $$
where  $H=\frak{f}_0W_1\frak{f}_0\otimes_{R}\frak{f}_0W_1\frak{f}_0$,
 and 
$$\frak{f}_0W_1\otimes_RW_1\otimes_RW_1\frak{f}_0=
[\bigoplus_{z,z'\in \frak{Z}}\frak{f}_0W_1\frak{f}_{z'}\otimes_R\frak{f}_{z'}W_1\frak{f}_z\otimes_R \frak{f}_zW_1\frak{f}_0 ]\oplus H',$$
where $H'$ is the direct sum of the bimodules 
$
\frak{f}_0W_1\frak{f}_0\otimes_{R}\frak{f}_0W_1\frak{f}_0\otimes_{R}\frak{f}_0W_1\frak{f}_0$,  
$$\bigoplus_{z\in \frak{Z}}\frak{f}_0W_1\frak{f}_0\otimes_{R}\frak{f}_0W_1\frak{f}_z\otimes_R\frak{f}_zW_1\frak{f}_0,  
\hbox{ and }
\bigoplus_{z\in \frak{Z}}\frak{f}_0W_1\frak{f}_z\otimes_R\frak{f}_zW_1\frak{f}_0\otimes_{R}\frak{f}_0W_1\frak{f}_0.
$$ 
Given $v_n,\ldots,v_1\in \hueca{B}$, with $v_1\in \hueca{B}_{\frak{f}_1,\frak{f}_0}, v_2\in \hueca{B}_{\frak{f}_2,\frak{f}_1},\ldots, v_n\in \hueca{B}_{\frak{f}_0,\frak{f}_{n-1}}$, 
an element $\tau\in R\frak{f}_0(v_n\otimes\cdots\otimes v_1)R\frak{f}_0$ can be written as $$\tau=c_{v_1,\ldots,v_n}(x,y)(v_n\otimes\cdots\otimes v_1),$$
where  $c_{v_1,\ldots,v_n}(x,y)\in k[x,y]_{h(x)h(y)}$. Each polynomial $c(x,y)\in k[x,y]$, has the form $c(x,y)=\sum_{i=0}^ma_i(y)x^i$, with $a_i(y)\in k[y]$, and for $v_1\in \hueca{B}_{\frak{f},\frak{f}_0}$ and $v_2\in \hueca{B}_{\frak{f}_0,\frak{f}}$,  we have 
$$c(x,y)(v_2\otimes v_1)=\sum_{i=0}^ma_i(x)v_2\otimes v_1 x^i.$$
For $v_1\in \hueca{B}_{\frak{f},\frak{f}_0}$, $v_2\in \hueca{B}_{\frak{f}',\frak{f}}$, and $v_3\in \hueca{B}_{\frak{f}_0,\frak{f}'}$, we have 
$$c(x,y)(v_3\otimes v_2\otimes v_1)=\sum_{i=0}^ma_i(x)v_3\otimes v_2\otimes v_1 x^i.$$

Let us define  
$\hueca{B}^a:=\bigcup_{z\in \frak{Z}}
\hueca{B}_{\frak{f}_z,\frak{f}_0}$ and 
$\hueca{B}^b:=\bigcup_{z\in \frak{Z}}
\hueca{B}_{\frak{f}_0,\frak{f}_z}$. Then, notice that  replacing  each $v \in \hueca{B}_{\frak{f}_0,\frak{f}_0}$ by $h(x)^mvh(x)^m$, for a suitable $m\in \hueca{N}$, we may assume that we have  for each $v\in \hueca{B}_{\frak{f}_0,\frak{f}_0}$ polynomials $c^v_{v_2,v_1}(x,y), c^v_{v_3,v_2,v_1}(x,y)\in k[x,y]$ such that
$$\delta(v)=\sum_{v_2\in \hueca{B}^b, v_1\in \hueca{B}^a}c^v_{v_2,v_1}(x,y)(v_2\otimes v_1)+r(v),$$
where $r(v)\in H$, and, if we set $\hueca{B}_{\frak{Z}}:=\bigcup_{z,z'\in \frak{Z}}\hueca{B}_{\frak{f}_{z'},\frak{f}_z}$,  
$$(1\otimes \delta)\delta(v)=(\delta\otimes 1)\delta(v)
=
\sum_{v_3\in\hueca{B}^b, 
v_2\in \hueca{B}_{\frak{Z}}, 
v_1\in \hueca{B}^a }
c^v_{v_3,v_2,v_1}(x,y)(v_3\otimes v_2\otimes v_1)+r'(v) , 
$$
where $r'(v)\in H'$. We will keep this notation from now on. 
\end{remark}

\begin{lemma}\label{L: de las lambda-compos}
Given $\lambda\in D(h)$ and $\frak{f},\frak{f}'\not=\frak{f}_0$,  we have:
\begin{enumerate}
\item For $v_1\in \hueca{B}_{\frak{f},\frak{f}_0}$ and $v_2\in \hueca{B}_{\frak{f}_0,\frak{f}}$, we have 
$$f_{v_2}^\lambda f_{v_1}^\lambda=(0,\psi_{v_2,v_1}^\lambda), \hbox{ with } \psi_{v_2,v_1}^\lambda(v)=c^v_{v_2,v_1}(\lambda,\lambda)id_{S_\lambda}, \hbox{ for }v\in \hueca{B}_{\frak{f}_0,\frak{f}_0}.$$
\item For $v_1\in \hueca{B}_{\frak{f},\frak{f}_0}$, 
 $v_2\in \hueca{B}_{\frak{f}',\frak{f}}$, and $v_3\in \hueca{B}_{\frak{f}_0,\frak{f}'}$ we have 
$$f^\lambda_{v_3} f_{v_2} f_{v_1}^\lambda=(0,\psi_{v_3,v_2,v_1}^\lambda), \hbox{ with } \psi_{v_3,v_2,v_1}^\lambda(v)=c^v_{v_3,v_2,v_1}(\lambda,\lambda)id_{S_\lambda}, \hbox{ for }v\in \hueca{B}_{\frak{f}_0,\frak{f}_0}.$$
\end{enumerate}
\end{lemma}

\begin{proof} (1): We have
$$\psi^\lambda_{v_1}\in \Hom_{R\g R}(W_1,\Hom_k(S_\lambda,S_{\frak{f}})) \hbox{ \ and \ } \psi^\lambda_{v_2}\in \Hom_{R\g R}(W_1,\Hom_k(S_{\frak{f}},S_\lambda)).$$
We claim that, for  $w_1\in \hueca{B}^a$, $w_2\in \hueca{B}^b$, and  $c(x,y)\in k[x,y]$, we have 
$$\pi (\psi^\lambda_{v_2}\otimes\psi^\lambda_{v_1})[c(x,y)(w_2\otimes w_1)]=c(\lambda,\lambda)\hat{v}_2(w_2)\hat{v}_1(w_1)id_{S_\lambda}.$$  
Indeed, if $c(x,y)=\sum_{i=0}^ma_i(y)x^i$, we have
$$\begin{matrix}
\pi (\psi^\lambda_{v_2}\otimes\psi^\lambda_{v_1})[c(x,y)(w_2\otimes w_1)](\frak{f}_0)
&=&
\sum_{i=0}^m\pi(\psi^\lambda_{v_2}[a_i(x)w_2]\otimes \psi^\lambda_{v_1}[w_1x^i])(\frak{f}_0)\hfill\\
&=&
\sum_{i=0}^ma_i(\lambda)\lambda^i\hat{v}_2(w_2)\hat{v}_1(w_1)\frak{f}_0\hfill\\
&=&
c(\lambda,\lambda)\hat{v}_2(w_2)\hat{v}_1(w_1)\frak{f}_0.\hfill\\
\end{matrix}$$
Then, as remarked in (\ref{R: compos de segundas comp en B-Mod}), with the description of $\delta(v)$ given  in (\ref{R: descriptions of delta and}), we have 
$$\begin{matrix} 
\psi^\lambda_{v_2,v_1}(v)&=&\pi (\psi^\lambda_{v_2}\otimes\psi^\lambda_{v_1})(\delta(v))\hfill\\
&=&
\sum_{w_1\in \hueca{B}^a,w_2\in \hueca{B}^b}\pi (\psi^\lambda_{v_2}\otimes\psi^\lambda_{v_1})(c^v_{w_2,w_1}(x,y)(w_2\otimes w_1))\hfill\\
&&+\,\pi (\psi^\lambda_{v_2}\otimes\psi^\lambda_{v_1})(r(v)).\hfill\\
\end{matrix}$$
Since $r(v)\in H =\frak{f}_0W_1\frak{f}_0\otimes_{R}\frak{f}_0W_1\frak{f}_0$, it is a sum of tensor products of the form $d_3w_2\otimes d_2 w_1d_1$, where $w_1,w_2\in \hueca{B}_{\frak{f}_0,\frak{f}_0}$ and $d_1,d_2,d_3\in R_0$. This expression in terms of  basis  elements in $\hueca{B}$ does not contain the elements $v_1$ and $v_2$, thus  $(\psi^\lambda_{v_2}\otimes\psi^\lambda_{v_1})(r(v))=0$.  Then (1) follows from our claim and the definition of the maps $\psi_{v_1}^\lambda$ and $\psi_{v_2}^\lambda$. 
\medskip

\noindent(2): This is similar to the preceding one. We have 
$$\psi^\lambda_{v_1}\in 
\Hom_{R\g R}(W_1,\Hom_k(S_\lambda,S_{\frak{f}})), 
\hbox{ \  \ } \psi_{v_2}\in 
\Hom_{R\g R}(W_1,\Hom_k(S_{\frak{f}},S_{\frak{f}'})),$$
and 
$\psi^\lambda_{v_3}\in 
\Hom_{R\g R}(W,\Hom_k(S_{\frak{f}'},S_\lambda))$. 
Now, we claim that, for any $w_1\in \hueca{B}^a$, $w_2\in \hueca{B}_{\frak{Z}}$, $w_3\in \hueca{B}^b$, and   $c(x,y)\in k[x,y]$, we have 
$$\pi (\psi^\lambda_{v_3}\otimes\psi_{v_2}\otimes\psi^\lambda_{v_1})[c(x,y)(w_3\otimes w_2\otimes w_1)]=c(\lambda,\lambda)\hat{v}_3(w_3)\hat{v}_2(w_2)\hat{v}_1(w_1)id_{S_\lambda}.$$ 
This is easily verified, as before.   
Now, as remarked in (\ref{R: compos de segundas comp en B-Mod}), with the description of $(1\otimes\delta)\delta(v)$ given  in (\ref{R: descriptions of delta and}), we have 

\vbox{$$
\psi^\lambda_{v_3,v_2,v_1}(v)
=\pi (\psi^\lambda_{v_3}\otimes\psi_{v_2}\otimes\psi^\lambda_{v_1})((1\otimes\delta)\delta(v))=$$
$$
\sum_{w_3\in \hueca{B}^b,w_2\in \hueca{B}_{\frak{Z}},w_1\in \hueca{B}^a}\pi (\psi^\lambda_{v_3}\otimes\psi_{v_2}\otimes\psi^\lambda_{v_1})(c^v_{w_3,w_2,w_1}(x,y)(w_3\otimes w_2\otimes w_1))$$
$$+\,\,\pi (\psi^\lambda_{v_3}\otimes\psi_{v_2}\otimes\psi^\lambda_{v_1})(r'(v)).$$}
Here, $r'(v)\in H'=\frak{f}_0W_1\frak{f}_0\otimes_R W_1\otimes_RW_1\frak{f}_0+
\frak{f}_0W_1\otimes_RW_1\otimes_R\frak{f}_0W_1\frak{f}_0$, so its expression as a combination of  tensor products involving  basis  elements in $\hueca{B}$ does not contain elements of the form $x^iv_3\otimes v_2\otimes v_1x^j$, so the last term in the preceding equality is zero.  Now, (2) follows from the preceding claim and the definition of the maps $\psi^\lambda_{v_3}$, $\psi_{v_2}$, and $\psi^\lambda_{v_1}$.  
\end{proof}

\begin{proposition}\label{P: facts de un endomorf radical de Slambda}
Take $\lambda\in D(h)$ and $f=(0,\psi)\in \rad\End_{\cal B}(S_\lambda)$. Then, $f$ is a $k$-linear combination of morphisms that factor through some morphism $f_{u_t}$ or some object $S_{\frak{f}_{z_t}}$ with $t\in [1,n]$ iff 
$$\psi=\sum_{(v_2,t,v_1)\in \hueca{I}_3}
d_{v_2,t,v_1}\psi^\lambda_{v_2,u_t,v_1}
+
\sum_{(v_2,v_1)\in \hueca{I}_2}d_{v_2,v_1}\psi^\lambda_{v_2,v_1}, $$
where $$\hueca{I}_3=\{(v_2,t,v_1)\mid v_1\in \hueca{B}_{\frak{f}_{l_t},\frak{f}_0}, v_2\in \hueca{B}_{\frak{f}_0,\frak{f}_{r_t}},  t\in [1,n] \}$$ and 
$$\hueca{I}_2=\{(v_2,v_1)\mid 
v_1\in\hueca{B}_{\frak{f}_{z_t},\frak{f}_0}, v_2\in \hueca{B}_{\frak{f}_0,\frak{f}_{z_t}}, t\in [1,n]\},$$
for some scalars $d_{v_2,t,v_1}$ and $d_{v_2,v_1}$ in $k$.  
\end{proposition}

\begin{proof} From (\ref{L: bases para rad entre simples}), we have $k$-basis
$$\begin{matrix}
\{f^\lambda_{v_1}\}_{v_1\in \hueca{B}_{\frak{f}_{l_t},\frak{f}_0}} &\hbox{ for }& \rad_{\cal B}(S_\lambda,S_{\frak{f}_{l_t}}),\hfill\\
\{f^\lambda_{v_2}\}_{v_2\in \hueca{B}_{\frak{f}_0,\frak{f}_{r_t}}} &\hbox{ for }& \rad_{\cal B}(S_{\frak{f}_{r_t}},S_\lambda),\hfill\\
\{f^\lambda_{v_1}\}_{v_1\in \hueca{B}_{\frak{f}_{z_t},\frak{f}_0}} &\hbox{ for }& \rad_{\cal B}(S_\lambda,S_{\frak{f}_{z_t}}),\hfill\\
\{f^\lambda_{v_2}\}_{v_2\in \hueca{B}_{\frak{f}_0,\frak{f}_{z_t}}} &\hbox{ for }& \rad_{\cal B}(S_{\frak{f}_{z_t}},S_\lambda). \hfill\\
\end{matrix}$$
Thus, $f$ is a linear combination as described in the hypothesis iff there are scalars $d_{v_2,t,v_1}$ and $d_{v_2,v_1}$ in $k$ such that 
$$f=\sum_{(v_2,t,v_1)\in \hueca{I}_3}
d_{v_2,t,v_1}f^\lambda_{v_2}f_{u_t}f^\lambda_{v_1}
+
\sum_{(v_2,v_1)\in \hueca{I}_2}
d_{v_2,v_1}f^\lambda_{v_2}f^\lambda_{v_1}. $$
This is equivalent, by (\ref{L: de las lambda-compos}), to the equality
$$(0,\psi)=\sum_{(v_2,t,v_1)\in \hueca{I}_3}
d_{v_2,t,v_1}(0,\psi^\lambda_{v_2,u_t,v_1})
+
\sum_{(v_2,v_1)\in \hueca{I}_2}d_{v_2,v_1}(0,\psi^\lambda_{v_2,v_1}).$$
\end{proof}

\begin{proposition}\label{P: facts de endomorfs radicales de Slambda}
Take $\lambda\in D(h)$. Then, every morphism   $f=(0,\psi)\in \rad\End_{\cal B}(S_\lambda)$ is a linear combination of morphisms that factor through some morphism $f_{u_t}$ or some  object $S_{\frak{f}_{z_t}}$ with $t\in [1,n]$ iff for each $w\in \hueca{B}_{\frak{f}_0,\frak{f}_0}$ we have 
$$\psi^{\lambda}_w=\sum_{(v_2,t,v_1)\in \hueca{I}_3}
d^w_{v_2,t,v_1}\psi^\lambda_{v_2,u_t,v_1}
+
\sum_{(v_2,v_1)\in \hueca{I}_2}d^w_{v_2,v_1}\psi^\lambda_{v_2,v_1}, $$
where $\hueca{I}_3$ and $\hueca{I}_2$ are the index sets defined in the previous proposition, 
for some scalars $d^w_{v_2,t,v_1}$ and $d^w_{v_2,v_1}$ in $k$. Moreover, the last condition is equivalent to the existence of scalars  $d^w_{v_2,t,v_1}$ and $d^w_{v_2,v_1}$ such that, for all $v,w\in \hueca{B}_{\frak{f}_0,\frak{f}_0}$, the following equation is  satisfied 
\begin{equation}\tag{$E(\lambda)_{v,w}$}
\sum_{(v_2,t,v_1)\in \hueca{I}_3}
d^w_{v_2,t,v_1}c^v_{v_2,u_t,v_1}(\lambda,\lambda)
+
\sum_{(v_2,v_1)\in \hueca{I}_2}
d^w_{v_2,v_1}c^v_{v_2,v_1}(\lambda,\lambda)=\delta_{v,w}
\end{equation}
where $\delta_{v,w}$ denotes the Kronecker  delta map.
\end{proposition}

\begin{proof} From (\ref{L: bases para rad entre simples}), we know that 
$\{f_w^\lambda\mid w\in \hueca{B}_{\frak{f}_0,\frak{f}_0}\}$ is a $k$-basis for $\rad_{\cal B}(S_\lambda,S_\lambda)$. Thus, every morphism $f\in \rad_{\cal B}(S_\lambda,S_\lambda)$ is a linear combination as in the first statement of our proposition  iff every morphism $f_w^\lambda$ with $w\in \hueca{B}_{\frak{f}_0,\frak{f}_0}$ is a linear combination of the same type. From (\ref{P: facts de un endomorf radical de Slambda}),  this occurs  if and only if there are scalars  $d^w_{v_2,t,v_1}$ and $d^w_{v_2,v_1}$ such that 
$$\psi^{\lambda}_w=\sum_{(v_2,t,v_1)\in \hueca{I}_3}
d^w_{v_2,t,v_1}\psi^\lambda_{v_2,u_t,v_1}
+
\sum_{(v_2,v_1)\in \hueca{I}_2}d^w_{v_2,v_1}\psi^\lambda_{v_2,v_1}.$$
From (\ref{L: de las lambda-compos}), we get that this last equation holds iff $E(\lambda)_{v,w}$ holds 
for every $v\in \hueca{B}_{\frak{f}_0,\frak{f}_0}$. Indeed, we have 
$$\psi^\lambda_w(v)=\hat{w}(v)id_{S_\lambda}=\delta_{v,w}id_{S_\lambda}, \hbox{ and }$$
\vbox{$$\sum_{(v_2,t,v_1)\in \hueca{I}_3}
d^w_{v_2,t,v_1}\psi^\lambda_{v_2,u_t,v_1}(v)
+
\sum_{(v_2,v_1)\in \hueca{I}_2}d^w_{v_2,v_1}\psi^\lambda_{v_2,v_1}(v)\hbox{\hskip4cm}$$
$$\hbox{\hskip1cm}=\sum_{(v_2,t,v_1)\in \hueca{I}_3}
d^w_{v_2,t,v_1}c^v_{v_2,u_t,v_1}(\lambda,\lambda)id_{S_\lambda}
+
\sum_{(v_2,v_1)\in \hueca{I}_2}
d^w_{v_2,v_1}c^v_{v_2,v_1}(\lambda,\lambda)id_{S_\lambda}.$$}
\end{proof}

\begin{corollary}\label{C: caract de que End(Mlambda)=k casi siempre} Given $\lambda\in D(h)$, we have that $\End_\Lambda(M(\lambda))=k$ iff ${\cal B}$ satisfies equations of type  $E(\lambda)_{v,w}$, for all $v,w\in \hueca{B}_{\frak{f}_0,\frak{f}_0}$. 
\end{corollary}

\begin{proof} Since $k$ is algebraically closed, $\End_\Lambda(M(\lambda))=k\oplus \rad\End_\Lambda(M(\lambda))$. If we define  $X_\lambda:=F(S_\lambda)$, we know that $\Coker(X_\lambda)=\Coker (F(S_\lambda))\cong M\otimes_\Gamma S_\lambda=M(\lambda)$. Since $\Coker:{\cal P}^1(\Lambda)\rightmap{}\Lambda\g\Mod$ is a full functor and $M(\lambda)$ is indecomposable, it induces a surjection 
$$\Coker: \rad_{{\cal P}^1(\Lambda)}(X_\lambda,X_\lambda)\rightmap{}\rad_\Lambda(M(\lambda),M(\lambda)).$$

Then, from (\ref{P: facts de endomorfs radicales de Slambda}), ${\cal B}$ satisfies equations of type  $E(\lambda)_{v,w}$, for all $v,w\in \hueca{B}_{f_0,f_0}$, iff every morphism in $\rad_{\cal B}(S_\lambda,S_\lambda)$ is a $k$-linear combination of 
morphisms that factor through some morphism $f_{u_t}$ or some object $S_{\frak{f}_{z_t}}$  with $t\in [1,n]$. Since $F$ is full and faithful, $F(f_{u_t})=\gamma_t$, and $F(S_{\frak{f}_{z_t}})=\hueca{S}_t$, this is equivalent to the fact that every morphism 
$u\in \rad_{{\cal P}^1(\Lambda)}(X_\lambda,X_\lambda)$ is a linear combination of 
morphisms that  factor through some morphism $\gamma_t$ or some object $\hueca{S}_t$  with  $t\in [1,n]$.

Then, from (\ref{P: caract de u con cokeru=0}) and (\ref{R: vale Prop para las gammas(t)}), if ${\cal B}$ satisfies equations of type  $E(\lambda)_{v,w}$, for all $v,w\in \hueca{B}_{\frak{f}_0,\frak{f}_0}$, we get that every morphism  $g\in\rad_\Lambda(M(\lambda),M(\lambda))$ is of the form $g=\Coker(u)=0$. Conversely, if 
$\rad_\Lambda(M(\lambda),M(\lambda))=0$, we get that every $u\in \rad_{{\cal P}^1(\Lambda)}(X_\lambda,X_\lambda)$ satisfies that $\Coker(u)=0$. Thus, from (\ref{P: caract de u con cokeru=0}) and (\ref{R: vale Prop para las gammas(t)}), $u$ is a linear combination of morphisms  that factor through some morphism $\gamma_t$  or some object $\hueca{S}_t$, for some $t\in [1,n]$. So ${\cal B}$ satisfies the specified equations. 
\end{proof}

\begin{proposition}\label{P: la eq E(x)}
For infinitely many $\lambda\in D(h)$ we have 
$\End_\Lambda(M(\lambda))=k$ iff there are elements  $d^w_{v_2,t,v_1}(x)$ and $d^w_{v_2,v_1}(x)$ in $k(x)$ such that for every $v,w\in \hueca{B}_{\frak{f}_0,\frak{f}_0}$, the following equation holds
 
\begin{equation}\tag{$E(x)_{v,w}$}
\sum_{(v_2,t,v_1)\in \hueca{I}_3}
d^w_{v_2,t,v_1}(x)c^v_{v_2,u_t,v_1}(x,x)
+
\sum_{(v_2,v_1)\in \hueca{I}_2}
d^w_{v_2,v_1}(x)c^v_{v_2,v_1}(x,x)=\delta_{v,w}
\end{equation}
\end{proposition}

\begin{proof} Consider the cardinalities  $c_0:=\vert \hueca{B}_{\frak{f}_0,\frak{f}_0}\vert $ and $c_1:=\vert \hueca{I}_3\cup \hueca{I}_2\vert$. Then, we have the matrix 
$$C(x):=\begin{pmatrix}
\begin{matrix} c^v_{v_2,u_t,v_1}(x,x)
\end{matrix}\\
-----\\
\begin{matrix}
c_{v_2,v_1}^v(x,x)\\ 
\end{matrix}
\end{pmatrix}\in M_{c_1\times c_0}(k(x)). $$
 Then, the existence of the elements $d^w_{v_2,t,v_1}(x)$ and $d^w_{v_2,v_1}(x)$ in $k(x)$ as in the equations  $E(x)_{v,w}$ is equivalent to the existence of a matrix $D(x)$ of the form 
$$D(x)=\begin{pmatrix}
\begin{matrix} d^w_{v_2,t,v_1}(x)
\end{matrix}\,\vert\,
\begin{matrix}
d_{v_2,v_1}^w(x)\\ 
\end{matrix}
\end{pmatrix}\in M_{c_0\times c_1}(k(x)) $$
 such that $D(x)C(x)=I_{c_0}$. This is equivalent to the fact that $C(x)$ has rank $c_0$. 

Now, from the last corollary, we know that  
$\End_\Lambda(M(\lambda))=k$ iff  the equation $E(\lambda)_{v,w}$ holds for every $v,w\in \hueca{B}_{\frak{f}_0,\frak{f}_0}$. That is iff  there is a  matrix $D_\lambda\in M_{c_0\times c_1}(k)$ with $D_\lambda C(\lambda)=I_{c_0}$, or, equivalently, that the matrix $C(\lambda)$ has rank $c_0$. 

Our statement follows from the fact that $C(x)$ has rank $c_0$ iff $C(\lambda)$ has rank $c_0$ for infinitely many $\lambda\in k$. This   is a consequence of the next remark (\ref{L: matrices racionales invertibles}).   
\end{proof}

With the notation of the preceding proof, we have the following. 

\begin{corollary}\label{C: caract de M(lambda) brics iff C(x) with rank c0} $\End_\Lambda(M(\lambda))=k$, for infinitely $\lambda\in D(h)$, if and only if the matrix $C(x)$ has rank $c_0$.  
\end{corollary}

\begin{remark}\label{L: matrices racionales invertibles} Consider any matrix $A(x)\in M_{n\times n}(k(x))$. Then, $A(x)$ is invertible in $M_{n\times n}(k(x))$
iff $A(\lambda)$ is invertible in $M_{n\times n}(k)$ for infinitely many $\lambda\in k$. 
This  is easy to prove.
\end{remark}

\section{Radical morphisms of ${\cal B}\g k(x)$-bimodules}

We start with some  remarks on proper bimodules over layered ditalgebras.

\begin{remark}\label{R: reminder de bimods, ext y forget-embed}
Given a ditalgebra ${\cal A}$ with layer $(R,W=W_0\oplus W_1)$, 
 we recall that an \emph{${\cal A}\g k(x)$-bimodule}  $X$ is an object in ${\cal A}\g \Mod$ together with a morphism of algebras $\alpha_X:k(x)\rightmap{}\End_{\cal A}(X)^{op}$, and it is a \emph{proper bimodule} if $\alpha_X$ is of the form $\alpha_X=(\alpha^0_X,0)$. Given ${\cal A}\g k(x)$-bimodules $X$ and $Y$,  a \emph{morphism of ${\cal A}\g k(x)$-bimodules} $f:X\rightmap{}Y$ is a morphism in ${\cal A}\g\Mod$ such that $f\alpha_X(d(x))=\alpha_Y(d(x))f$, for all $d(x)\in k(x)$. The category of ${\cal A}\g k(x)$-bimodules is denoted by ${\cal A}\g k(x)\g\Mod$. In the following, we consider mainly proper bimodules and  the corresponding full subcategory ${\cal A}\g k(x)\g\Mod_p$ of ${\cal A}\g k(x)\g\Mod$.  Recall that any functor $F:{\cal A}'\g\Mod\rightmap{}{\cal A}\g\Mod$ induces a functor $F^{k(x)}:{\cal A}'\g k(x)\g \Mod\rightmap{}{\cal A}\g k(x)\g\Mod$. The latter restricts to proper bimodules whenever $F$ is a reduction functor. 
 
  We will assume some familiarity with bimodules, see \cite{CB2} or \cite{BSZ}. 
 \end{remark}
 
 \begin{remark}\label{R: morfismos de cal(A)-k(x)-bimods}
 Recall that any ${\cal A}\g k(x)$-bimodule $X$ has a natural structure of $R\g k(x)$-bimodule, with $md(x)=\alpha^0_X(d(x))[m]$, for $m\in X$ and $d(x)\in k(x)$.   
 
 Given two proper ${\cal A}\g k(x)$-bimodules $X$ and $Y$, the morphisms of ${\cal A}\g k(x)$-bimodules $f=(f^0,f^1):X\rightmap{}Y$ are given by pairs $(f^0,f^1)$ such that $f^0:X\rightmap{}Y$ is a morphism of $R\g k(x)$-bimodules and $f^1\in \Hom_{R\g R}(W_1,\Hom_k(X,Y))$ is such that, for any $w\in W_1$, the map $f^1(w):X\rightmap{}Y$ is a morphism of $k(x)$-modules. 
 
 Thus a morphism of proper ${\cal A}\g k(x)$-bimodules 
 $(f^0,f^1):X\rightmap{}Y$ is a pair such that $f^0\in \Hom_{R\g k(x)}(X,Y)$ and 
 $$f^1\in \Hom_{R\g R}(W_1,\Hom_{k(x)}(X,Y))\cong \Hom_{S\g S}(\hat{W}_1,\Hom_{k(x)}(X,Y)).$$
 We will denote by $\rad_{{\cal A}\g k(x)}$ the radical of the additive $k$-category ${\cal A}\g k(x)\g\Mod_p$.
 
 We have the \emph{scalar extension functor} $(-)^{k(x)}:{\cal A}\g\Mod\rightmap{}{\cal A}\g k(x)\g\Mod_p$ given by $X\longmapsto X\otimes_kk(x)$, where $\alpha_{X\otimes_kk(x)}(d)=(id_X\otimes did_{k(x)},0)$, for $d\in k(x)$. Moreover, given a morphism $f=(f^0,f^1):X\rightmap{}Y$ in ${\cal A}\g\Mod$, the functor $(-)^{k(x)}$ maps $f$ onto $f\otimes id:=(f^0\otimes id,f^1\otimes id)$, where $(f^1\otimes id)(v):=f^1(v)\otimes id\in \Hom_{k(x)}(X\otimes_k k(x),Y\otimes_k k(x))$, for $v\in W_1$. 
 
 We also have the \emph{forgetful-embedding functor}  
 $\sigma_{\cal A}:{\cal A}\g k(x)\g\Mod_p\rightmap{}{\cal A}\g\Mod$, which maps each proper bimodule $X$ onto the underlying $A$-module $X$, and maps each morphism $f=(f^0,f^1):X\rightmap{}Y$ onto $f=(f^0,f^1)$, which makes sense  because 
 $$f^0,f^1(v)\in \Hom_{k(x)}(X,Y)\subset \Hom_k(X,Y), \hbox{ for all } v\in W_1.$$ 
 Finally, notice that, for each reduction ${\cal A}\longmapsto {\cal A}^z$ of type $z\in \{a,r,d,e,u,X,l\}$, 
 there is an up to isomorphism  commutative diagram
$$\begin{matrix}
{\cal A}^z\g\Mod&\rightmap{F^z}&{\cal A}\g\Mod\\
\shortlmapup{\sigma_{{\cal A}^z}}&&\shortrmapup{\sigma_{\cal A}}\\
{\cal A}^z\g k(x)\g\Mod_p&\rightmap{(F^z)^{k(x)}}&{\cal A}\g k(x)\g\Mod_p\\
\shortlmapup{(-)^{k(x)}}&&\shortrmapup{(-)^{k(x)}}\\
{\cal A}^z\g\Mod&\rightmap{F^z}&{\cal A}\g\Mod,\\
\end{matrix}$$
 where the superindex $z$ denotes the type of reduction to which $F^{z}$ is associated:  $F^r$ corresponds to regularization, $F^d$ corresponds to deletion of idempotents,  $F^e$ to edge reduction, $F^u$ to unravelling of a loop, $F^a$ to absorption of a loop, $F^X$ to reduction with respect to an admissible module, and $F^l$ to a localization, see \cite[\S25]{BSZ}  
\end{remark}  
 
 \begin{remark}\label{R: introd de S(super x)} 
Again, ${\cal B}$ denotes our fixed minimal ditalgebra and we keep the notation of the preceding sections. 

 We have the simple proper ${\cal B}\g k(x)$-bimodule $k(x)\frak{f}_0=\frak{f}_0k(x)$, and, for any non-marked idempotent $\frak{f}$, we have the simple proper ${\cal B}\g k(x)$-bimodule $S^x_\frak{f}:=S_\frak{f}\otimes_kk(x)$. If we define  $\hat{\frak{f}}:=\frak{f}\otimes 1\in S_\frak{f}^x$, we have the one-dimensional $k(x)$-vector space $S_\frak{f}^x=\hat{\frak{f}}k(x)$ with basis $\hat{\frak{f}}$. 
\end{remark}

\begin{lemma}\label{L: k(x)-bases para rad entre bimodulos simples} Whenever $\frak{f},\frak{f}'$ are non-marked idempotents, we have the following:  
\begin{enumerate}
\item  For  $v\in \hueca{B}_{\frak{f}',\frak{f}}$, consider the morphism  $f^x_v:=(0,\hat{\psi}^x_v)\in \rad_{{\cal B}\g k(x)}(S_\frak{f}^x,S^x_{\frak{f}'})$, where   $\hat{\psi}^x_v\in \Hom_{R\g R}(W_1,\Hom_{k(x)}(S^x_\frak{f},S^x_{\frak{f}'}))$ extends the homomorphism $\psi^x_v\in \Hom_{S\g S}(\hat{W}_1,\Hom_{k(x)}(S^x_\frak{f},S^x_{\frak{f}'}))$, given by  
$\psi^x_v(w)[\hat{\frak{f}}]=\hat{v}(w)\hat{\frak{f}}'$, for $w\in \hat{W}_1$. Then,    
$$\{f^x_v\mid v\in \hueca{B}_{\frak{f}',\frak{f}}\}\hbox{ is a $k(x)$-basis for }
\rad_{{\cal B}\g k(x)}(S^x_\frak{f},S^x_{\frak{f}'}).$$   
 \item  For  $v\in \hueca{B}_{\frak{f},\frak{f}_0}$, consider the morphism  $f^x_v:=(0,\hat{\psi}^x_v)\in \rad_{{\cal B}\g k(x)}(k(x)\frak{f}_0,S^x_{\frak{f}})$, where   $\hat{\psi}^x_v\in \Hom_{R\g R}(W_1,\Hom_{k(x)}(k(x)\frak{f}_0,S^x_{\frak{f}}))$ extends the homomorphism  $\psi^x_v\in \Hom_{S\g S}(\hat{W}_1,\Hom_{k(x)}(k(x)\frak{f}_0,S^x_{\frak{f}}))$, given by  $\psi^x_v(w)[\frak{f}_0]=\hat{v}(w)\hat{\frak{f}}$, for $w\in \hat{W}_1$. Then,  
$$\{f^x_v\mid v\in \hueca{B}_{\frak{f},\frak{f}_0}\}\hbox{ is a $k(x)$-basis for }\rad_{{\cal B}\g k(x)}(k(x)\frak{f}_0,S^x_{\frak{f}}).$$ 
\item For  
$v\in \hueca{B}_{\frak{f}_0,\frak{f}}$, consider the morphism  $f^x_v:=(0,\hat{\psi}^x_v)\in \rad_{{\cal B}\g k(x)}(S^x_{\frak{f}},k(x)\frak{f}_0)$, where   $\hat{\psi}^x_v\in \Hom_{R\g R}(W_1,\Hom_{k(x)}(S^x_{\frak{f}},k(x)\frak{f}_0))$ extends the homomorphism  $\psi^x_v\in \Hom_{S\g S}(\hat{W}_1,\Hom_{k(x)}(S^x_{\frak{f}},k(x)\frak{f}_0))$, given by  $\psi^x_v(w)[\hat{\frak{f}}]=\hat{v}(w)\frak{f}_0$, for $w\in \hat{W}_1$. Then,  
$$\{f^x_v\mid v\in 
\hueca{B}_{\frak{f}_0,\frak{f}}\}\hbox{ is a $k(x)$-basis for }\rad_{{\cal B}\g k(x)}(S^x_{\frak{f}},k(x)\frak{f}_0).$$ 
 \item For  
 $v\in \hueca{B}_{\frak{f}_0,\frak{f}_0}$, consider the morphism  $f^x_v:=(0,\hat{\psi}^x_v)\in \rad_{{\cal B}\g k(x)}(k(x)\frak{f}_0,k(x)\frak{f}_0)$, where
 $\hat{\psi}^x_v\in \Hom_{R\g R}(W_1,\Hom_{k(x)}(k(x)\frak{f}_0,k(x)\frak{f}_0))$ extends the homomorphism  $\psi^x_v\in \Hom_{S\g S}(\hat{W}_1,\Hom_{k(x)}(k(x)\frak{f}_0,k(x)\frak{f}_0))$, given by  $\psi^x_v(w)[\frak{f}_0]=\hat{v}(w)\frak{f}_0$, for $w\in \hat{W}_1$. Then,  
$$\{f^x_v\mid v\in \hueca{B}_{\frak{f}_0,\frak{f}_0}\}\hbox{ is a $k(x)$-basis for }
\rad_{{\cal B}\g k(x)}(k(x)\frak{f}_0,k(x)\frak{f}_0).$$  
\end{enumerate}
\end{lemma}

\begin{proof}  Given 
$N,N'\in \{k(x)\frak{f}_0\}\cup\{S^x_{\frak{f}}\mid \frak{f}\in \frak{F} \hbox{ with } \frak{f}\not=\frak{f}_0 \}$, which are one-dimensional $k(x)$-vector spaces, we can show, as in (\ref{L: descripcion de rad(N,N')}), that $\rad_{{\cal B}\g k(x)}(N,N')$ consists   of the morphisms $(f^0,f^1)$ of ${\cal B}\g k(x)$-bimodules from $N$ to $N'$, as described in (\ref{R: morfismos de cal(A)-k(x)-bimods}), with $f^0=0$. 

In each case, we have to describe 
$\rad_{{\cal B}\g k(x)}(N,N')$ where  $N$ and $N'$ are ${\cal B}\g k(x)$-bimodules, which are one-dimensinal over $k(x)$. In order to simplify the  notation, define $\hat{\frak{f}}_0:=\frak{f}_0$, and consider the basis  elements 
 $\hat{\frak{f}}, \hat{\frak{f}}'$, 
 such that, as $k(x)$-vector spaces,  $N=\hat{\frak{f}}k(x)$ and 
 $N'=\hat{\frak{f}}'k(x)$, and take $v\in \hueca{B}_{\frak{f}',\frak{f}}$.  We have the $k(x)$-linear isomorphisms:
$$\begin{matrix}
D(\frak{f}'\hat{W}_1\frak{f})\otimes_kk(x)
&\rightmap{\theta^x} &
\Hom_{S\g S}(\hat{W}_1,\Hom_{k(x)}(N,N'))\hfill\\
&\cong&
\Hom_{R\g R}(W_1,\Hom_{k(x)}(N,N'))\hfill\\
&\cong&
\rad_{{\cal B}\g k(x)}(N,N'),
\hfill\\
\end{matrix}$$
where the first isomorphism $\theta^x$, similar to the one given  in (\ref{L: para construir morfismos entre simples}), maps each generator $\xi\otimes d(x)\in D(\frak{f}'\hat{W}_1\frak{f})\otimes_kk(x)$ onto the morphism 
$\theta^x(\xi\otimes d(x))$ defined by 
$\theta^x(\xi\otimes d(x))(w)[\hat{\frak{f}}]
=\xi(\frak{f}'w\frak{f})
\hat{\frak{f}}'d(x)$; the second  is the extension, and the third one follows from the first paragraph of this proof.

Let us verify that $\theta^x$ is indeed an isomorphism. Notice that we have an isomorphism of $k(x)$-vector spaces $\Hom_{k(x)}(N,N')\cong qk(x)$, where $q:N\rightmap{}N'$ is the $k(x)$-linear map such that $q(\hat{\frak{f}})=\hat{\frak{f}}'$. Every  $\psi\in \Hom_{S\g S}(\hat{W}_1,\Hom_{k(x)}(N,N'))$  is dermined uniquely by the values $\psi(v)$, for $v\in \hueca{B}_{\frak{f}',\frak{f}}$, that is by elements $d_v(x)\in k(x)$, such that $\psi(v)=qd_v(x)$. So, we see that $\theta^x$ is a morphism between $k(x)$-vector spaces with the same dimension $\vert 
\hueca{B}_{\frak{f}',\frak{f}}\vert$.  We claim that $\theta^x$ is surjective, because  
$\psi=\theta^x(
\sum_{v\in \hueca{B}_{\frak{f}',\frak{f}}}v^*\otimes d_v(x))$. Indeed, given $w\in \hueca{B}_{\frak{f}',\frak{f}}$, we have 
$$\theta^x(
\sum_{v\in \hueca{B}_{\frak{f}',\frak{f}}}v^*\otimes d_v(x))(w)[\hat{\frak{f}}]=
\sum_v v^*(w)\hat{\frak{f}}'d_v(x)
=\hat{\frak{f}}'d_w(x)
=q(\hat{\frak{f}})d_w(x)=\psi(w)[\hat{\frak{f}}].$$

 The $k(x)$-basis $\{v^*\otimes 1\mid v\in \hueca{B}_{\frak{f}',\frak{f}}\}$ of $D(\frak{f}'\hat{W}_1\frak{f})\otimes_kk(x)$ is mapped by this composition to the corresponding proposed basis of radical morphisms. In order to verify this, we  notice that, for $w\in \hat{W}_1$, we have  
$$\theta^x(v^*\otimes 1)(w)[\hat{\frak{f}}]=v^*(\frak{f}'w\frak{f})\hat{\frak{f}}'=\hat{v}(w)\hat{\frak{f}}'.$$
\end{proof}

 \begin{lemma}\label{L: de las compos de x-basicos}
Given $\frak{f},\frak{f}'\not=\frak{f}_0$,  we have:
\begin{enumerate}
\item For $v_1\in \hueca{B}_{\frak{f},\frak{f}_0}$ and $v_2\in \hueca{B}_{\frak{f}_0,\frak{f}}$, we have 
$$f_{v_2}^x f_{v_1}^x=(0,\psi_{v_2,v_1}^x), \hbox{ with } \psi_{v_2,v_1}^x(v)=c^v_{v_2,v_1}(x,x)id_{k(x)\frak{f}_0}, \hbox{ for } v\in \hueca{B}_{\frak{f}_0,\frak{f}_0}.$$
\item For $v_1\in \hueca{B}_{\frak{f},\frak{f}_0}$, 
 $v_2\in \hueca{B}_{\frak{f}',\frak{f}}$, and $v_3\in \hueca{B}_{\frak{f}_0,\frak{f}'}$ we have 
$$f^x_{v_3} f^x_{v_2} f_{v_1}^x=(0,\psi_{v_3,v_2,v_1}^x), \hbox{ with } \psi_{v_3,v_2,v_1}^x(v)=c^v_{v_3,v_2,v_1}(x,x)id_{k(x)\frak{f}_0}, \hbox{ for } v\in \hueca{B}_{\frak{f}_0,\frak{f}_0}.$$
\end{enumerate}
\end{lemma}

\begin{proof} (1): By definition, we have,
$\psi^x_{v_1}\in \Hom_{R\g R}(W_1,\Hom_{k(x)}(k(x)\frak{f}_0,S^x_\frak{f}))$  and  $\psi^x_{v_2}\in \Hom_{R\g R}(W_1,\Hom_{k(x)}(S^x_\frak{f},k(x)\frak{f}_0)).$
We claim that, for $w_1,w_2\in \hueca{B}$ and  any polynomial $c(x,y)\in k[x,y]$, we have 
$$\pi (\psi^x_{v_2}\otimes\psi^x_{v_1})[c(x,y)(w_2\otimes w_1)]=c(x,x)\hat{v}_2(w_2)\hat{v}_1(w_1)id_{k(x)\frak{f}_0}.$$  
Indeed, if $c(x,y)=\sum_{i=0}^ma_i(y)x^i$, we have
$$\begin{matrix}
\pi (\psi^x_{v_2}\otimes\psi^x_{v_1})[c(x,y)(w_2\otimes w_1)](\frak{f}_0)
&=&
\sum_{i=0}^m\pi(\psi^x_{v_2}[a_i(x)w_2]\otimes \psi^x_{v_1}[w_1x^i])(\frak{f}_0)\hfill\\
&=&
\sum_{i=0}^ma_i(x)x^i\hat{v}_2(w_2)\hat{v}_1(w_1)\frak{f}_0\hfill\\
&=&
c(x,x)\hat{v}_2(w_2)\hat{v}_1(w_1)\frak{f}_0.\hfill\\
\end{matrix}$$
Then, as remarked in (\ref{R: compos de segundas comp en B-Mod}), with the description of $\delta(v)$ given  in (\ref{R: descriptions of delta and}), we have 
$$\begin{matrix} 
\psi^x_{v_2,v_1}(v)&=&\pi (\psi^x_{v_2}\otimes\psi^x_{v_1})(\delta(v))\hfill\\
&=&
\sum_{w_1\in \hueca{B}^a,w_2\in \hueca{B}^b}\pi (\psi^x_{v_2}\otimes\psi^x_{v_1})(c^v_{w_2,w_1}(x,y)(w_2\otimes w_1))\hfill\\
&&+\,\pi (\psi^x_{v_2}\otimes\psi^x_{v_1})(r(v)).\hfill\\
\end{matrix}$$
Since $r(v)\in H$, its expression as a combination of tensor product involving  basis  elements in $\hueca{B}$ does not contain the elements $v_1$ and $v_2$, thus  $(\psi^x_{v_2}\otimes\psi^x_{v_1})(r(v))=0$.  Then (1) follows from our claim and the definition of the maps $\psi_{v_1}^x$ and $\psi_{v_2}^x$. 
\medskip

\noindent(2): As before, by definition, 
$\psi^x_{v_1}\in \Hom_{R\g R}(W_1,\Hom_{k(x)}(k(x)\frak{f}_0,S^x_\frak{f}))$,  
$\psi_{v_2}\in \Hom_{R\g R}(W_1,\Hom_{k(x)}(S^x_\frak{f},S^x_{\frak{f}'})),$
and $\psi^x_{v_3}\in \Hom_{R\g R}(W,\Hom_{k(x)}(S^x_{\frak{f}'},k(x)\frak{f}_0))$. Now, we claim that, for $w_1, w_2, w_3\in \hueca{B}$ and any polynomial $c(x,y)\in k[x,y]$, we have 
$$\pi (\psi^x_{v_3}\otimes\psi^x_{v_2}\otimes\psi^x_{v_1})[c(x,y)(w_3\otimes w_2\otimes w_1)]=c(x,x)\hat{v}_3(w_3)\hat{v}_2(w_2)\hat{v}_1(w_1)id_{k(x)\frak{f}_0}.$$ 
Which is easily verified, as before.   
Now, as remarked in (\ref{R: compos de segundas comp en B-Mod}), with the description of $(1\otimes\delta)\delta(v)$ given  in (\ref{R: descriptions of delta and}), we have 

\vbox{$$
\psi^x_{v_3,v_2,v_1}(v)
=\pi (\psi^x_{v_3}\otimes\psi^x_{v_2}\otimes\psi^x_{v_1})((1\otimes\delta)\delta(v))=$$
$$
\sum_{w_3\in \hueca{B}^b,w_2\in \hueca{B}_{\frak{Z}},w_1\in \hueca{B}^a}\pi (\psi^x_{v_3}\otimes\psi^x_{v_2}\otimes\psi^x_{v_1})(c^v_{w_3,w_2,w_1}(x,y)(w_3\otimes w_2\otimes w_1))$$
$$+\,\,\pi (\psi^x_{v_3}\otimes\psi^x_{v_2}\otimes\psi^x_{v_1})(r'(v)).$$}
Here, $r'(v)\in H'$, so its expression as a combination involving  tensor products of basis  elements in $\hueca{B}$ does not contain elements of the form $x^iv_3\otimes v_2\otimes v_1x^j$, so the last term in the preceding equality is zero.  Now, (2) follows from the preceding claim and the definition of the maps $\psi^x_{v_3}$, $\psi^x_{v_2}$, and $\psi^x_{v_1}$.  
\end{proof}

 \begin{remark} Let us denote by 
 ${\cal P}^1_{k(x)}(\Lambda)$ the category of 
 ${\cal P}^1(\Lambda)\g k(x)$-bimodules. Recall that its objects $X$ are objects  $X\in {\cal P}^1(\Lambda)$ equipped with an algebra morphism $\alpha_X:k(x)\rightmap{}\End_{{\cal P}^1(\Lambda)}(X)^{op}$. Given two ${\cal P}^1(\Lambda)\g k(x)$-bimodules $X$ and $X'$, a morphism  $u:X\rightmap{}X'$ of ${\cal P}^1(\Lambda)\g k(x)$-bimodules is a morphism in ${\cal P}^1(\Lambda)$ with $u\alpha_X(d)=\alpha_{X'}(d)u$, for all $d\in k(x)$. Thus the objects in ${\cal P}_{k(x)}^1(\Lambda)$ are triples $X=(P_1,P_2,\phi)\in {\cal P}^1(\Lambda)$, where $P_1$ and $P_2$ are $\Lambda\g k(x)$-bimodules and a morphism $u:X\rightmap{}X'$ in ${\cal P}_{k(x)}^1(\Lambda)$  is a morphism in ${\cal P}^1(\Lambda)$ such that $u_1$ and $u_2$ are morphisms of $\Lambda\g k(x)$-bimodules. We can identify ${\cal P}_{k(x)}^1(\Lambda)$ with ${\cal P}^1(\Lambda\otimes_kk(x))$, see (\ref{R: extension and radical}).
 
 Notice that we have the following up to isomorphism commutative diagram
$$\begin{matrix}
{\cal D}\g\Mod&\rightmap{ \ \Xi_\Lambda \ }&{\cal P}^1(\Lambda)\\
\shortlmapup{\sigma_{{\cal D}}}&&\shortrmapup{\sigma}\\
{\cal D}\g k(x)\g\Mod_p&\rightmap{ \Xi_\Lambda^{k(x)}}&{\cal P}^1_{k(x)}(\Lambda)\\
\shortlmapup{(-)^{k(x)}}&&\shortrmapup{(-)^{k(x)}}\\
{\cal D}\g\Mod&\rightmap{ \ \Xi_\Lambda \ }&{\cal P}^1(\Lambda),\\
\end{matrix}$$
where $\sigma$ denotes the forgetful-embedding functor and $(-)^{k(x)}$ is the scalar extension functor of (\ref{R: morfismos de cal(A)-k(x)-bimods}). 
 \end{remark}

\begin{lemma}\label{L: F conmuta con otimes k(x)} The functor $F:{\cal B}\g\Mod\rightmap{}{\cal P}^1(\Lambda)$ satisfies that, for any $X\in {\cal B}\g\Mod$, we have  a natural isomorphism 
$\Psi_X:F(X\otimes_kk(x))\rightmap{} F(X)\otimes_k k(x)$.  
So, for $t\in [1,n]$, we have 
$$
F(S_{\frak{f}_{z_t}}^x)= F(S_{\frak{f}_{z_t}}\otimes_kk(x))\cong F(S_{\frak{f}_{z_t}})\otimes_kk(x)=\hueca{S}_t\otimes_kk(x)$$
and, since the scalar extension functor 
$(-)^{k(x)}:{\cal B}\g\Mod\rightmap{}
{\cal B}\g{k(x)}\g\Mod_p$ maps $f_{u_t}$ onto  $f^x_{u_t}=f_{u_t}\otimes id$, we have the commutative diagram   
$$\begin{matrix}
F(S_{\frak{f}_{l_t}}^x)&=&F(S_{\frak{f}_{l_t}}\otimes_kk(x))&\rightmap{\Psi}&F(S_{\frak{f}_{l_t}})\otimes_kk(x)=\hueca{L}_t\otimes_kk(x)\\
\rmapdown{F(f^x_{u_t})}&&\hbox{ \hskip.6cm }\rmapdown{F(f_{u_t}\otimes id)}&&\rmapdown{F(f_{u_t})\otimes id}\hbox{\hskip 2.5cm}\rmapdown{\hat{\gamma}_t}\\
F(S_{\frak{f}_{r_t}}^x)&=&F(S_{\frak{f}_{r_t}}\otimes_kk(x))&\rightmap{\Psi}&F(S_{\frak{f}_{r_t}})\otimes_kk(x)=\hueca{R}_t\otimes_kk(x).\\
\end{matrix}$$
\end{lemma}

\begin{proof} Recall that the functor $F$ fixed at the beginning of \S4 is a composition of functors of the form $F=\Xi_\Lambda F^{z_1}\ldots F^{z_n}$, where $\Xi_\Lambda:{\cal D}\g\Mod\rightmap{}{\cal P}^1(\Lambda)$ is the usual equivalence, with ${\cal D}={\cal D}(\Lambda)$ the Drozd's ditalgebra of $\Lambda$, and $F^{z_1},\ldots,F^{z_n}$ are reduction functors of types $z_i\in \{a,r,d,e,u,X,l\}$. Here, 
the minimal ditalgebra ${\cal B}$ is obtained 
after a finite sequence of reductions 
$${\cal D}\mapsto {\cal D}^{z_1}\mapsto\cdots\mapsto {\cal D}^{z_1z_2\cdots z_n}={\cal B}$$
the superindex $z_i$ of each reduction functor $F^{z_i}:{\cal D}^{z_1\cdots z_i}\g\Mod\rightmap{}{\cal D}^{z_1\cdots z_{i-1}}\g\Mod$   denotes the type of reduction to which $F^{z_i}$ is associated.   

Concatenating the diagrams corresponding to each reduction functor $F^{z_i}$ described in (\ref{R: morfismos de cal(A)-k(x)-bimods}) and the diagram of the last remark, we obtain an up to isomorphism commutative diagram 
$$\begin{matrix}
{\cal B}\g\Mod&
\rightmap{\hbox{\hskip.5cm}  F \hbox{\hskip.5cm}}&{\cal P}^1(\Lambda)\\
\shortlmapup{\sigma_{\cal D}}&&\shortrmapup{\sigma}\\
{\cal B}\g k(x)\g\Mod_p&\rightmap{ \ \ F^{k(x)} \ \ }&{\cal P}_{k(x)}^1(\Lambda)\\
\shortlmapup{(-)^{k(x)}}&&\shortrmapup{(-)^{k(x)}}\\
{\cal B}\g\Mod&
\rightmap{ \hbox{\hskip.5cm}  F \hbox{\hskip.5cm}}&{\cal P}^1(\Lambda).\\
\end{matrix}$$
The last part follows immediately from the naturality of $\Psi$. 
\end{proof}

\begin{proposition}\label{P: radEnd(G)=0 -> radEnd(M(lambda))=0, casi todo lambda}
If $\rad\End_\Lambda(G)=0$, then $\rad \End_\Lambda(M(\lambda))=0$, for infinitely many $\lambda\in D(h)$. 
\end{proposition}

\begin{proof} If $\rad\End_\Lambda(G)=0$, then clearly, $\rad\End_{\Lambda\g k(x)}(G)=0$.  Assume that  $f\in \rad\End_{{\cal B}\g k(x)}(k(x)\frak{f}_0)$, then we have  $\Coker F(f)\in \rad\End_{\Lambda\g k(x)}(G)$. Therefore, $\Coker(F(f))=0$.  From (\ref{R: extendidos de generadores son generadores}), (\ref{R: isomorfos a generadores son generadores}), and 
(\ref{L: F conmuta con otimes k(x)}), we know that the families 
$\{id_{F(S^x_{\frak{f}_{z_t}})}\}_t$ and $\{F(f_{u_t}^x):
F(S^x_{\frak{f}_{l_t}})\rightmap{}F(S^x_{\frak{f}_{r_t}})\}_t$ are a  
kernel-generating system for $\Coker:{\cal P}^1(\hat{\Lambda})\rightmap{}\hat{\Lambda}\g\Mod$. Then we have that $F(f)$ is a finite sum of morphisms that  factor  through an object of the form $F(S^x_{\frak{f}_{z_t}})$ 
 or through a morphism of the form $F(f^x_{u_t})$. 

Since the functor $F$ is full and faithful, we get that the morphism $f$ is a finite sum of morphisms that  factor through an object of the form $S^x_{\frak{f}_{z_t}}$ or through a morphism of the form $f^x_{u_t}:S^x_{\frak{f}_{l_t}}\rightmap{}S^x_{\frak{f}_{r_t}}$. This holds for every morphism in $\rad\End_{{\cal B}\g k(x)}(k(x)\frak{f}_0)$.  
In particular, for the radical morphism $f^x_v$, for each $v\in \hueca{B}_{\frak{f}_0,\frak{f}_0}$.   Then, by (\ref{L: k(x)-bases para rad entre bimodulos simples}), we know 
there are elements $d^v_{v_2,t,v_1}(x)$ and $d^v_{v_2,v_1}(x)$ in $k(x)$ such that 
$$f^x_v=\sum_{(v_2,t,v_1)\in \hueca{I}_3}
d_{v_2,t,v_1}^v(x)f^x_{v_2}f^x_{u_t}f^x_{v_1}
+
\sum_{(v_2,v_1)\in \hueca{I}_2}
d^v_{v_2,v_1}(x)f^x_{v_2}f^x_{v_1}. $$
This is equivalent, by (\ref{L: de las compos de x-basicos}), to the equality
$$\psi_v^x=\sum_{(v_2,t,v_1)\in \hueca{I}_3}
d^v_{v_2,t,v_1}(x)\psi^x_{v_2,u_t,v_1}
+
\sum_{(v_2,v_1)\in \hueca{I}_2}d^v_{v_2,v_1}(x)\psi^x_{v_2,v_1}.$$
Evaluating at $w\in \hueca{B}_{\frak{f}_0,\frak{f}_0}$, we obtain 
$$\delta_{v,w}=\sum_{(v_2,t,v_1)\in \hueca{I}_3}
d^v_{v_2,t,v_1}(x)c^w_{v_2,u_t,v_1}(x,x)
+
\sum_{(v_2,v_1)\in \hueca{I}_2}d^v_{v_2,v_1}(x)c^w_{v_2,v_1}(x,x).$$
 Then, for infinitely $\lambda\in D(h)$, we have 
$$\delta_{v,w}=\sum_{(v_2,t,v_1)\in \hueca{I}_3}
d^v_{v_2,t,v_1}(\lambda)c^w_{v_2,u_t,v_1}(\lambda,\lambda)
+
\sum_{(v_2,v_1)\in \hueca{I}_2}d^v_{v_2,v_1}(\lambda)c^w_{v_2,v_1}(\lambda,\lambda).$$
From (\ref{C: caract de que End(Mlambda)=k casi siempre}), we get that 
$\End_\Lambda(M(\lambda))=k$. 
\end{proof}

\section{More on radical morphisms of ${\cal B}$-modules}

We want to describe a nice  set of generators of the radicals $\rad_{\cal B}(k(x)\frak{f}_0,k(x)\frak{f}_0)$ and $\rad_{\cal B}(S^x_{\frak{f}},k(x)\frak{f}_0)$, and to have a practical description of their compositions. So, we need some preliminary lemmas. 

\begin{lemma}\label{L: cocientes combin lin de basicos en k(x)} For any $r,s\in \hueca{N}$ and $\lambda\in k$, we have in $k(x)$: 
$$x^r/(x-\lambda)^s=h_{r,s}(x)+u_{r,s}(x)+\lambda^r/(x-\lambda)^s,$$
where $h_{r,s}(x)$ is a polynomial in $k[x]$ and $u_{r,s}(x)\in k(x)$ is a $k$-linear combination of the rational functions $\{1/(x-\lambda)^i\mid 0<i<s\}$. 
\end{lemma}

\begin{proof} Consider the isomorphism of algebras $k[x]\rightmap{}k[x]$ defined by $x\mapsto x-\lambda$. Thus, any  $f\in k[x]$ with degree $r$   
has the form $f(x)=\sum_{i=0}^r\lambda_i(x-\lambda)^i$, with $\lambda_0,\ldots,\lambda_r\in k$. For $f(x)=x^r$, evaluating at $\lambda$, we get $\lambda^r=f(\lambda)=\lambda_0$. Thus, $x^r=\sum_{i=1}^r\lambda_i(x-\lambda)^i+\lambda^r$. Then, we can divide by $(x-\lambda)^s$ on both sides of this equation to obtain  the expression in the statement of the lemma. 
\end{proof}

\begin{notation}\label{N: def de qsubc} Given any $r(x)\in k(x)$, we denote by $\mu_{r(x)}\in \End_k(k(x))$ the linear map given by multiplication by $r(x)$. 
For any linear map $q\in \End_k(k(x))$, we consider the associated morphism of $k$-algebras $k[x,y]\rightmap{}\End_k(k(x))$ such that $x\mapsto q\mu_x$ and $y\mapsto \mu_xq$. 
We denote the image of each $c\in k[x,y]$ by $q_c$. Thus, if $c(x,y)=\sum_{j=0}^ma_j(y)x^j\in k[x,y]$, with $a_j(y)\in k[y]$, and $z\in k(x)$, we have 
$$q_c(z)=\sum_{j=0}^ma_j(x)q(x^jz)=\left[\sum_ja_j(x)qx^j\right](z).$$
Given $c_1,c_2\in k[x,y]$, we have 
$q_{c_1c_2}=q_{c_1}q_{c_2}$. 
 Notice that, if $c,c_1\in k[x,y]$ satisfy $c_1(x,y)=c(x,y)s(x)$, for some non-zero $s(x)\in k[x]$, then the linear endomorphism  $q':=q\mu_{s(x)}^{-1}\in \End_k(k(x))$, is such that   $q'_{c_1}=q_c$. 
\end{notation}

\begin{lemma}[{\bf Convolution Lemma}]\label{L: convolution map sobre} If $0\not=c(x,y)\in k[x,y]$,  the map 
$$\End_k(k(x))\rightmap{}\End_k(k(x))\hbox{ such that }q\mapsto q_c$$ is surjective. 
\end{lemma}

\begin{proof} Assume that $c(x,y)=\sum_{j=0}^ma_j(y)x^j$, with $a_m(y)\not=0$, and take a linear map $p\in \End_k(k(x))$. Let us consider first the following:
\medskip

\noindent\emph{Special Case: There is no $\lambda\in k$ such that 
$c(\lambda,y)=\sum_{j=0}^ma_j(y)\lambda^j=0.$}
\medskip

In order to define the appropriate linear map $q:k(x)\rightmap{}k(x)$, we have to define its values on the $k$-basis of $k(x)$ formed by the elements $x^i$ and $1/(x-\lambda)^j$, for $i\geq 0$, $j\geq 1$, and $\lambda\in k$. We are looking for a linear map $q$ such that the equations 
\begin{equation}\tag{$Eq(z)$}
 p(z)=\sum_{j=0}^ma_j(x)q(x^jz)
\end{equation}
 are satisfied for each one of these basis elements $z$.

Having in mind the basis element $z=1$, we first define the images $q(1)$, $q(x)$,...,$q(x^m)$ such that the equation $Eq(1)$ is satisfied:   
$$a_0(x)q(1)+a_1(x)q(x)+\cdots+a_m(x)q(x^m)=p(1).$$
For this, we first notice that $c(x,y)\not=0$ implies that $c(x^n,x)\not=0$, for $n\in \hueca{N}$ big enough. Then, we take $q(1):=c(x^n,x)^{-1}p(1)$ and, $q(x^j):=x^{nj}q(1)$, for $j\in [1,m]$. Equation $Eq(1)$ holds because  we have 
$$\begin{matrix}
\sum_{j=0}^ma_j(x)q(x^j)&=&a_0(x)c(x^n,x)^{-1} p(1)+\sum_{j=1}^ma_j(x)x^{nj}c(x^n,x)^{-1}p(1)\hfill\\
&=&
(\sum_{j=0}^ma_j(x)x^{nj})c(x^n,x)^{-1}p(1)=p(1).\hfill\\
\end{matrix}$$

Then, having in mind the basis element $z=x$,  we define 
$$q(x^{m+1}):=a_m(x)^{-1}[p(x)-a_0(x)q(x)-a_1(x)q(x^2)-\cdots -a_{m-1}(x)q(x^m)],$$
so that the equation $Eq(x)$ is satisfied. Similarly, in order to comply with the equations $Eq(x^i)$, we define inductively 
$$q(x^{m+i}):=a_m(x)^{-1}[p(x^i)-a_0(x)q(x^i)-a_1(x)q(x^{i+1})-\cdots -a_{m-1}(x)q(x^{m-1+i})].$$
Next, we want to define the value $q(1/(x-\lambda))$ in order to comply with equation $Eq(1/(x-\lambda))$, that is with 
$$p(1/(x-\lambda))=\sum_{j=0}^ma_j(x)q(x^j/(x-\lambda)).$$
From (\ref{L: cocientes combin lin de basicos en k(x)}), we know that $x^j/(x-\lambda)=h_j(x)+\lambda^j/(x-\lambda)$, for some $h_j(x)\in k[x]$, for each $j\in [1,m]$. So, the preceding equation can be written as 
$$\begin{matrix}
p(1/(x-\lambda))
&=&
\sum_{j=0}^ma_j(x)q[h_j(x)+\lambda^j/(x-\lambda)]\hfill\\
&=&
\sum_{j=0}^ma_j(x)q(h_j(x)) +\sum_{j=0}^ma_j(x)\lambda^jq(1/(x-\lambda))\hfill\\
&=&
\sum_{j=0}^ma_j(x)q(h_j(x)) +c(\lambda,x)q(1/(x-\lambda)).\hfill\\
\end{matrix}
$$
So, if we define 
$$q(1/(x-\lambda)):=c(\lambda,x)^{-1}\left[p(1/(x-\lambda))-\sum_{j=0}^ma_j(x)q(h_j(x))\right],$$
the equation $Eq(1/(x-\lambda))$ is satisfied. Now, assume that $s\geq 1$ and that $q(1/(x-\lambda)^i)$ has been defined for $i\in [1,s]$ such that the corresponding equations $Eq(1/(x-\lambda)^i)$ hold. As in the preceding case, we 
want to define the value $q(1/(x-\lambda)^{s+1})$ in order to comply with equation $Eq(1/(x-\lambda)^{s+1})$, that is with 
$$p(1/(x-\lambda)^{s+1})=\sum_{j=0}^ma_j(x)q(x^j/(x-\lambda)^{s+1}).$$
From (\ref{L: cocientes combin lin de basicos en k(x)}), we know that $x^j/(x-\lambda)^{s+1}=h_j(x)+u_j(x)+\lambda^j/(x-\lambda)^{s+1}$, for some $h_j(x)\in k[x]$ and $u_j(x)\in k(x)$, where $u_j(x)$ is a $k$-linear combination of the rational functions $\{1/(x-\lambda)^i\mid 0<i\leq s\}$, for each $j\in [1,m]$. So, the preceding equation can be written as 
$$\begin{matrix}
p(1/(x-\lambda)^{s+1})
&=&
\sum_{j=0}^ma_j(x)q[h_j(x)+u_j(x)+\lambda^j/(x-\lambda)^{s+1}]\hfill\\
&=&
\sum_{j=0}^ma_j(x)q[h_j(x)+u_j(x)] +c(\lambda,x)q(1/(x-\lambda)^{s+1}).\hfill\\
\end{matrix}
$$
So, if we define 
$$q(1/(x-\lambda)^{s+1}):=c(\lambda,x)^{-1}\left[p(1/(x-\lambda)^{s+1})-\sum_{j=0}^ma_j(x)q[h_j(x)+u_j(x)]\right],$$
the equation $Eq(1/(x-\lambda)^{s+1})$ is satisfied. Thus there is $q\in \End_k(k(x))$ such that $q_c=p$. 
 \medskip

\noindent\emph{General Case:} 
\medskip

By assumption, $0\not=c(x,y)\in k[x,y]$. We claim that there are $\lambda_1,\ldots,\lambda_t\in k$ and a polynomial $c_t(x,y)\in k[x,y]$ such that 
$$c(x,y)=c_t(x,y)(x-\lambda_1)\cdots(x-\lambda_t),$$
and $c_t(\lambda,y)\not=0$, for all $\lambda\in k$. 

Indeed, if $c(\lambda_1,y)=0$, for some $\lambda_1\in k$, the scalar $\lambda_1$ is a root of $c(x,y)\in k[y][x]$, so $c(x,y)=c_1(x,y)(x-\lambda_1)$. If $c_1(\lambda_2,y)=0$, for some $\lambda_2\in k$, we obtain $c(x,y)=c_2(x,y)(x-\lambda_1)(x-\lambda_2)$, with $c_2(x,y)\in k[x,y]$. This process can only be repeated a finite number of times because the degree of the polynomial $c(x,y)\in k[y][x]$ is finite. Our claim follows from this.

Thus, 
  $c(x,y)=c_t(x,y)s(x)$, with $s(x)=(x-\lambda_1)\cdots(x-\lambda_t)$, for some $\lambda_1,\ldots,\lambda_t\in k$, and $c_t(x,y)\in k[x,y]$ such that $c_t(\lambda,y)\not=0$ for all $\lambda\in k$. Thus, given $p\in \End_k(k(x))$, there is some $q\in \End_k(k(x))$ with $q_{c_t}=p$, by the preceding special case. From the remark at the end of (\ref{N: def de qsubc}), we know that $q':=q\mu^{-1}_{s(x)}\in \End_k(k(x))$ is such that $q'_c=q_{c_t}=p$. 
\end{proof}

\begin{remark}\label{L: generadores de End_calB(k(x)f_0)}
\begin{enumerate}
\item For $q\in \End_k(k(x))$ and $v\in \hueca{B}_{\frak{f}_0,\frak{f}_0}$, we have the morphism  
$$f^x_{v,q}=(0,\hat{\psi}^x_{v,q})\in \rad_{\cal B}(k(x)\frak{f}_0,k(x)\frak{f}_0),$$
where $\psi^x_{v,q}\in \Hom_{S\g S}(\hat{W_1},\Hom_k(k(x)\frak{f}_0,k(x)\frak{f}_0))$ satisfies $\psi^x_{v,q}(w)[z\frak{f}_0]=\hat{v}(w)q(z)\frak{f}_0$, for $w\in \hat{W}_1$ and $z\in k(x)$. 
The set 
$$\{ f_{v,q}^x\mid q\in \End_k(k(x))\hbox{ and } v\in \hueca{B}_{\frak{f}_0,\frak{f}_0}\}$$
generates the $k$-vector space $\rad_{\cal B}(k(x)\frak{f}_0,k(x)\frak{f}_0)$.
\item For $q\in \End_k(k(x))$ and $v\in \hueca{B}_{\frak{f}_0,\frak{f}}$, with $\frak{f}\not=\frak{f}_0$, we have the morphism  
$$f^x_{v,q}=(0,\hat{\psi}^x_{v,q})\in \rad_{\cal B}(S^x_\frak{f},k(x)\frak{f}_0),$$
where $\psi^x_{v,q}\in \Hom_{S\g S}(\hat{W_1},\Hom_k(S^x_\frak{f},k(x)\frak{f}_0))$ satisfies $\psi^x_{v,q}(w)[\hat{\frak{f}}z]=\hat{v}(w)q(z)\frak{f}_0$, for $w\in \hat{W}_1$ and $z\in k(x)$.  
The set 
$$\{ f_{v,q}^x\mid q\in \End_k(k(x))\hbox{ and } v\in \hueca{B}_{\frak{f}_0,\frak{f}}\}$$
generates the $k$-vector space $\rad_{\cal B}(S^x_\frak{f},k(x)\frak{f}_0)$.
\end{enumerate}
\medskip

Indeed, consider the isomorphism $\End_k(k(x))\rightmap{}\End_k(k(x)\frak{f}_0)$ mapping each linear map $q$ onto $\hat{q}$ defined by $\hat{q}(z\frak{f}_0)=q(z)\frak{f}_0$, for $z\in k(x)$.

Notice that $f_{v,q}^x=f_{v,\hat{q}}$, as defined in (\ref{L: radEnd(k(x)f0}). So, we have $(1)$. For $(2)$, we have $f^x_{v,q}=f_{v,\hat{q}}$ as defined in (\ref{L: decompos of HomB(gen,sumadesimples})(2), for the $k$-vector space $V=k(x)$, where $\hat{q}(\hat{\frak{f}}z)=q(z)\frak{f}_0$.

The next lemma involves, for $v_1\in \hueca{B}_{\frak{f},\frak{f}_0}$, where $\frak{f}$ is a non-marked idempotent, the morphisms 
$f^x_{v_1}\in \Hom_{{\cal B}\g k(x)}(k(x)\frak{f}_0,S^x_\frak{f})\subseteq \Hom_{{\cal B}}(k(x)\frak{f}_0,S^x_\frak{f})$ introduced in (\ref{L: k(x)-bases para rad entre bimodulos simples}). It also  involves, for $v_2\in \hueca{B}_{\frak{f}',\frak{f}}$, where $\frak{f}$ and $\frak{f}'$ are non-marked idempotents, the morphisms 
$f^x_{v_2}\in \Hom_{{\cal B}\g k(x)}(S^x_\frak{f},S^x_{\frak{f}'})\subseteq \Hom_{{\cal B}}(S^x_\frak{f},S^x_{\frak{f}'})$ introduced in that same lemma. 
Moreover, it also refers to the  morphisms  mentioned (1) and (2). 
\end{remark}

\begin{lemma}\label{L: compos con basicos y fx(v2,v1,q)} Given $q\in \End_k(k(x))$, composing in ${\cal B}\g\Mod$, the following holds:
\begin{enumerate}
\item For $v_1\in \hueca{B}_{\frak{f},\frak{f}_0}$ and $v_2\in \hueca{B}_{\frak{f}_0,\frak{f}}$, where $\frak{f}\not=\frak{f}_0$, we have 
$$f^x_{v_2,q}f^x_{v_1}=(0,\psi^x_{v_2,v_1,q}), \hbox{ with } \psi^x_{v_2,v_1,q}(v)=\hat{q}_{c^v_{v_2,v_1}}, \hbox{ for } v\in \hueca{B}_{\frak{f}_0,\frak{f}_0}.$$
\item For $v_1\in \hueca{B}_{\frak{f},\frak{f}_0}$, $v_2\in \hueca{B}_{\frak{f}',\frak{f}}$,  and $v_3\in \hueca{B}_{\frak{f}_0,\frak{f}'}$, with $\frak{f}\not=\frak{f}_0\not=\frak{f}'$, we have 
$$f^x_{v_3,q}f^x_{v_2}f^x_{v_1}=(0,\psi^x_{v_3,v_2,v_1,q}), \hbox{ with } \psi^x_{v_3,v_2,v_1,q}(v)=\hat{q}_{c^v_{v_3,v_2,v_1}},\hbox{ for } v\in \hueca{B}_{\frak{f}_0,\frak{f}_0}$$
\end{enumerate}
\end{lemma}

\begin{proof} (1): We proceed as in (\ref{L: de las compos de x-basicos}). 
 By definition, we have the morphisms 
$\psi^x_{v_1}\in \Hom_{R\g R}(W_1,\Hom_k(k(x)\frak{f}_0,S^x_\frak{f}))$  and  $\psi^x_{v_2,q}\in \Hom_{R\g R}(W_1,\Hom_k(S^x_\frak{f},k(x)\frak{f}_0)).$
As remarked in (\ref{R: compos de segundas comp en B-Mod}), with the description of $\delta(v)$ given  in (\ref{R: descriptions of delta and}), we have 
$$\begin{matrix} 
\psi^x_{v_2,v_1,q}(v)&=&\pi (\psi^x_{v_2,q}\otimes\psi^x_{v_1})(\delta(v))\hfill\\
&=&
\sum_{w_1\in \hueca{B}^a,w_2\in \hueca{B}^b}\pi (\psi^x_{v_2,q}\otimes\psi^x_{v_1})(c^v_{w_2,w_1}(x,y)(w_2\otimes w_1))\hfill\\
&&+\,\pi (\psi^x_{v_2,q}\otimes\psi^x_{v_1})(r(v)).\hfill\\
\end{matrix}$$
Since $r(v)\in H$, its expression as a combination of tensor products involving  basis  elements in $\hueca{B}$ does not contain the elements $v_1$ and $v_2$. It follows that  $(\psi^x_{v_2,q}\otimes\psi^x_{v_1})(r(v))=0$. From the definition of $\psi^x_{v_2,q}$ and $\psi^x_{v_1}$, we know that the summand corresponding to the pair $(w_2, w_1)$ is zero, whenever $(w_2,w_1)\not=(v_2,v_1)$. So, we have 
$$\psi^x_{v_2,v_1,q}(v)=\pi (\psi^x_{v_2,q}\otimes\psi^x_{v_1})(c^v_{v_2,v_1}(x,y)(v_2\otimes v_1)). 
$$
Assuming that $c^v_{v_2,v_1}(x,y)=\sum_{j=0}^ma_j(y)x^j\in k[x,y]$, we have 
$$\begin{matrix}\psi^x_{v_2,v_1,q}(v)

&=&
\sum_{j=0}^m\psi^x_{v_2,q}(a_j(x)v_2)\psi^x_{v_1}(v_1x^j)\hfill\\
&=&
\sum_{j=0}^ma_j(x)\psi^x_{v_2,q}(v_2)\psi^x_{v_1}(v_1)x^j.\hfill\\
\end{matrix}$$
Then, evaluating at $z\frak{f}_0$, we have
$$\begin{matrix}
\psi^x_{v_2,v_1,q}(v)[z\frak{f}_0]
&=&
\sum_{j=0}^ma_j(x)\psi^x_{v_2,q}(v_2)[\psi^x_{v_1}(v_1)[x^jz\frak{f}_0]]\hfill\\
&=&
\sum_{j=0}^ma_j(x)\psi^x_{v_2,q}(v_2)[\hat{\frak{f}}x^jz]\hfill\\
&=&
\sum_{j=0}^ma_j(x)q(x^jz)\frak{f}_0
=\hat{q}_{c^v_{v_2,v_1}}(z\frak{f}_0).\hfill\\
\end{matrix}$$
\medskip

\noindent(2): As before, we have the maps  
$\psi^x_{v_1}\in \Hom_{R\g R}(W_1,\Hom_k(k(x)\frak{f}_0,S^x_\frak{f}))$,  
$\psi^x_{v_2}\in \Hom_{R\g R}(W_1,\Hom_k(S^x_\frak{f},S^x_{\frak{f}'})),$
and $\psi^x_{v_3,q}\in \Hom_{R\g R}(W_1,\Hom_k(S^x_{\frak{f}'},k(x)\frak{f}_0))$. 
As remarked in (\ref{R: compos de segundas comp en B-Mod}), with the description of $(1\otimes\delta)\delta(v)$ given  in (\ref{R: descriptions of delta and}), we have 

\vbox{$$
\psi^x_{v_3,v_2,v_1,q}(v)
=\pi (\psi^x_{v_3,q}\otimes\psi^x_{v_2}\otimes\psi^x_{v_1})((1\otimes\delta)\delta(v))=$$
$$
\sum_{w_3\in \hueca{B}^b,w_2\in \hueca{B}_{\frak{Z}},w_1\in \hueca{B}^a}\pi (\psi^x_{v_3,q}\otimes\psi^x_{v_2}\otimes\psi^x_{v_1})(c^v_{w_3,w_2,w_1}(x,y)(w_3\otimes w_2\otimes w_1))$$
$$+\,\,\pi (\psi^x_{v_3,q}\otimes\psi^x_{v_2}\otimes\psi^x_{v_1})(r'(v)).$$}
Here, $r'(v)\in H'$, so its expression as a combination of  tensor products involving basis elements in $\hueca{B}$ does not contain elements of the form $x^iv_3\otimes v_2\otimes v_1x^j$, so the last term in the preceding equality is zero.  By definition of the maps $\psi^x_{v_3,q}$, $\psi^x_{v_2}$, and $\psi^x_{v_1}$, whenever $(v_3,v_2,v_1)\not=(w_3,w_2,w_1)$, the corresponding summand is zero. So we get
$$
\psi^x_{v_3,v_2,v_1,q}(v)
=\pi (\psi^x_{v_3,q}\otimes\psi^x_{v_2}\otimes\psi^x_{v_1})(c^v_{v_3,v_2,v_1}(x,y)(v_3\otimes v_2\otimes v_1)).$$
Assume that $c_{v_3,v_2,v_1}^v(x,y)=\sum_{j=0}^ma_j(y)x^j$, then we have
$$\begin{matrix}
\psi^x_{v_3,v_2,v_1,q}(v)
&=&\sum_{j=0}^m
\psi^x_{v_3,q}(a_j(x)v_3)\psi^x_{v_2}(v_2)\psi^x_{v_1}(v_1x^j)\hfill\\
&=&
\sum_{j=0}^m
a_j(x)\psi^x_{v_3,q}(v_3)\psi^x_{v_2}(v_2)\psi^x_{v_1}(v_1)x^j.\hfill\\
\end{matrix}$$
Evaluating at $z\frak{f}_0\in k(x)\frak{f}_0$, we get 
$$
\begin{matrix}
\psi^x_{v_3,v_2,v_1,q}(v)[z\frak{f}_0]
&=&
\sum_{j=0}^ma_j(x)\psi^x_{v_3,q}(v_3)[\psi^x_{v_2}(v_2)[\psi_{v_1}^x(v_1)[x^jz\frak{f}_0]]]\hfill\\
&=&
\sum_{j=0}^ma_j(x)\psi_{v_3,q}(v_3)[\psi^x_{v_2}(v_2)[\hat{\frak{f}}x^jz]]\hfill\\
&=&
\sum_{j=0}^ma_j(x)\psi^x_{v_3,q}(v_3)[\hat{\frak{f}}'x^jz]\hfill\\
&=&
\sum_{j=0}^ma_j(x)q(x^jz)\frak{f}_0=\hat{q}_{c_{v_3,v_2,v_1}^v}(z\frak{f}_0).\hfill
\end{matrix}$$
\end{proof}

\section{Normalization and proof of the  main theorem}

For the last part of  the proof of (\ref{T: gen-bricks vs bricks}), we need to make some adjustment to the minimal ditalgebra ${\cal B}$ in order to gain some control on the behaviour of its differential. 
 
\begin{definition}\label{D: hat(delta) y admis} Given a minimal ditalgebra  ${\cal B}$ and a basis $\hueca{B}$ of $W_1$ as before, we consider the following $R\frak{f}_0\g R\frak{f}_0$-bimodule
$$U:=
\left[\bigoplus_{t=1}^n
\frak{f}_0W_1\frak{f}_{z_t}
\otimes_R\frak{f}_{z_t}W_1\frak{f}_0\right]
\oplus\left[\bigoplus_{t=1}^n
\frak{f}_0W_1\frak{f}_{r_t}\otimes_Rku_t\otimes_R\frak{f}_{l_t}W_1\frak{f}_0\right].$$
Consider the directed subset of $U$  
 $$\hueca{T}=\{v_2\otimes v_1\mid (v_2,v_1)\in \hueca{I}_2\} \cup \{v_2\otimes u_t\otimes v_1\mid (v_2,t,v_1) \in \hueca{I}_3\}.$$
 Thus, the set $\hueca{T}$ freely generates the $R\frak{f}_0\g R\frak{f}_0$-bimodule $U$. 
 Consider also the morphism of $R\frak{f}_0\g R\frak{f}_0$-bimodules $\hat{\delta}:\frak{f}_0W_1\frak{f}_0\rightmap{}U$ given as the composition
 $$\frak{f}_0W_1\frak{f}_0\rightmap{ \ \  (\delta,[1\otimes\delta]\delta)^t \ \ }(\frak{f}_0W_1\otimes_RW_1\frak{f}_0)\oplus 
 (\frak{f}_0W_1\otimes_RW_1\otimes_RW_1\frak{f}_0)
 \rightmap{\nu}U,$$
 where $\nu$ is the projection map. The  morphism $\hat{\delta}$ is completely determined by its values taken on the  basis  elements $v\in \hueca{B}_{\frak{f}_0,\frak{f}_0}$, which have the form 
 $$\hat{\delta}(v)=\sum_{\tau\in \hueca{T}}c_\tau^v(x,y)\tau,$$
 where $c_t^v(x,y)\in k[x,y]$, see (\ref{R: descriptions of delta and}).  
 For any numberings $v_1,\ldots,v_{c_0}$ of $\hueca{B}_{\frak{f}_0,\frak{f}_0}$ and $\tau_1,\ldots, \tau_{c_1}$ of $\hueca{T}$, we have the matrix
 $$C(x,y):=(c_{i,j}(x,y))\in M_{c_0\times c_1}(k[x,y]), \hbox{ with } c_{i,j}(x,y):=c^{v_i}_{\tau_j}(x,y),$$
 which we call \emph{the coefficient matrix of $\hat{\delta}$ associated to those numberings}. With this notation, we have 
 $$\hat{\delta}(v_i)=\sum_{j=1}^{c_1}c_{i,j}(x,y)\tau_j.$$
 We say that the matrix \emph{$C(x,y)$ is normal} iff $c_{i,j}=0$ whenever $i>j$ and $c_{i,i}\not=0$, for all $i,j\in [1,c_0]$.  
 We will say that ${\cal B}$ is \emph{admissible} if there is a basis $\hueca{B}$ of $W_1$ and such numberings of $\hueca{B}_{\frak{f}_0,\frak{f}_0}$ and $\hueca{T}$ such that the associated matrix $C(x,y)$ is  normal.   
\end{definition}

\begin{proposition}\label{P: C normal --> rad(k(x)f0) se factorizan bien} Assume that ${\cal B}$ is a minimal admissible ditalgebra,  then every morphism in 
$\rad \End_{\cal B}(k(x)\frak{f}_0)$ is a $k$-linear combination of morphisms that  factor through some object $S^x_{\frak{f}_{z_t}}$ or  some morphism $f^x_{u_t}$, for some $t\in [1,n]$. 
\end{proposition}

\begin{proof} From (\ref{L: generadores de End_calB(k(x)f_0)})(1), it will be enough to show that every morphism $f^x_{v,q}$, where $v\in \hueca{B}_{\frak{f}_0,\frak{f}_0}$ and $q\in \End_k(k(x))$, is a $k$-linear combination of morphisms which admit such factorizations. 

We adopt the notation of (\ref{D: hat(delta) y admis}) and proceed by induction over $i\in [1,c_0]$ to show that, for any $q\in \End_k(k(x))$, the morphism $f^x_{v_i,q}$ has the wanted form.  So, we consider the base of the induction case $i=1$ and take $q\in \End_k(k(x))$.  

There are two possibilities for the form of $\tau_1$. The first one is $\tau_1=w_2\otimes w_1$, with $(w_2,w_1)\in \hueca{I}_2$. So, $w_1\in \hueca{B}_{\frak{f}_{z_t},\frak{f}_0}$, $w_2\in \hueca{B}_{\frak{f}_0,\frak{f}_{z_t}}$, for some $t\in [1,n]$. Since $C(x,y)$ is normal, we get that $c^{v_1}_{w_2,w_1}(x,y)=c^{v_1}_{\tau_1}(x,y)=c_{1,1}(x,y)\not=0$. From the convolution lemma, we know that there is some $p\in \End_k(k(x))$ such that $q=p_{c^{v_1}_{w_2,w_1}}$. According to (\ref{L: compos con basicos y fx(v2,v1,q)})(1), the composition of the morphisms 
$f^x_{w_1}\in \Hom_{\cal B}(k(x)\frak{f}_0,S^x_{\frak{f}_{z_t}})$ and $f^x_{w_2,p}\in \Hom_{\cal B}(S^x_{\frak{f}_{z_t}},k(x)\frak{f}_0)$ is 
$$f^x_{w_2,p}f^x_{w_1}=(0,\psi^x_{w_2,w_1,p}).$$ 
Here, we have $\psi^x_{w_2,w_1,p}(v_1)=\hat{p}_{c^{v_1}_{w_2,w_1}}=\hat{q}$.  For $i>1$, we have  $\psi^x_{w_2,w_1,p}(v_i)=\hat{p}_{c^{v_i}_{w_2,w_1}}=0$, because, by the normality of $C(x,y)$, we have $c^{v_i}_{w_2,w_1}=c^{v_i}_{\tau_1}=c_{i,1}=0$. Hence, $\psi^x_{w_2,w_1,p}=\psi^x_{v_1,q}$ and, therefore, $f^x_{w_2,p}f^x_{w_1}=f^x_{v_1,q}$. That is $f^x_{v_1,q}$ factors through $S^x_{\frak{f}_{z_t}}$. 

The other possibility for the form of $\tau_1$ is  $\tau_1=w_2\otimes u_t\otimes w_1$, with $(w_2,t,w_1)\in \hueca{I}_3$. As before, there is  $p\in \End_k(k(x))$ such that $q=p_{c^{v_1}_{w_2,u_t,w_1}}$. Using (\ref{L: compos con basicos y fx(v2,v1,q)})(2), we obtain
$$f^x_{w_2,p}f^x_{u_t}f^x_{w_1}=f^x_{v_1,q}.$$
Thus, the morphism $f^x_{v_1,q}$ factors through $f^x_{u_t}$. 

Now, take $i\in [2,c_0]$ and assume that $f^x_{v_j,q_j}$ has the wanted form, for every  $j\in [1,i-1]$ and $q_j\in \End_k(k(x))$. Take $q\in \End_k(k(x))$ and let us show that $f^x_{v_i,q}$ is a $k$-linear combination of morphisms which factor as wanted.  

Again, since $C(x,y)$ is normal, we have $c^{v_i}_{\tau_i}(x,y)=c_{i,i}(x,y)\not=0$ and there is some $p\in \End_k(k(x))$ with $q=p_{c_{\tau_i}^{v_i}}$. Now, assume that $\tau_i$ has the form $\tau_i=w_2\otimes u_t\otimes w_1$, for some $(w_2,t,w_1)\in \hueca{I}_3$. Thus, $q=p_{c_{w_2,t,w_1}^{v_i}}$ and, from (\ref{L: compos con basicos y fx(v2,v1,q)})(2), we have the composition
$$f^x_{w_2,p}f^x_{u_t}f^x_{w_1}=(0,\psi^x_{w_2,u_t,w_1,p}).$$  
Here, we have $\psi^x_{w_2,u_t,w_1,p}(v_i)=\hat{p}_{c^{v_i}_{w_2,t,w_1}}=\hat{q}$. 
For $s>i$, we have $\psi^x_{w_2,u_t,w_1,p}(v_s)=0$, because, by the normality of $C(x,y)$, we have $c_{w_2,u_t,w_1}^{v_s}=c^{v_s}_{\tau_i}=c_{s,i}=0$. Therefore, we have 
$$\psi^x_{w_2,u_t,w_1,p}=\psi^x_{v_i,q}+\psi^x,$$
 where $\psi^x\in \Hom_{S\g S}(\hat{W}_1,\Hom_k(k(x)\frak{f}_0,k(x)\frak{f}_0))$ is such that  $\psi^x(v_s)=0$, for all $s\geq i$. This implies that $\psi^x=\sum_{j=1}^{i-1}\psi^x_{v_j,q_j}$, where $q_j\in \End_k(k(x))$ is such that $\psi^x(v_j)=\hat{q}_j$. Therefore, we have 
  $$f^x_{w_2,p}f^x_{u_t}f^x_{w_1}=f^x_{v_i,q}+\sum_{j=1}^{i-1}f^x_{v_j,q_j}.$$
  Then, we can apply the induction hypothesis to each $f^x_{v_j,q_j}$, to obtain that $f^x_{v_i,q}$ is a $k$-linear combination of morphisms as wanted. 
  
  In the remaining case, where $\tau_i=w_2\otimes w_1$, we proceed similarly. 
\end{proof}
 
 In the following, we will show that the  minimal ditalgebra ${\cal B}$ fixed at the begining of \S4 and satisfying (\ref{P: el funtor F}), can be replaced by a minimal ditalgebra of the same type but which is, in addition, admissible. 

\begin{remark}[\bf Reduction by localization]\label{R: localization}

Given a seminested ditalgebra ${\cal A}$, we have the seminested ditalgebra ${\cal A}^l$ obtained from ${\cal A}$ by localization.  It has an associated full and faithful reduction functor $F^l:{\cal A}^l\g\Mod\rightmap{}{\cal A}\g\Mod$, see \cite[(17.7) and (26.4)]{BSZ}. Here, we concentrate on the particular case of minimal ditalgebras with only one marked point, which is  the context of interest in the following. 

We adopt the notation of the beginning of \S4 for the constituents of our minimal ditalgebra ${\cal B}$ and consider a non-zero polynomial $g\in k[x]$. Then, the localization of ${\cal B}$ with respect to the element $g(x)$ is the minimal ditalgebra ${\cal B}^l$, with layer $(R^l,W_1^l)$ such that $R^l=k[x]_{g(x)h(x)}\frak{f}_0\times (\times_{z\in \frak{Z}}k\frak{f}_z)$ and $W_1^l=R^l\otimes_RW_1\otimes_RR^l$. The canonical epimorphism of algebras $R\rightmap{}R^l$ and the  morphism of $R$-$R$-bimodules $\phi:W_1\rightmap{}W^l_1$, which maps each $w$ onto $1\otimes w\otimes 1$, induce a morphism of ditalgebras $\hat{\phi}:{\cal B}\rightmap{}{\cal B}^l$, that is a morphism of graded algebras 
$\hat{\phi}:T_R(W_1)\rightmap{}T_{R^l}(W_1^l)$ such that $\delta^l\hat{\phi}=\hat{\phi}\delta$, where $\delta$ and $\delta^l$ denote the differentials of ${\cal B}$ and ${\cal B}^l$ respectively. 

We want to compare the morphisms $\hat{\delta}$ and $\hat{\delta}^l$, see (\ref{D: hat(delta) y admis}). Notice that, if we denote with the same symbols the differentials  $\delta$ and $\delta^l$, and their restrictions to $W_1$ and $W_1^l$, respectively, we have:
\begin{equation}\tag{1} \delta^l\phi=(\phi\otimes\phi)\delta
\end{equation}
and
\begin{equation}\tag{2} \begin{matrix}(1\otimes\delta^l)\delta^l\phi
&=&
(1\otimes\delta^l)(\phi\otimes\phi)\delta
=
(\phi\otimes\delta^l\phi)\delta\hfill\\
&=&
(\phi\otimes(\phi\otimes\phi)\delta)\delta=(\phi\otimes\phi\otimes\phi)(1\otimes\delta)\delta.\end{matrix}
\end{equation}
For the minimal ditalgebra ${\cal B}^l$, we have the $R^l\frak{f}_0\g R^l\frak{f}_0$-bimodule
$$U^l=
\left[\bigoplus_{t=1}^n
\frak{f}_0W^l_1\frak{f}_{z_t}
\otimes_{R^l}\frak{f}_{z_t}W^l_1\frak{f}_0\right]
\oplus\left[\bigoplus_{t=1}^n
\frak{f}_0W^l_1\frak{f}_{r_t}\otimes_{R^l}k(1\otimes u_t\otimes 1)\otimes_{R^l}\frak{f}_{l_t}W^l_1\frak{f}_0\right]
$$
and the corresponding morphism $\hat{\delta}^l:\frak{f}_0W^l_1\frak{f}_0\rightmap{}U^l$. From (1) and (2), we derive the equality  
$$\hat{\phi}\hat{\delta}=\hat{\delta}^l\phi.$$
 
For $v\in \hueca{B}_{\frak{f}_0,\frak{f}_0}$, we have $\hat{\delta}(v)=\sum_{\tau\in \hueca{T}}c_\tau^v\tau$. Therefore,  
$\hat{\delta}^l(\phi(v))=\sum_{\tau\in \hueca{T}}c^v_\tau\hat{\phi}(\tau)$. There is a natural choice of directed basis $\hueca{B}^l:=\phi(\hueca{B})$ for the $R^l\g R^l$-bimodule $W_1^l$, with this choice, we see that $\hueca{T}^l=\hat{\phi}(\hueca{T})$ is the chosen directed basis for $U^l$. Thus, the  coefficients defining $\hat{\delta}^l$ are the same that the coefficients defining $\hat{\delta}$.   

Any numberings $v_1,\ldots,v_{c_0}$ of $\hueca{B}_{\frak{f}_0,\frak{f}_0}$ and $\tau_1,\ldots,\tau_{c_1}$ of $\hueca{T}$, determines the numberings $\phi(v_1),\ldots,\phi(v_{c_0})$ of $\hueca{B}^l_{\frak{f}_0,\frak{f}_0}$ and $\hat{\phi}(\tau_1),\ldots,\hat{\phi}(\tau_{c_1})$ of $\hueca{T}^l$. With these numberings, we have the equality of coefficient matrices $C(x,y)=C^l(x,y)$.  
\end{remark}

\begin{lemma}\label{L: triangulacion de C}
If the coefficient matrix $C(x,y)$ of our minimal ditalgebra ${\cal B}$ has rank $c_0$, then there is a polynomial $g(x)\in k[x]$ such that the minimal ditalgebra ${\cal B}^l$ obtained by localization of ${\cal B}$ with respect to this polynomial $g(x)$ is admissible. 
\end{lemma}

\begin{proof} We have our fixed directed basis $\hueca{B}$ for $W_1$ and consider any numberings $v_1,\ldots,v_{c_0}$ of $\hueca{B}_{\frak{f}_0,\frak{f}_0}$ and $\tau_1,\ldots,\tau_{c_1}$ of $\hueca{T}$. So, we have $\delta(v_i)=\sum_{j=1}^{c_1}c_{i,j}(x,y)\tau_j$ and the coefficient matrix $C(x,y)=(c_{i,j}(x,y))\in M_{c_0\times c_1}(k[x,y])$ has rank $c_0$ by  assumption.  

We can reorder the elements of $\hueca{T}$ in  such a way that $C(x,y)=(C_0,C_1)$, with $C_0\in M_{c_0\times c_0}(k[x,y])$ and $C_1\in M_{c_0\times (c_1-c_0)}(k[x,y])$, where $C_0$ has non-zero determinant.   The matrix $C_0$ has entries in the polynomial algebra $k(y)[x]$. Using elementary operations on the rows of $C_0$, we can transform it into an upper triangular matrix. So, there is a  matrix $A\in M_{c_0\times c_0}(k(y)[x])$, which is invertible and such that $AC_0$ is an upper triangular matrix. Consider $g(y)\in k[y]$ such that $A\in M_{c_0\times c_0}(k[y]_{g(y)}[x])$ and the minimal ditalgebra ${\cal B}^l$ obtained from ${\cal B}$ by localization with respect to $g(x)$. 

Since the matrix $A=(a_{i,j}(x,y))$ is invertible in 
$$M_{c_0\times c_0}(k[y]_{g(y)}[x])\subseteq M_{c_0\times c_0}(k[x]_{h(x)g(x)}\otimes_kk[x]_{h(x)g(x)})=
M_{c_0\times c_0}(R^l\frak{f}_0\otimes_kR^l\frak{f}_0),$$
 the elements $v^l_1,\ldots,v^l_{c_0}$ defined by 
$v^l_i:=\sum_{j=1}^{c_0}{a_{i,j}}\phi(v_j)$
form a basis for the 
$R^l\frak{f}_0\g  R^l\frak{f}_0$-bimodule 
$\frak{f}_0W_1^l\frak{f}_0$, see 
\cite[(26.1)]{BSZ}.     From (\ref{D: hat(delta) y admis}) and (\ref{R: localization}), we know that   
 $$\hat{\delta}^l(\phi(v_i))=\sum_{t=1}^{c_1}c_{i,t}\hat{\phi}(\tau_t).$$
 Then, we have 
 $$\begin{matrix}
 \hat{\delta}^l(v^l_i)
 &=&
 \sum_{j=1}^{c_0}{a_{i,j}}\hat{\delta}^l\phi(v_j)
 =\sum_{j,t}a_{i,j}c_{j,t}\hat{\phi}(\tau_t)\hfill\\
 &=&
 \sum_{t=1}^{c_0}c^l_{i,t}\hat{\phi}(\tau_t)
 +
 \sum_{t>c_0}c^l_{i,t}\hat{\phi}(\tau_t),\hfill\\ 
 \end{matrix}$$
 where $c^l_{i,t}:=\sum_{j=1}^{c_0}a_{i,j}c_{j,t}$, for all $i\in [1,c_0]$ and $t\in [1,c_1]$. For $t\in [1,c_0]$, the coefficient $c^l_{i,t}$ is the $(i,t)$-entry of the upper triangular matrix $AC_0$, so $c^l_{i,t}=0$, for $i>t$. Since $AC_0$ has rank $c_0$, we have $c^{l}_{i,i}\not=0$, for each $i\in [1,c_0]$. Finally, as we did before in (\ref{R: descriptions of delta and}), we can  replace each one of the new basis  elements 
 $v^l_i$ by the product $h(x)^mv_i^lg(x)^m$ for a suitable $m\in \hueca{N}$ to obtain that the corresponding coefficient matrix is normal (with entries in $k[x,y]$).  
\end{proof}

\begin{remark} We already have a minimal ditalgebra ${\cal B}$, with only one marked point, and a full and faithful functor $F:{\cal B}\g\Mod\rightmap{}{\cal P}^1(\Lambda)$ with the properties listed in (\ref{P: el funtor F}). We also made in (\ref{R: choice of special basis}) a special choice of basis $\hueca{B}$ such that for special elements $u_t\in \hueca{B}_{\frak{f}_{r_t},\frak{f}_{l_t}}$, we have $F(f_{u_t})=\gamma_t$, for all $t\in [1,n]$.
 Now, if the coefficient matrix $C(x,y)$ has rank $c_0$, we can consider the minimal ditalgebra ${\cal B}^l$ constructed in the last lemma and the corresponding full and faithful functor $F^l:{\cal B}^l\g\Mod\rightmap{}{\cal B}\g\Mod$. It is not hard to see that the functor  $FF^l:{\cal B}\g\Mod\rightmap{}{\cal P}^1(\Lambda)$ has the same properties enumerated for $F$. Then, in this case, we may assume that the minimal ditalgebra ${\cal B}$ is already admissible, so we have  (\ref{P: C normal --> rad(k(x)f0) se factorizan bien}) at our disposal.   
\end{remark}

\begin{proposition}\label{P: el resto del main thm} If $\End_\Lambda(M(\lambda))=k$, for infinitely many $\lambda\in D(h)$, then 
$$\End_\Lambda(G)\cong k(x).$$ 
\end{proposition}

\begin{proof} Assume that 
$\End_\Lambda(M(\lambda))=k$, for infinitely many $\lambda\in D(h)$. From (\ref{C: caract de M(lambda) brics iff C(x) with rank c0}), we know that the  matrix $C(x)=(c^{v_i}_{\tau_j}(x,x))$ has rank $c_0$. This implies that the coefficient matrix $C(x,y)=(c^{v_i}_{\tau_j}(x,y))$ has rank $c_0$. Then, as observed in the previous remark, we may assume that ${\cal B}$ is admissible and we can apply (\ref{P: C normal --> rad(k(x)f0) se factorizan bien}) to ${\cal B}$, so every morphism  
$f\in \rad \End_{\cal B}(k(x)\frak{f}_0)$ is a $k$-linear combination of morphisms which factor through some object $S^x_{\frak{f}_{z_t}}$ or through some morphism $f^x_{u_t}$, for some $t\in [1,n]$. 

We may assume that $\hat{G}=F(k(x)\frak{f}_0)$. Given any morphism $g\in \rad\End_\Lambda(G)$, there is some $\hat{g}\in \rad\End_{{\cal P}^1(\Lambda)}(\hat{G})$, with $\Coker(\hat{g})=g$. Since $F:{\cal B}\g\Mod\rightmap{}{\cal P}^1(\Lambda)$ is full and faithful, we have $F(f)=\hat{g}$, for some $f\in\rad\End_{\cal B}(k(x)\frak{f}_0)$. So, there are finitely many commutative diagrams in ${\cal B}\g\Mod$: 
$$\begin{matrix}
k(x)\frak{f}_0&\rightmap{f_i}&k(x)\frak{f}_0\\
\shortlmapdown{h_i}&&\shortrmapup{g_i}\\
S^x_{\frak{f}_{l_{t_i}} }&\rightmap{f^x_{u_{t_i}}}&
S^x_{\frak{f}_{r_{t_i}} }
\end{matrix} \hbox{ \hskip 1cm}
\begin{matrix}
k(x)\frak{f}_0\hbox{\hskip-.5cm}&\hbox{\hskip.5cm}\rightmap{\ f'_j \ }&k(x)\frak{f}_0\\
&\ddmapdown{h'_j}\hbox{\hskip.4cm}&\ddmapup{g'_j}\hbox{\hskip.8cm}\\
&\hbox{\hskip .8cm}S^x_{\frak{f}_{z_{t_j}}}&\\
\end{matrix}$$
such that $f=\sum_if_i+\sum_jf'_j$. 
Therefore, after applying $F$ and using the diagrams of (\ref{L: F conmuta con otimes k(x)}), we obtain finitely many commutative diagrams in ${\cal P}^1(\Lambda)$: 
$$\begin{matrix}
\hat{G}&\rightmap{\hat{f}_i}&\hat{G}\\
\shortlmapdown{\hat{h}_i}&&\shortrmapup{\hat{g}_i}\\
L(\Lambda e_{t_i})\otimes_kV
&\rightmap{\gamma_{t_i}\otimes id}&
R(\Lambda e_{t_i})\otimes_kV 
\end{matrix} 
\hbox{ \hskip 1cm}
\begin{matrix}
\hat{G}\hbox{\hskip-1.5cm}&\hbox{\hskip1cm}\rightmap{ \ \hat{f}'_j \ }&\hat{G}\\
&\ddmapdown{\hat{h}'_j}\hbox{\hskip.4cm}&\ddmapup{\hat{g}'_j}\hbox{\hskip.4cm}\\
&\hbox{\hskip .6cm}S(\Lambda e_{t_j})\otimes_kV&\\
\end{matrix}$$
such that $\hat{g}=F(f)=\sum_i\hat{f}_i+\sum_j\hat{f}'_j$, 
where $V$ denotes the $k$-vector space $k(x)$. 
Here, each morphism $\gamma_t\otimes id=(\beta_t\otimes id)(\alpha_t\otimes id)$ factors through $U(\Lambda e_t)\otimes_kV$, so 
 we get that $g=\Coker(\hat{g})=0$. 
\end{proof}

\noindent{\bf Proof of the main Theorem (\ref{T: gen-bricks vs bricks}):} 
It follows immediately from (\ref{P: radEnd(G)=0 -> radEnd(M(lambda))=0, casi todo lambda})
and (\ref{P: el resto del main thm}) that the realization $M$ of $G$ that we have been studying satisfies the equivalence stated in (\ref{T: gen-bricks vs bricks}). As observed in the introduction, this proves our main theorem. 
\hfill$\square$

\medskip
\noindent{\bf Proof of the Corollary 
(\ref{C: tame conj M-P}):} If $G$ is a generic brick, then a realization $M$ of $G$ as in Theorem  (\ref{T: CB-realizations}) is free by the right, with rank equal to the endolength of $G$. Then, by (\ref{T: gen-bricks vs bricks}), $M$ determines an infinite family $\{M(\lambda)\}_\lambda$ of finite-dimensional bricks  $M(\lambda)=M\otimes_\Gamma \Gamma/(x-\lambda)$ with $k$-dimension $d=\Endol{G}$. 

Conversely, assume that $\{M_i\}_{i\in I}$ is an infinite family of non-isomorphic bricks with dimension $d$. Applying \cite[(5.5)]{CB2} to $d$, we obtain a generic $\Lambda$-module $G$ and a realization $M$ of $G$ over some rational algebra $\Gamma$, such that for  almost every $i\in I$,  $M_i\cong M\otimes_\Gamma\Gamma/(x-\lambda_i)^{m_i}$, for some $\lambda_i\in k$ and  $m_i\in\hueca{N}$. We may assume that this realization is as in (\ref{T: CB-realizations}), so $M\otimes_\Gamma-$ preserves almost split sequences. This implies that for each $\lambda$, we have   
 almost split sequences in $\Lambda\g\mod$ of the form: 
$$\begin{matrix} \  E_1  \rightmap{\alpha_1} E_2\rightmap{\beta_1} E_1,\\
 \,\\
 \ E_m \rightmap{(\alpha_m,\beta_{m-1})} E_{m+1}\oplus E_{m-1}\rightmap{(\beta_m,\alpha_{m-1})}E_m, \hbox{ for } m \geq  2,\\
\end{matrix}$$
with $E_m\cong M\otimes_\Gamma\Gamma/(x-\lambda)^m$. Then, for $m\geq 2$, the composition 
$$E_m\rightmap{ \ \beta_{m-1} \ }
E_{m-1}
\rightmap{ \ \alpha_{m-1} \ }E_m,$$
of irreducible morphisms with $\beta_{m-1}$ surjective and $\alpha_{m-1}$ injective is a non-zero morphism in $\rad\End_\Lambda(E_m)$.   This implies that $M_i\cong M(\lambda_i)$, for infinitely many $i\in I$. Hence, $G$ is a generic brick. 
\hfill$\square$

\begin{remark} Notice that in the statement of Theorem (\ref{T: gen-bricks vs bricks}), we can replace the second item by
\begin{enumerate}
\item[\it 2'.] \emph{For almost all $\lambda\in k$, we have that $M(\lambda)$ is a brick.}
\end{enumerate}
Indeed, we just have to replace the phrase ``for infinitely many $\lambda\in k$" by ``for almost all $\lambda\in k$" throughout the whole given proof of (\ref{T: gen-bricks vs bricks}). 
\end{remark}

\medskip
\noindent{\bf Acknowledgements.} The authors are indebted to M. Mousavand for drawing their  attention to the problem which motivated this work, during the workshop “Silting in Representation Theory, Singularities, and Noncommutative Geometry”,  held at Casa Matem\'atica  Oaxaca in  2023.
The authors acknowledge partial financial support from the \emph{Programa de Intercambio Acad\'emico UNAM-UADY}.

\hskip2cm

\vbox{\noindent R. Bautista\\
Centro de Ciencias Matem\'aticas\\
Universidad Nacional Aut\'onoma de M\'exico\\
Morelia, M\'exico\\
raymundo@matmor.unam.mx\\}

\vbox{\noindent E. P\'erez\\
Facultad de Matem\'aticas\\
Universidad Aut\'onoma de Yucat\'an\\
M\'erida, M\'exico\\
jperezt@correo.uady.mx\\}

\vbox{\noindent L. Salmer\'on\\
Centro de Ciencias  Matem\'aticas\\
Universidad Nacional Aut\'onoma de M\'exico\\
Morelia, M\'exico\\
salmeron@matmor.unam.mx\\}

 \end{document}